\documentclass[a4paper,11pt]{article}
\pdfoutput=1
\usepackage[english]{babel}
\usepackage{amsfonts}
\usepackage{amsthm}
\usepackage{amsmath}
\usepackage{amssymb}
\usepackage{enumerate}
\usepackage{mathtools}
\usepackage{tikz-cd}
\usepackage[all]{xy}
\usepackage{graphicx}
\usepackage{graphicx}
\usepackage{adjustbox}
\usepackage{extpfeil}
\usepackage{comment}
\usepackage{float}
\usepackage[labelformat=empty]{caption}
\usepackage[toc]{appendix}
\usepackage[left=3cm,top=2.5cm,right=3cm,bottom=2.5cm]{geometry}
\usepackage{hyperref}
\usepackage{resizegather}
\usepackage[activate={true, nocompatibility}, final, tracking=true, kerning=true, spacing=true, factor=1100, stretch=10, shrink=10]{microtype}
\usepackage{authblk}

\theoremstyle{plain}
\newtheorem{thm}{Theorem}[section]
\newtheorem{coroll}[thm]{Corollary}
\newtheorem{defn}[thm]{Definition}
\newtheorem{lemma}[thm]{Lemma}

\newtheorem{example}[thm]{Example}
\newtheorem{prop}[thm]{Proposition}

\newtheorem{remark}[thm]{Remark}

\newtheorem*{theorem}{Main Theorem}
\newtheorem*{lemma*}{Lemma}
\newtheorem*{prop*}{Proposition}

\newcommand{\bseries}[1]{ [\hspace{-0,5mm}[ {#1} ]\hspace{-0,5mm}] }

\newcommand\dhxrightarrow[2][]{%
  \mathrel{\ooalign{$\xrightarrow[#1\mkern4mu]{#2\mkern4mu}$\cr%
  \hidewidth$\rightarrow\mkern4mu$}}
}

\tikzset{
  symbol/.style={
    draw=none,
    every to/.append style={
      edge node={node [sloped, allow upside down, auto=false]{$#1$}}}
  }
}

\microtypecontext{spacing=nonfrench}

\begin{document}
\thispagestyle{empty}

\title{\textbf{Harder-Narasimhan stratification for the moduli stack of parabolic vector bundles}}
\author{Andres Fernandez Herrero}

\date{}
\maketitle

\begin{abstract}
For every set of parabolic weights, we construct a Harder-Narasimhan stratification for the moduli stack of parabolic vector bundles on a curve. It is based on the notion of parabolic slope, introduced by Mehta and Seshadri. We also prove that the stratification is schematic, that each stratum is complete, and establish an analogue of Behrend's conjecture for parabolic vector bundles. A comparison with recent $\Theta$-stratification approaches is discussed.
\end{abstract}


\tableofcontents

\begin{section}{Introduction}
Let $C$ be a curve over a field $k$. Fix a finite set of $k$-points $x_i$ in $C$. In this paper we investigate the geometry of moduli stacks $\text{Bun}_{\mathcal{V}}$ of vector bundles over $C$ with the additional data of a flag of vector spaces in the fiber at each $x_i$. The objects parametrized by $\text{Bun}_{\mathcal{V}}$ are parabolic vector bundles, as in \cite{mehta-seshadri}.

The purpose of this article is to describe a stratification of $\text{Bun}_{\mathcal{V}}$ analogous to the classical Harder-Narasimhan stratification for the moduli of vector bundles \cite{harder-narasimhan}. It is the parabolic analogue for $G=\text{GL}_n$ of the results of Gurjar and Nitsure \cite{gurjar-nitsure} when $\text{char}(k) = 0$, and recently generalized in \cite{nitsuregurjar2, gurjar2020hardernarasimhan} to higher dimensional varieties and general $k$. Our version of the Harder-Narasimhan stratification is based on the notion of slope for a parabolic vector bundle. This notion of slope depends on the choice of some stability parameters $\overline{\lambda}$ known as parabolic weights. We use the slope to define a parabolic version of the Harder-Narasimhan filtration. This filtration provides us with an invariant of parabolic vector bundles, which we call the HN datum. For each HN datum $P$, we define a stratum $\text{Bun}_{\mathcal{V}}^{\leq P}$ consisting of parabolic vector bundles with HN datum bounded by $P$. We prove the following properties of the strata in Theorems \ref{thm: openess of strata}, \ref{thm: quasicompactness of strata}, \ref{thm: completeness of strata} and \ref{thm: representability of moduli of vb with fixed P}.
\begin{theorem} \label{thm: main theorem introduction}
For every HN datum $P$, the following properties hold. \quad
\begin{enumerate}[(A)]
    \item The stratum $\text{Bun}_{\mathcal{V}}^{\leq P}$ is a quasicompact open substack of $\text{Bun}_{\mathcal{V}}$.
    \item The stratum $\text{Bun}_{\mathcal{V}}^{\leq P}$ is complete.
    \item The locus $\text{Bun}_{\mathcal{V}}^{=P}$ of parabolic vector bundles with HN datum $P$ can be naturally equipped with the structure of a locally closed substack of $\text{Bun}_{\mathcal{V}}$. In fact, $\text{Bun}_{\mathcal{V}}^{= P} \hookrightarrow \text{Bun}_{\mathcal{V}}^{\leq P}$ is a closed immersion.
\end{enumerate}
\end{theorem}

The article by Mehta and Seshadri \cite{mehta-seshadri} contains a definition of parabolic degree similar to the one we give. They use the notion of parabolic degree to build a moduli space of semistable parabolic vector bundles on a curve of genus $g \geq 2$ over an algebraically closed field.  They also prove a version of our Theorem \ref{thm: openess of strata} in the special case of a semistable stratum by embedding the semistable locus into a GIT problem. In contrast we deal with the whole moduli stack, including the unstable locus. Our proofs work intrinsically with the stack, without appealing to GIT.

Parabolic vector bundles can be alternatively seen as torsors for a parahoric Bruhat-Tits group scheme, as in \cite{pappas-rapoport-moduli}. See Proposition \ref{prop: parabolic vb as parahoric torsors} below for a proof of this. It follows from \cite{heinloth-uniformization}[Prop. 1] that $\text{Bun}_{\mathcal{V}}$ is a smooth algebraic stack. There has been recent interest in developing stratifications for moduli stacks of torsors for such group schemes. Gaitsgory and Lurie \cite{gaitsgory-lurie} give an ad hoc construction of stratifications in the semisimple case in order to derive a Grothendieck-Lefschetz trace formula for the moduli stack of torsors. Balaji and Seshadri \cite{balaji-seshadri} described Harder-Narasimhan stratifications for parahoric Bruhat-Tits group schemes under the assumption that the generic fiber is split. They focus on the construction of a moduli space for semistable torsors. They obtain results valid when $C$ is a curve of genus $g \geq 2$ over a ground field of characteristic $0$. 

Heinloth \cite{heinloth-hilbertmumford} and Alper, Halpern-Leistner, Heinloth \cite{alper-existence}[\S 8] develop a theory of $\Theta$-stratifications for the moduli of torsors of a parahoric Bruhat-Tits group scheme in arbitrary characteristic, under some mild tameness assumptions on the generic fiber. They use a modern variation of ideas from GIT for stacks as in \cite{halpernleistner2014structure} in order to obtain a stratification associated to a line bundle on the stack. In Appendix \ref{appendix: comparison} we explain how to build certain $\mathbb{R}$-line bundles on $\text{Bun}_{\mathcal{V}}$ associated to a parabolic weight $\overline{\lambda}$. We explicitly describe the subset of parabolic weights $\overline{\lambda}$ such that the corresponding line bundle is admissible (cf. \cite{heinloth-hilbertmumford}). Whenever this is the case, the line bundle induces a $\Theta$-stratification that we show coincides with our parabolic slope stratification. It should be remarked that this admissibility condition on parabolic weights is restrictive and we do not impose it in this text.

Our techniques to establish the Main Theorem make contact with the classical theory for vector bundles already present in \cite{mehta-seshadri}. In order to deal with subbundles of parabolic vector bundles in families, we define an analogue of the Quot scheme for parabolic vector bundles (Definition \ref{defn: second quot moduli}). We also construct an iterated Quot scheme $Fil^{\alpha}_{\mathcal{W}}$ parametrizing filtrations of a parabolic vector bundle with some given fixed set of invariants (Definition \ref{defn: parabolic quot filtration scheme}).

We also prove certain properties of the stratification that do not hold in the generality of \cite{heinloth-hilbertmumford} when the ground field has positive characteristic. The obstruction in the more general case is the failure of Behrend's conjecture for certain reductive groups in bad characteristics (established by Heinloth \cite{heinloth-behrends-conjecture}). Recall that Behrend's conjecture \cite{behrend-canonical-bundles} states that the canonical parabolic reduction of a principal $G$-bundle over a curve with $G$-reductive does not admit nontrivial deformations. We prove that an analogue of Behrend's conjecture holds in our context of parabolic vector bundles:
\begin{prop*}[= Proposition \ref{prop: rigidity of HN filtration}]
Let $k[\epsilon] \vcentcolon = k[T] \, / \, (T^2)$ be the ring of dual numbers. We denote by $\sigma: \text{Spec}(k) \rightarrow \text{Spec} \;k[\epsilon]$ the unique section of the structure map $p: \text{Spec} \;k[\epsilon] \rightarrow \text{Spec}(k)$. Let $\mathcal{V}$ be a parabolic vector bundle over $C$. Suppose that we are given a filtration $0 = \mathcal{W}_0 \subset \mathcal{W}_1 \subset \cdots \, \subset \mathcal{W}_{l-1} \subset \mathcal{W}_l = p^{*} \mathcal{V}$ of $p^{*} \mathcal{V}$ by parabolic subbundles.
Assume that the pulled-back filtration 
$0 = \sigma^{*}\mathcal{W}_0 \subset \sigma^{*}\mathcal{W}_1 \subset \cdots \, \subset \sigma^{*}\mathcal{W}_{l-1} \subset \sigma^{*}\mathcal{W}_l = \mathcal{V}$ over the closed point
is the Harder-Narasimhan filtration of $\mathcal{V}$. Then $\mathcal{W}_j = p^{*} \sigma^{*} \mathcal{W}_j$ for all $1 \leq j \leq l-1$.
\end{prop*}
We use this result along with the construction of the filtration scheme $Fil^{\alpha}_{\mathcal{W}}$ in order to develop a theory of schematic Harder-Narasimhan filtrations. This is assertion (C) of the Main Theorem (see also Theorem \ref{thm: representability of moduli of vb with fixed P}). It is the parabolic analogue of the work done by Gurjar and Nitsure in \cite{gurjar-nitsure}, which deals with principal $G$-bundles over a curve when the characteristic of the ground field is $0$.

Finally, we also prove in Theorem \ref{thm: completeness of strata} that for every $P$ the stratum $\text{Bun}^{\leq P}_{\mathcal{V}}$ satisfies a strong lifting criterion for families over a discrete valuation ring. This is assertion (B) of the Main Theorem. Heinloth \cite[Remark 3.20]{heinloth-hilbertmumford} noted that the weaker valuative criterion \cite[\href{https://stacks.math.columbia.edu/tag/0CLK}{Tag 0CLK}]{stacks-project} holds for the semistable locus for more general parahoric Bruhat-Tits group schemes. See Section \ref{section: completeness of strata} for more details on this.

We conclude this introduction with an outline of the paper.

In Section \ref{section: parabolic vector bundles curve} we recall the notions of parabolic vector bundles (Definition \ref{defn: parabolic vector bundle on curve}), parabolic weights (Definition \ref{defn: parabolic weights}), degree and slope (Definition \ref{defn: deg of vector bundle}) and the existence and uniqueness of the Harder-Narasimhan filtration (Proposition \ref{prop: existence and uniqueness of HN filtrations}). The section concludes with a discussion of parabolic vector bundles in families (Subsection \ref{subsection: parabolic vb families}). This includes a proof of the compatibility of the Harder-Narasimhan filtration with extensions of the ground field (Lemma \ref{lemma: HN filtrations and scalar change}) and a proof of the analogue of Behrend's conjecture for parabolic vector bundles (Proposition \ref{prop: rigidity of HN filtration}). In Section \ref{section: stacks of parabolic vector bundles} we define the stack $\text{Bun}_{\mathcal{V}}$ of parabolic vector bundles of a given type (Definition \ref{defn: stack of parabolic vector bundles}) and show that it is isomorphic to the stack of torsors for a parahoric Bruhat-Tits group scheme over $C$ (Proposition \ref{prop: parabolic vb as parahoric torsors}). We define the HN datum of a vector bundle (Definition \ref{defn: hn datum}) and the strata $\text{Bun}_{\mathcal{V}}^{\leq P}$ (Definition \ref{defn: HN strata}).

In Section \ref{section: parabolic Quot schemes} we construct a parabolic version $\text{Quot}_{\mathcal{V}}^{P, \, b_i^{(m)}}$ of the Quot scheme (Definition \ref{defn: second quot moduli}). More generally, we consider nested versions $\text{Fil}^{\alpha}_{\mathcal{V}}$ of parabolic Quot schemes (Definition \ref{defn: parabolic quot filtration scheme}). The main purpose of this section is proving that $\text{Fil}^{\alpha}_{\mathcal{V}}$ is represented by a separated scheme of finite type over the base (Proposition \ref{prop: representability of filtration moduli}). This result is used in Sections \ref{section: constructibility of strata} and \ref{section: harder-narasimhan filtrations in families}. We recommend the reader to skip the proofs in Section \ref{section: parabolic Quot schemes} upon their first read.

The purpose of Section \ref{big section: openness of strata} is to prove that each stratum $\text{Bun}_{\mathcal{V}}^{\leq P}$ is represented by an open substack of $\text{Bun}_{\mathcal{V}}$ (Theorem \ref{thm: openess of strata}). This is achieved by first showing that $\text{Bun}_{\mathcal{V}}^{\leq P}$ is constructible (Proposition \ref{prop: constructibility of strata}) and then proving that $\text{Bun}_{\mathcal{V}}^{\leq P}$ is closed under generalization (Proposition \ref{prop: strata closed under generalization}). In Section \ref{section: quasicompactness of strata} we prove that each stratum $\text{Bun}_{\mathcal{V}}^{\leq P}$ is quasicompact (Theorem \ref{thm: quasicompactness of strata}). For this purpose, we use the Harder-Narasimhan stratification of the moduli stack of classical vector bundles, which we recall in Subsection \ref{subsection: classical hn stratication}. Our strategy is to show that the forgetful morphism $\text{Bun}_{\mathcal{V}}^{\leq P} \to \text{Bun}_{\text{GL}_n}(C)$ is quasicompact (Proposition \ref{prop: properness of forgetful map}), and that the image lies in finitely many classical Harder-Narasimhan strata (Corollary \ref{coroll: factoring of forgetful map}). In Section \ref{section: completeness of strata} we prove that each stratum is complete (see Definition \ref{defn: completeness of stacks} and Theorem \ref{thm: completeness of strata}). We adapt the parabolic analogue of Langton's algorithm in Mehta-Seshadri \cite{mehta-seshadri} so that it applies to a general unstable stratum $\text{Bun}_{\mathcal{V}}^{\leq P}$.

In Section \ref{section: harder-narasimhan filtrations in families} we define a notion of relative Harder-Narasimhan filtration in families (Definition \ref{defn: relative HN filtration}). We use this to equip each locally closed stratum $\text{Bun}_{\mathcal{V}}^{= P}$ with the structure of a locally closed substack of $\text{Bun}_{\mathcal{V}}$ (Theorem \ref{thm: representability of moduli of vb with fixed P}). The strategy of proof is to show that the natural morphism $\text{Bun}_{\mathcal{V}}^{=P} \to \text{Bun}_{\mathcal{V}}^{\leq P}$ is proper (using the representability of $\text{Fil}^{\alpha}_{\mathcal{V}}$), radicial (by the uniqueness of Harder-Narasimhan filtrations) and unramified (by the parabolic analogue of Behrend's conjecture in Proposition \ref{prop: rigidity of HN filtration}).

In Appendix \ref{appendix: descent} we provide an exposition of descent for finite morphisms, which we adapt from \cite{grothendieck-descente}[B.3]. In Appendix \ref{appendix: comparison} we discuss the relation of our stratification with the $\Theta$-stratification for the moduli stack of parahoric torsors constructed in \cite{heinloth-hilbertmumford, alper-existence}.

\textbf{Acknowledgements} The notion of slope employed here is a reinterpretation of the definition suggested by J. Lurie in the outline of his Arizona Winter School project \cite{lurie-arizona}, and which we futher relate to \cite{mehta-seshadri}. Some of the arguments greatly benefited from discussion with other members of the project during the evening sessions. I would particularly like to thank Aaron Landesman, David Yang and Bogdam Zavyalov for helpful discussions. I am happy to thank my advisor Nicolas Templier for encouraging me to write down these results and providing very valuable input for the redaction of the manuscript. I thank Nitin Nitsure for helpful comments and David Mehrle for helping out with \LaTeX. I acknowledge support by NSF grants DMS-1454893 and DMS-2001071.

\begin{subsection}{Notation}
We work over a fixed ground field $k$. All schemes are understood to be schemes over $k$. An undecorated product of $k$-schemes (e.g. $A\times B$) should always be interpreted as a fiber product over $k$. If $R$ is a $k$ algebra and $S$ is a $k$-scheme, we may sometimes use the notation $S_R$ to denote the fiber product $S \times \text{Spec}(R)$. 

We write $s \in S$ to mean that $s$ is a (set theoretic) point of $S$. Equivalently, $s$ is the $\text{Spec}$ of the residue field of a topological point of $S$. For such $s \in S$, we will write $\kappa(s)$ for the residue field.

We will often deal with pullbacks of quasicoherent sheaves on schemes. Let $X$ and $Y$ be schemes and let $\mathcal{Q}$ be quasicoherent sheaf on $Y$. It will usually be the case that there is a clear choice of a morphism from $f: X \rightarrow Y$. In such a situation we will not mention the choice of $f$ and write $\mathcal{Q}|_{X}$ to denote the pullback $f^{*}\mathcal{Q}$ of the quasicoherent sheaf $\mathcal{Q}$ by $f$.

We fix once and for all a curve $\pi: C \rightarrow \text{Spec}(k)$ that is smooth, projective and geometrically irreducible over $k$. We choose a finite set $\{x_i\}_{i \in I}$ of $k$-points $x_i$ of $C$. We will often refer to these $x_i$ as points of degeneration. For all $i \in I$, let us denote by $q_i: x_i \rightarrow C$ the closed immersion of $x_i$ into $C$. We also fix the choice of an ample line bundle $\mathcal{O}(1)$ of degree one on $C$.
\end{subsection}
\end{section}

\begin{section}{Parabolic Vector Bundles} \label{section: parabolic vector bundles curve}
\begin{subsection}{Parabolic vector bundles on a curve}

\begin{defn} \label{defn: parabolic vector bundle on curve}
A parabolic vector bundle $\mathcal{V}$ over $C$ with parabolic structure at $\{x_i\}_{i \in I}$ consists of the data of:
\begin{enumerate}[(a)]
    \item A vector bundle $\mathcal{E}^{(0)}$ on $C$.
    \item For each $i \in I$, a chain of vector bundles $\mathcal{E}^{(0)} \,\subset \, \mathcal{E}_i^{(1)} \subset \cdots\; \subset \, \mathcal{E}_i^{(N_i)} = \mathcal{E}^{(0)}(x_i)$.
\end{enumerate}
We define the rank of $\mathcal{V}$ to be $\text{rank} \, \mathcal{V} \vcentcolon = \, \text{rank} \, \mathcal{E}^{(0)}$. For each $i \in I$, we call $N_i$ the chain length at $i$.
\end{defn}

For a parabolic vector bundle $\mathcal{V}$ as above, we will use the convention $\mathcal{E}_i^{(0)} \vcentcolon = \mathcal{E}^{(0)}$ for all $i \in I$. Note that for each $i \in I$ and $1 \leq m \leq N_i$, the quotient $\mathcal{E}_i^{(m)} \, / \, \mathcal{E}_i^{(m-1)}$ is the pushforward $(q_i)_*(V)$ of a finite dimensional $k$-vector space $V$. We will want to keep track of the dimension of these vector spaces. For a parabolic vector bundle $\mathcal{V}$, we write
\[\mathcal{V} = \left[ \; \mathcal{E}^{(0)} \,\overset{a_i^{(1)}}{\subset} \, \mathcal{E}_i^{(1)} \overset{a_i^{(2)}}{\subset} \cdots\; \overset{a_{i}^{(N_i)}}{\subset} \, \mathcal{E}_i^{(N_i)}= \mathcal{E}^{(0)}(x_i) \;\right]_{i \in I} \]
to indicate that the $\mathcal{E}_i^{(m)}$s are the chains of vector bundles appearing in Definition \ref{defn: parabolic vector bundle on curve} and $\text{dim}_{k} \left(\mathcal{E}_i^{(m)} \, / \, \mathcal{E}_i^{(m-1)}\right) \, = \, a_i^{(m)}$.

\begin{defn} \label{defn: type of parabolic bundles} Let $\mathcal{V} = \left[ \; \mathcal{E}^{(0)} \,\overset{a_i^{(1)}}{\subset} \, \mathcal{E}_i^{(1)} \overset{a_i^{(2)}}{\subset} \cdots\; \overset{a_{i}^{(N_i)}}{\subset} \, \mathcal{E}_i^{(N_i)}= \mathcal{E}^{(0)}(x_i) \;\right]_{i \in I}$ and \\$\mathcal{W} = \left[ \;\mathcal{F}^{(0)} \,\overset{b_i^{(1)}}{\subset} \, \mathcal{F}_i^{(1)} \overset{b_i^{(2)}}{\subset} \cdots\; \overset{b_{i}^{(M_i)}}{\subset} \,\mathcal{F}_i^{(M_i)}= \, \mathcal{F}^{(0)}(x_i) \;\right]_{i \in I}$ be two parabolic vector bundles. We say that $\mathcal{W}$ has type $\mathcal{V}$ (or equivalently $\mathcal{V}$ has type $\mathcal{W}$) if 
 \begin{enumerate}[(a)]
     \item $\text{rank} \, \mathcal{W} = \text{rank} \, \mathcal{V}$.
     \item For all $i \in I$, $M_i = N_i$.
     \item For all $i \in I$ and $1 \leq m \leq N_i$, we have $a_i^{(m)} = b_i^{(m)}$.
 \end{enumerate}
\end{defn}
The class of parabolic vector bundles with a given fixed set of chain lengths forms a category. For the next two definitions, we let $\mathcal{V} = \left[ \; \mathcal{E}^{(0)} \,\overset{a_i^{(1)}}{\subset} \, \mathcal{E}_i^{(1)} \overset{a_i^{(2)}}{\subset} \cdots\; \overset{a_{i}^{(N_i)}}{\subset} \, \mathcal{E}_i^{(N_i)}= \mathcal{E}^{(0)}(x_i) \;\right]_{i \in I}$ and $\mathcal{W} = \left[ \;\mathcal{F}^{(0)} \,\overset{b_i^{(1)}}{\subset} \, \mathcal{F}_i^{(1)} \overset{b_i^{(2)}}{\subset} \cdots\; \overset{b_{i}^{(N_i)}}{\subset} \,\mathcal{F}_i^{(N_i)}= \, \mathcal{F}^{(0)}(x_i) \;\right]_{i \in I}$ be two parabolic vector bundles with the same chain length $N_i$ at each $i \in I$.
\begin{defn}
A morphism of parabolic vector bundles $f : \mathcal{W} \rightarrow \mathcal{V}$ consists of the data of a morphism of sheaves $f_i^{(m)}: \mathcal{F}_i^{(m)} \rightarrow \mathcal{E}_i^{(m)}$ for each $i \in I$ and $0 \leq m \leq N_i$  satisfying the following two conditions
\begin{enumerate}[(a)]
\item $f_i^{(0)} = f_j^{(0)}$ for all $i, j \in I$.
\item For all $i \in I$ and $1 \leq m \leq N_i$, the following diagram commutes
\begin{figure}[H]
    \centering
\begin{tikzcd}
 \mathcal{E}_i^{(m-1)} \ar[r,  symbol= \subset] & \mathcal{E}_i^{(m)} \\
 \mathcal{F}_i^{(m-1)} \ar[r, symbol= \subset] \ar[u, "f_i^{(m-1)}"] & \mathcal{F}_i^{(m)} \ar[u, "f_i^{(m)}"]
\end{tikzcd}
\end{figure}
\end{enumerate}
\end{defn}

Recall that for a given vector bundle $\mathcal{E}$ on $C$, we say that a subsheaf $\mathcal{F} \subset \mathcal{E}$ is a subbudle if $\mathcal{F}$ and $\mathcal{E} \, / \, \mathcal{F}$ are themselves vector bundles. In this case $\mathcal{E} \, / \, \mathcal{F}$ is called a quotient bundle.
\begin{defn}
Let $f: \mathcal{W} \rightarrow \mathcal{V}$ be a morphism.
\begin{enumerate}[(1)]
\item We say that $\mathcal{V}$ is a parabolic quotient bundle of $\mathcal{W}$ (equivalently $f$ is an admissible epimorphism) if for each $i \in I$ and $0 \leq m \leq N_i$, the morphism $f_i^{(m)}$ witnesses $\mathcal{E}_i^{(m)}$ as a quotient vector bundle of $\mathcal{F}_i^{(m)}$.
\item We say that $\mathcal{W}$ is a subbundle of $\mathcal{V}$ (equivalently $f$ is an admissible monomorphism) if it satisfies the following two conditions
    \begin{enumerate}[(a)]
        \item For each $i \in I$ and $0 \leq m \leq N_i$, the morphism $f_i^{(m)}$ witnesses $\mathcal{F}_i^{(m)}$ as a vector subbundle of $\mathcal{E}_i^{(m)}$.
        \item For each $i \in I$ and $0 \leq m \leq N_i$, we have $\mathcal{F}_i^{(m)} = \mathcal{E}_i^{(m)} \cap \mathcal{F}^{(0)}(x_i)$.
    \end{enumerate}
\end{enumerate}
\end{defn}

We can use these definitions to define the structure of an exact category on the category of parabolic vector bundles over $C$ with fixed chain lengths $N_i$. All morphisms in this category admit a kernel and an image. This follows from the fact that subsheaves of vector bundles on $C$ are always vector bundles (because $C$ is a regular curve).
\end{subsection}

\begin{subsection}{Parabolic weights and degree}
In order to stratify moduli stacks of parabolic bundles, we will define a notion of slope. First we need the following definition.
\begin{defn} \label{defn: parabolic weights}
A set of parabolic weights $\overline{\lambda}$ consists of the data of a tuple of real numbers $\left(\lambda_i^{(1)}, \lambda_i^{(2)}, \, \cdots, \lambda_i^{(N_i)} \right)$ for each $i \in I$ satisfying:
\begin{enumerate}[(a)]
    \item $0 < \lambda_i^{(m)} < 1$ for all $i \in I$ and $1 \leq m \leq N_i$.
    \item $\lambda_i^{(m-1)} < \lambda_i^{(m)}$ for all $i \in I$ and $2 \leq m \leq N_i$.
\end{enumerate}
$N_i$ is called the chain length of $\overline{\lambda}$ at $i$.
\end{defn}

\begin{defn}
Let $\overline{\lambda}$ be a set of parabolic weights with chain length $N_i$ at $i \in I$. We denote by $\text{Vect}_{\overline{\lambda}}$ the exact category of parabolic vector bundles with chain length $N_i$ at each $i \in I$.
\end{defn}

For the rest of this paper, we fix a choice of parabolic weights $\overline{\lambda}$ with chain length $N_i$ at $i \in I$. Unless otherwise stated, all parabolic vector bundles will be in $\text{Vect}_{\overline{\lambda}}$.

\begin{defn} \label{defn: deg of vector bundle}
Let $\mathcal{V} = \left[ \; \mathcal{E}^{(0)} \,\overset{a_i^{(1)}}{\subset} \, \mathcal{E}_i^{(1)} \overset{a_i^{(2)}}{\subset} \cdots\; \overset{a_{i}^{(N_i)}}{\subset} \, \mathcal{E}_i^{(N_i)}= \mathcal{E}^{(0)}(x_i) \;\right]_{i \in I}$
be a parabolic bundle. The degree of $\mathcal{V}$ is defined to be
\[\text{deg}\; \mathcal{V} = \text{deg} \; \mathcal{E}^{(0)} + \sum_{i \in I} \left(n - \sum_{j=1}^{N_i} \lambda_i^{(j)} \, a_i^{(j)} \right) \]
The slope of $\mathcal{V}$ is defined by $\mu\left(\mathcal{V}\right) \vcentcolon = \frac{\text{deg} \, \mathcal{V}}{\text{rank} \, \mathcal{V}} $.
\end{defn}
\begin{remark}
In \cite{mehta-seshadri}, Mehta and Seshadri consider a slightly different notion of parabolic vector bundle. We consider flags of vector spaces inside $\mathcal{E}(x_i)/\mathcal{E}$. \cite{mehta-seshadri} deals with flags in the fiber $\mathcal{E}/\mathcal{E}(-x_i)$. The inclusion $\mathcal{O}_{C}(-x_i) \hookrightarrow \mathcal{O}_{C}$ of the ideal sheaf of $x_i$ induces an identification $\mathcal{E}(x_i)/\mathcal{E} \cong \mathcal{E}/\mathcal{E}(-x_i)$. Therefore our notion of parabolic vector bundle is in correspondence with the one in Mehta-Seshadri. Our altered definition of degree reflects this change of convention. 
\end{remark}
Let us state a couple lemmas that will come in handy in later sections. They are immediate consequences of the definitions given above.
\begin{lemma} \label{lemma: deg of parabolic vs regular vector bundles}
Let $\mathcal{V} = \left[ \; \mathcal{E}^{(0)} \,\overset{a_i^{(1)}}{\subset} \, \mathcal{E}_i^{(1)} \overset{a_i^{(2)}}{\subset} \cdots\; \overset{a_{i}^{(N_i)}}{\subset} \, \mathcal{E}_i^{(N_i)}= \mathcal{E}^{(0)}(x_i) \;\right]_{i \in I}$ be a parabolic vector bundle of rank $n$. Then, we have 
\[ deg \left(\mathcal{E}^{(0)} \right) \leq deg \left(\mathcal{V} \right) \leq deg\left(\mathcal{E}^{(0)} \right) + n|I| \] \qed
\end{lemma}

\begin{lemma} \label{lemma: deg in ses}
Let $0 \rightarrow \mathcal{W} \rightarrow \mathcal{V} \rightarrow \mathcal{Q} \rightarrow 0$ be a short exact sequence in $\text{Vect}_{\overline{\lambda}}$. Then, we have $\text{deg} \, \mathcal{V} \, = \, \text{deg} \, \mathcal{W} \, + \, \text{deg} \, \mathcal{Q}$. \qed
\end{lemma}
We end this section with a useful technical lemma.

\begin{lemma} \label{lemma: saturating subsheaves}
Let $\mathcal{V} = \left[ \; \mathcal{E}^{(0)} \,\overset{a_i^{(1)}}{\subset} \, \mathcal{E}_i^{(1)} \overset{a_i^{(2)}}{\subset} \cdots\; \overset{a_{i}^{(N_i)}}{\subset} \, \mathcal{E}_i^{(N_i)}= \mathcal{E}^{(0)}(x_i) \;\right]_{i \in I} $ and \\$\mathcal{W} = \left[ \;\mathcal{F}^{(0)} \,\overset{b_i^{(1)}}{\subset} \, \mathcal{F}_i^{(1)} \overset{b_i^{(2)}}{\subset} \cdots\; \overset{b_{i}^{(N_i)}}{\subset} \,\mathcal{F}_i^{(N_i)}= \, \mathcal{F}^{(0)}(x_i) \;\right]_{i \in I}$ be two parabolic vector bundles. Suppose that $\mathcal{W}$ is a suboject of $\mathcal{V}$ (not necessarily a parabolic subbundle). Then,
\begin{enumerate}[(a)]
    \item There exists a parabolic subbundle $\mathcal{W}^{sat}$ of $\mathcal{V}$ containing $\mathcal{W}$ as a subobject and satisfying $\mu(\mathcal{W}^{sat}) \geq \mu(\mathcal{W})$.
    \item  If we have $\mu(\mathcal{W}^{sat}) = \mu(\mathcal{W})$ in the construction for part (a), then in fact $\mathcal{W}^{sat} = \mathcal{W}$.
    \item Given two parabolic subobjects $\mathcal{W}_1 \hookrightarrow \mathcal{W}_2 \hookrightarrow \mathcal{V}$, we have $\mathcal{W}_1^{sat} \subset \mathcal{W}_2^{sat}$.
\end{enumerate}
\end{lemma}
\begin{proof}
 Set $n =\text{rank} \, \mathcal{W}$. For all $i \in I$ and $0 \leq m \leq N_i$, define $\mathcal{P}_i^{(m)} \vcentcolon = \mathcal{F}^{(0)}(x_i) \, \cap \, \mathcal{E}^{(m)}_i$. Observe that $P_i^{(m)} \subset \mathcal{F}^{(0)}(x_i)$ is a vector bundle, because subsheaves of locally free sheaves in $C$ are locally free. Hence the diagram of sheaves $\mathcal{U} = \left[ \;\mathcal{P}^{(0)} \,\overset{c_i^{(1)}}{\subset} \, \mathcal{P}_i^{(1)} \overset{c_i^{(2)}}{\subset} \cdots\; \overset{c_{i}^{(N_i)}}{\subset} \,\mathcal{P}_i^{(N_i)}= \, \mathcal{P}^{(0)}(x_i) \;\right]_{i \in I}$ is a parabolic vector bundle. We have that $\mathcal{F}^{(m)}_i \subset \mathcal{P}^{(m)}_i$ and $\mathcal{P}^{(0)} = \mathcal{F}^{(0)}$. Therefore $\mathcal{U}$ contains $\mathcal{W}$ as a suboject. We now claim that $\mu(\mathcal{U}) \geq \mu(\mathcal{W})$. Since $\text{rank} \, \mathcal{U} = \text{rank} \, \mathcal{W} = n$, the claim is equivalent to $\text{deg} \, \mathcal{U} \geq \text{deg} \, \mathcal{W}$. Since $\mathcal{F}^{(0)} = \mathcal{P}^{(0)}$, the inequality of degrees amounts to showing that
 \[ \sum_{i \in I} \sum_{j=1}^{N_i} \lambda_i^{(j)} \, c_i^{(j)} \, \leq \sum_{i \in I} \sum_{j=1}^{N_i} \lambda_i^{(j)} \, b_i^{(j)} \]
 We will actually show that for all $i \in I$ we have $\sum_{j=1}^{N_i} \lambda_i^{(j)} \, c_i^{(j)} \, \leq \sum_{j=1}^{N_i} \lambda_i^{(j)} \, b_i^{(j)}$. We can use summation by parts to rewrite both sides of the inequality.  It suffices to show
 \[ n \lambda^{(N_i)} \, + \,\sum_{j=1}^{N_i-1} (\lambda_i^{(j)} - \lambda^{(j+1)}_i)  \, \sum_{m=1}^{j} c_i^{(j)} \, \leq \,  n \lambda^{(N_i)}\, + \, \sum_{j=1}^{N_i-1} (\lambda_i^{(j)} - \lambda^{(j+1)}_i)  \, \sum_{m=1}^{j} b_i^{(j)}  \]
Since we have $\lambda_i^{(j)} < \lambda_i^{(j+1)}$ for all $1 \leq j \leq N_i-1$, it is sufficient to show that $\sum_{m=1}^{j} c_i^{(m)} \, \geq \sum_{m=1}^{j} b_i^{(m)}$. By definition, we have that $\sum_{m=1}^{j} c_i^{(m)} = \text{dim}_{k} \, \mathcal{P}_i^{(j)} \, / \, \mathcal{F}^{(0)}$. Similarly, we have that $\sum_{m=1}^{j} b_i^{(m)} = \text{dim}_{k} \, \mathcal{F}_i^{(j)} \, / \, \mathcal{F}^{(0)}$. Hence the inequality above is equivalent to
 \[\text{dim}_{k} \, \mathcal{P}_i^{(j)} \, / \, \mathcal{F}^{(0)} \geq \text{dim}_{k} \, \mathcal{F}_i^{(j)} \, / \, \mathcal{F}^{(0)}\]
 This is now obvious, because $\mathcal{P}_i^{(m)} \supset \mathcal{F}_i^{(m)}$.
 Suppose that $\mu(\mathcal{U}) = \mu(\mathcal{W})$. Then all of the inequalities above must be equalities. In particular we have $\text{dim}_{k} \, \mathcal{P}_i^{(j)} \, / \, \mathcal{F}^{(0)} = \text{dim}_{k} \, \mathcal{F}_i^{(j)} \, / \, \mathcal{F}^{(0)}$ for all $i,j$. This implies that $\mathcal{U} = \mathcal{W}$. It is also clear that inclusions of subobjects as in part (c) are preserved by this construction.
 
Replacing $\mathcal{W}$ by $\mathcal{U}$ in the proposition, we can assume from the beginning that we have $\mathcal{F}_i^{(m)} = \mathcal{F}^{(0)}(x_i) \, \cap \, \mathcal{E}^{(m)}_i$. Let us define $\mathcal{Q}_i^{(m)} \vcentcolon = \mathcal{E}_i^{(m)} \, / \, \mathcal{F}_i^{(m)}$. This means that we have the following commutative diagram with exact columns:
\begin{figure}[H]
\centering
\begin{tikzcd}
    0  & 0 & \cdots & 0 & 0 & 0\\
    \mathcal{Q}^{(0)} \ar [u] \ar[r, symbol= \subset] &  \mathcal{Q}_i^{(1)} \ar [u] \ar[r, symbol= \subset] & \cdots & \mathcal{Q}_i^{(Ni-2)} \ar [u] \ar[r, symbol= \subset] & \mathcal{Q}_i^{(N_i-1)} \ar [u] \ar[r, symbol= \subset] & \mathcal{Q}^{(0)}(x_i) \ar [u]\\
    \mathcal{E}^{(0)} \ar [u] \ar[r, symbol= \subset] &  \mathcal{E}_i^{(1)} \ar [u] \ar[r, symbol= \subset] & \cdots & \mathcal{E}_i^{(Ni-2)} \ar [u] \ar[r, symbol= \subset] & \mathcal{E}_i^{(N_i-1)} \ar [u] \ar[r, symbol= \subset] & \mathcal{E}^{(0)}(x_i) \ar [u]\\
    \mathcal{F}^{(0)} \ar [u] \ar[r, symbol= \subset] &  \mathcal{F}_i^{(1)} \ar [u] \ar[r, symbol= \subset] & \cdots & \mathcal{F}_i^{(Ni-2)} \ar [u] \ar[r, symbol= \subset] & \mathcal{F}_i^{(N_i-1)} \ar [u] \ar[r, symbol= \subset] & \mathcal{F}^{(0)}(x_i) \ar [u]\\
    0 \ar[u]  & 0 \ar[u] & \cdots & 0 \ar[u] & 0 \ar[u] & 0 \ar[u]
\end{tikzcd}
\caption{Diagram 1}
\label{diagram: 1}
\end{figure}
	The problem is that the $\mathcal{Q}_i^{(m)}$s do not necessarily form a parabolic bundle, since they could have nontrivial torsion. In order to solve this, we kill the $C$-torsion. Let $\xi$ denote the generic point of $C$. Let $\iota : \xi \longrightarrow C$ be the corresponding inclusion. Define
	\[K_i^{(m)} \vcentcolon = \text{Ker} \left(  \mathcal{Q}_i^{(m)} \, \xrightarrow{unit} \, i_* i^* \, \mathcal{Q}_i^{(m)} \right)\]
	\[\mathcal{Q}_i^{(m), \, sat} \vcentcolon = \text{Im} \left( \mathcal{Q}_i^{(m)} \, \xrightarrow{unit} \, i_* i^* \, \mathcal{Q}_i^{(m)} \right)\]
	\[\mathcal{F}_i^{(m), \, sat} \vcentcolon = \text{Ker} \left( \mathcal{E}_i^{(m)} \, \twoheadrightarrow \, \mathcal{Q}_i^{(m)} \rightarrow \,  \mathcal{Q}_i^{(m), \, sat} \right)\]
	There is a similar commutative diagram with exact columns:
\begin{figure}[H]
\centering
\begin{adjustbox}{width=\textwidth}
\begin{tikzcd}
    0  & 0 & \cdots & 0 & 0 & 0\\
    \mathcal{Q}^{(0), \, sat} \ar [u] \ar[r, symbol= \subset] &  \mathcal{Q}_i^{(1), \, sat} \ar [u] \ar[r, symbol= \subset] & \cdots & \mathcal{Q}_i^{(Ni-2), \, sat} \ar [u] \ar[r, symbol= \subset] & \mathcal{Q}_i^{(N_i-1), \, sat} \ar [u] \ar[r, symbol= \subset] & \mathcal{Q}^{(0), \, sat}(x_i) \ar [u]\\
    \mathcal{E}^{(0)} \ar [u] \ar[r, symbol= \subset] &  \mathcal{E}_i^{(1)} \ar [u] \ar[r, symbol= \subset] & \cdots & \mathcal{E}_i^{(Ni-2)} \ar [u] \ar[r, symbol= \subset] & \mathcal{E}_i^{(N_i-1)} \ar [u] \ar[r, symbol= \subset] & \mathcal{E}^{(0)}(x_i) \ar [u]\\
    \mathcal{F}^{(0), \, sat} \ar [u] \ar[r, symbol= \subset] &  \mathcal{F}_i^{(1), \, sat} \ar [u] \ar[r, symbol= \subset] & \cdots & ^s\mathcal{F}_i^{(Ni-2), \, sat} \ar [u] \ar[r, symbol= \subset] & \mathcal{F}_i^{(N_i-1), \, sat} \ar [u] \ar[r, symbol= \subset] & \mathcal{F}^{(0), \, sat}(x_i) \ar [u]\\
    0 \ar[u]  & 0 \ar[u] & \cdots & 0 \ar[u] & 0 \ar[u] & 0 \ar[u]
\end{tikzcd}
\end{adjustbox}
\caption{Diagram 2}
\label{diagram: 2}
\end{figure}
We set $\mathcal{W}^{sat} \vcentcolon = \left[ \;\mathcal{F}^{(0), \, sat} \;\overset{c_i^{(1)}}{\subset} \; \mathcal{F}_i^{(1), \, sat} \; \overset{c_i^{(2)}}{\subset} \cdots\; \overset{c_{i}^{(N_i)}}{\subset} \;\mathcal{F}_i^{(N_i),\, sat}= \, \mathcal{F}^{(0),\, sat}(x_i) \;\right]_{i \in I}$. By construction, $\mathcal{W}^{sat}$ is a parabolic subbundle of $\mathcal{V}$ containing $\mathcal{W}$. It is clear that this construction preserves inclusions of subobject as described in part (c). We are left to show the claims about the slopes. 
By definition, we have the following commutative diagram with exact rows:
\begin{figure}[H]
\centering
\begin{tikzcd}
0 \ar[r] &  K_i^{(m-1)} \ar[r] \ar[d] & \mathcal{Q}_i^{(m-1)} \ar[r] \ar[d] & \mathcal{Q}_i^{(m-1), \, sat} \ar[r] \ar[d] & 0\\
0 \ar[r]&  K_i^{(m)} \ar[r] & \mathcal{Q}_i^{(m)} \ar[r] & \mathcal{Q}_i^{(m), \, sat} \ar[r] & 0
\end{tikzcd}
\end{figure}
An application of the Snake Lemma yields a surjection $ \mathcal{Q}_i^{(m)} \, / \, \mathcal{Q}_i^{(m-1)} \, \twoheadrightarrow \mathcal{Q}_i^{(m), \, sat} \, / \, \mathcal{Q}_i^{(m-1), \, sat}$. Therefore, we have 
\[\text{dim}_{k} \,\left( \mathcal{Q}_i^{(m), \, sat}\, / \,  \mathcal{Q}_i^{(m-1), \, sat}\right) \; \leq \; \text{dim}_{k} \, \left(\mathcal{Q}_i^{(m)}\, / \,  \mathcal{Q}_i^{(m-1)}\right)\]
By applying the Snake Lemma to consecutive columns of Diagram \ref{diagram: 1} and Diagram \ref{diagram: 2} respectively, we obtain the two short exact sequences
\[0 \longrightarrow \mathcal{F}_i^{(m)} \, / \,\mathcal{F}_i^{(m-1)} \longrightarrow \mathcal{E}_i^{(m)} \, / \, \mathcal{E}_i^{(m-1)} \longrightarrow \mathcal{Q}_i^{(m)} \, / \, \mathcal{Q}_i^{(m-1)}  \longrightarrow 0 \]
\[0 \longrightarrow \mathcal{F}_i^{(m), \, sat} \, / \,\mathcal{F}_i^{(m-1), \, sat} \longrightarrow \mathcal{E}_i^{(m),\, sat} \, / \, \mathcal{E}_i^{(m-1),\, sat} \longrightarrow \mathcal{Q}_i^{(m),\, sat} \, / \, \mathcal{Q}_i^{(m-1), \, sat}  \longrightarrow 0 \]
We conclude that $c_i^{(m)} \geq b_i^{(m)}$. But note that for all $i \in I$ we have 
\[\sum_{m=1}^{N_i} c_i^{(m)} = \text{rank} \left(\mathcal{W}^{sat}\right) = \text{rank} \left(\mathcal{W}\right) = \sum_{m=1}^{N_i} b_i^{(m)}\]
Therefore, we must have $c_i^{(m)} = b_i^{(m)}$ for all $i$ and $m$. By construction we have $\mathcal{F}^{(0)} \subset \, \mathcal{F}^{(0), \, sat}$. In particular, $deg \,\mathcal{F}^{(0)} \leq deg\, \mathcal{F}^{(0), \, sat}$ as vector bundles. We therefore get a chain of (in)equalities:
\begin{equation*} 
\begin{aligned}
deg \, \mathcal{W}^{sat} &  = \; deg \; \mathcal{F}^{(0), \,  sat} + \sum_{i \in I} \left(\text{rank} \, \mathcal{W}^{sat} - \sum_{m=1}^{N_i} \lambda_i^{(m)} \, c_i^{(m)} \right) \\
&  = \; deg \; \mathcal{F}^{(0), \,  sat} + \sum_{i \in I} \left(\text{rank} \, \mathcal{W}^{sat} - \sum_{m=1}^{N_i} \lambda_i^{(m)} \, b_i^{(m)} \right) \\
& \geq \; deg \; \mathcal{F}^{(0)} \, + \; \sum_{i \in I} \left(\text{rank} \, \mathcal{W} - \sum_{m=1}^{N_i} \lambda_i^{(m)} \, b_i^{(m)} \right) \\
& = \; deg \, \mathcal{W}
\end{aligned}
\end{equation*}
Part (a) of the lemma follows. 

Suppose that $deg \, \mathcal{W} = deg \, \mathcal{W}^{sat}$. From the chain of (in)equalities above we can conclude that $deg \,\mathcal{F}^{(0)} = deg \, \mathcal{F}^{(0), \, sat}$. Hence we must have $\mathcal{F}^{(0)} = \,\mathcal{F}^{(0), \, sat}$. We can use the fact that $\mathcal{F}_i^{(m)} \subset \, \mathcal{F}_i^{(m), \, sat}$ and that
\[  \text{rank} \, \left(\mathcal{F}_i^{(m)} \, / \, \mathcal{F}^{(0)}\right) \, = \, \sum_{j=1}^{m} b_i^{(j)} \, = \, \text{rank} \, \left(\mathcal{F}_i^{(m), \, sat} \, / \, \mathcal{F}^{(0)}\right) \]
in order to conclude that $\mathcal{F}_i^{(m)} = \, \mathcal{F}_i^{(m), \, sat}$ for all $i$ and $m$. Therefore, we have that $ \mathcal{W} = \, \mathcal{W}^{ sat}$, as stipulated in part (b).
\end{proof}
\end{subsection}

\begin{subsection}{Harder-Narasimhan filtrations}
\begin{defn}
A parabolic vector bundle $\mathcal{V}$ in $\text{Vect}_{\overline{\lambda}}$ is called semistable (with respect to $\overline{\lambda}$) if for all parabolic subbundles $\mathcal{W} \subset \mathcal{V}$, we have $\mu\left(\mathcal{W}\right) \, \leq \, \mu\left(\mathcal{V}\right)$.
\end{defn}

\begin{defn} \label{defn: HN filtration}
Let $\mathcal{V}$ be a parabolic vector bundle in $\text{Vect}_{\overline{\lambda}}$. A Harder-Narasimhan filtration of $\mathcal{V}$ is a filtration $0 = \mathcal{W}_0 \subset \mathcal{W}_1 \subset \cdots \, \subset \mathcal{W}_{l-1} \subset \mathcal{W}_l = \mathcal{V}$ by parabolic subbundles $\mathcal{W}_j$ satisfying the following two conditions:
\begin{enumerate}[(a)]
    \item For all $1 \leq j \leq l$, the parabolic bundle $\mathcal{W}_j \, / \, \mathcal{W}_{j-1}$ is semistable.
    \item For all $2 \leq j \leq l$, we have $\mu\left( \mathcal{W}_{j-1} \, / \, \mathcal{W}_{j-2} \right) \, > \, \mu\left( \mathcal{W}_{j} \, / \, \mathcal{W}_{j-1} \right) $.
\end{enumerate}
\end{defn}
In order to prove the existence of such filtrations, we will need a few lemmas along the way.

\begin{lemma} \label{lemma: filtration by parabolic line bundles}
Let $\mathcal{V}$ be a parabolic bundle. There exists a filtration $0 = \mathcal{W}_0 \subset \mathcal{W}_1 \subset \cdots \, \subset \mathcal{W}_{l-1} \subset \mathcal{W}_l = \mathcal{V}$
by subbundles $\mathcal{W}_j$ such that $\mathcal{W}_j \, / \, \mathcal{W}_{j-1}$ has rank 1 for all $1 \leq j \leq l$.
\end{lemma}
\begin{proof}
By induction on the rank of $\mathcal{V}$, it suffices to show that $\mathcal{V}$ has a parabolic subbundle $\mathcal{W}$ of rank 1. Suppose that $\mathcal{V} = \left[ \; \mathcal{E}^{(0)} \,\overset{a_i^{(1)}}{\subset} \, \mathcal{E}_i^{(1)} \overset{a_i^{(2)}}{\subset} \cdots\; \overset{a_{i}^{(N_i)}}{\subset} \, \mathcal{E}_i^{(N_i)}= \mathcal{E}^{(0)}(x_i) \;\right]_{i \in I}$. It is well known that $\mathcal{E}^{(0)}$ admits a vector subbundle $\mathcal{F}^{(0)}$ of rank 1. Indeed, by Riemann-Roch there exists some $N>0$ such that $\mathcal{E}^{(0)}(N)$ has a nonzero global section $\mathcal{O}_C \rightarrow \mathcal{E}^{(0)}(N)$. 
Tensoring by $\mathcal{O}(-N)$ gives us a subsheaf $\mathcal{O}(-N) \rightarrow \mathcal{E}^{(0)}$ that is a line bundle. Let $j: \xi \rightarrow C$ denote the inclusion of the generic point $\xi$ into $C$.  Set $\mathcal{F}^{(0)} \vcentcolon = \text{Ker}\left( \mathcal{E}^{(0)} \rightarrow \mathcal{E}^{(0)}\, / \, \mathcal{O}(-N) \xrightarrow{unit} j_* j^{*} \;  \mathcal{E}^{(0)}\, / \, \mathcal{O}(-N) \right)$. Then the quotient $\mathcal{E}^{(0)}\, / \, \mathcal{F}^{(0)}$ is torsion free, because it is a subsheaf of $j_* j^{*} \;  \mathcal{E}^{(0)}\, / \, \mathcal{O}(-N)$. So $\mathcal{F}^{(0)}$ is a subbundle of $\mathcal{E}^{(0)}$. On the other hand, we have that $\mathcal{F}^{(0)}\, / \, \mathcal{O}(-N)$ is torsion. Therefore $\text{rank} \, \mathcal{F}^{(0)} = \text{rank} \, \mathcal{O}(-N) = 1$.

We can now set $\mathcal{F}_i^{(m)} \vcentcolon = \mathcal{E}_i^{(m)} \, \cap \, \mathcal{F}^{(0)}(x_i)$ and define $\mathcal{W}$ to be the diagram of sheaves $\mathcal{W} = \left[ \;\mathcal{F}^{(0)} \,\overset{b_i^{(1)}}{\subset} \, \mathcal{F}_i^{(1)} \overset{b_i^{(2)}}{\subset} \cdots\; \overset{b_{i}^{(N_i)}}{\subset} \,\mathcal{F}_i^{(N_i)}= \, \mathcal{F}^{(0)}(x_i) \;\right]_{i \in I}$. It is clear that $\mathcal{W}$ is a parabolic subbundle of rank 1, because all $\mathcal{F}_i^{(m)} \subset \mathcal{F}^{(0)}(x_i)$ are line bundles.
\end{proof}

\begin{coroll} \label{coroll: boundedness degrees subbundles}
Let $\mathcal{V}$ be a parabolic vector bundle of rank $n$. Then,
\begin{enumerate}[(a)]
\item There is a constant $M$ such that any subbundle $\mathcal{U} \subset \mathcal{V}$ satisfies $\text{deg} \, \mathcal{U} \leq M$. 
\item The set of all slopes of parabolic subbundles of $\mathcal{V}$ has a maximal element.
\end{enumerate}
\end{coroll}
\begin{proof}\
\begin{enumerate}[(a)]
\item Let $0 = \mathcal{W}_0 \subset \mathcal{W}_1 \subset \cdots \, \subset \mathcal{W}_{l-1} \subset \mathcal{W}_l = \mathcal{V}$ be a filtration as in Lemma \ref{lemma: filtration by parabolic line bundles}. We set $M = \underset{A \subset \{1, 2, \cdots, l\}}{\text{max}} \, \left\{\sum_{j \in A} \text{deg} \,\mathcal{W}_j \, / \, \mathcal{W}_{j-1} \right\}$. Let $\mathcal{U}$ be a parabolic subbundle of $\mathcal{V}$. We can take intersections in the category of parabolic vector bundles, since $\text{Vect}_{\overline{\lambda}}$ admits kernels. Consider the induced filtration of $\mathcal{U}$ by parabolic subbundles
\[ 0 = \mathcal{W}_0 \, \cap \mathcal{U} \, \subset \mathcal{W}_1\, \cap \mathcal{U} \, \subset \cdots \, \subset \mathcal{W}_{l-1}\, \cap \mathcal{U} \, \subset \mathcal{W}_l\, \cap \mathcal{U} \, = \mathcal{U} \]
For all $1 \leq j \leq l$ we have that $\left(\mathcal{W}_j \, \cap \, \mathcal{U}\right) \, / \, \left(\mathcal{W}_{j-1} \, \cap \, \mathcal{U} \right)$ is isomorphic to a parabolic subbundle of  $\mathcal{W}_j \, / \, \mathcal{W}_{j-1}$. Since $\mathcal{W}_j \, / \, \mathcal{W}_{j-1}$ has rank 1, we must have either
$\left(\mathcal{W}_j \, \cap \, \mathcal{U}\right) \, / \, \left(\mathcal{W}_{j-1}\, \cap \, \mathcal{U} \right) \, = \, 0$ or $\left(\mathcal{W}_j \, \cap \, \mathcal{U}\right) \, / \, \left(\mathcal{W}_{j-1}\, \cap \, \mathcal{U} \right) \; \cong \; \mathcal{W}_j \, / \, \mathcal{W}_{j-1}$. Let
\[ A_{\mathcal{U}} \vcentcolon = \left\{ \; j \in \{1, 2, \cdots, l\} \; \mid \; \left(\mathcal{W}_j \, \cap \, \mathcal{U}\right) \, / \, \left(\mathcal{W}_{j-1}\, \cap \, \mathcal{U} \right) \; \cong \; \mathcal{W}_j \, / \, \mathcal{W}_{j-1}  \; \right\} \]
A repeated application of Lemma \ref{lemma: deg in ses} yields $\text{deg} \, \mathcal{U}  =  \sum_{j \in A_{\mathcal{U}}} \text{deg} \,\mathcal{W}_j \, / \, \mathcal{W}_{j-1} \leq M$.
\item Fix a rank $k$ with $1 \leq k \leq n$. We claim that the set $\left\{ \; \mu(\mathcal{U}) \; \mid \; \mathcal{U} \; \text{is a subbundle of $\mathcal{V}$ of rank $k$} \; \right\}$ has a maximal element. Once this claim is proven, we can take the maximum among all possible ranks $1 \leq k \leq n$ to produce the maximal element as in part (b) of the corollary. We are left to prove the claim. 

Let $\mathcal{U}$ be any subbundle of rank $k$. By definition, we have $\mu(\mathcal{U}) = \frac{1}{k} \text{deg} \, \mathcal{U}$. So it suffices to produce a subbundle $\mathcal{U}$ of maximal degree among those of rank $k$. Consider the finite set $S \vcentcolon = \left\{ A_{\mathcal{U}} \; \mid \; \mathcal{U} \, \text{is a subbundle of $\mathcal{V}$ of rank $k$} \; \right\}$, where $A_{\mathcal{U}}$ is defined as in part (a). Let $A_{\mathcal{U}}$ be an element of $S$ maximizing the quantity $\sum_{j \in A_{\mathcal{U}}} \text{deg} \,\mathcal{W}_j$. The same argument as in part (a) shows that $\mathcal{U}$ is a subbundle of rank $k$ with maximal degree among subbundles of the same rank.
\end{enumerate}
\end{proof}

\begin{lemma} \label{lemma: semistable bundles and ses}
Let $\mathcal{V}$, $\mathcal{W}$ be two semistable parabolic vector bundles. Then,
\begin{enumerate}[(a)]
\item For any subobject $\mathcal{P} \hookrightarrow \mathcal{W}$ (not necessarily a parabolic subbundle), we have $\mu(\mathcal{P}) \leq \mu(\mathcal{W})$.
\item If $\mathcal{Q}$ is a parabolic quotient bundle of $\mathcal{W}$, then $\mu\left(\mathcal{Q}\right) \, \geq \, \mu\left(\mathcal{W}\right)$.
\item If $\mu\left(\mathcal{W}\right) \, > \, \mu\left(\mathcal{V}\right)$, then $\text{Hom}_{\text{Vect}_{\overline{\lambda}}} \left(\mathcal{W}, \, \mathcal{V}\right) = 0$.
\item Let $\mathcal{P}$ fit into an exact sequence $0 \rightarrow \mathcal{V} \rightarrow \mathcal{P} \rightarrow \mathcal{W} \rightarrow 0$. If $\mathcal{U}$ is a parabolic subbundle of $\mathcal{P}$, then $\mu(U) \leq \text{max} \left\{ \mu(\mathcal{V}) , \, \mu(\mathcal{W}) \right\}$
\end{enumerate}
\end{lemma}
\begin{proof}
\begin{enumerate}[(a)]
    \item By Lemma \ref{lemma: saturating subsheaves}, we have a subbundle $\mathcal{P}^{sat}$ of $\mathcal{W}$ with $\mu(\mathcal{P}^{sat}) \geq \mu(\mathcal{P})$. By semistability, $\mu(\mathcal{W}) \geq \mu(\mathcal{P}^{sat})$.
    
    \item Let $m = \text{rank} \, \mathcal{Q}$ and $n = \text{rank}\, \mathcal{W}$. Define $\mathcal{K}$ to be the kernel of the admissible epimorphism $\mathcal{W} \twoheadrightarrow \mathcal{Q}$. By Lemma \ref{lemma: deg in ses}, we have $\text{deg} \, \mathcal{W} = \text{deg} \, \mathcal{K} \, + \, \text{deg}\, \mathcal{Q}$. We can rewrite this as $\mu(\mathcal{W}) = \frac{n-m}{n} \mu(\mathcal{K}) + \frac{m}{n} \mu(\mathcal{Q})$. Since $\mathcal{W}$ is semistable, we have $\mu(\mathcal{K}) \leq \mu(\mathcal{W})$. Therefore
    \[ \mu(\mathcal{Q}) = \frac{n}{m} \left( \mu(\mathcal{W}) - \frac{n-m}{n} \mu(\mathcal{K}) \right) \geq  \frac{n}{m} \left( \mu(\mathcal{W}) - \frac{n-m}{n} \mu(\mathcal{W}) \right) = \mu(\mathcal{W}) \]
    \item Suppose that there exists a nontrivial morphism $f: \mathcal{W} \rightarrow \mathcal{V}$. Consider the nonzero image $\mathcal{I} = \text{Im}\left( \,\mathcal{W} \xrightarrow{f} \mathcal{V} \, \right)$ in $\text{Vect}_{\overline{\lambda}}$. Then $\mathcal{I}$ is a parabolic quotient bundle of $\mathcal{W}$. By part (b), we know that $\mu(\mathcal{I}) \geq \mu(\mathcal{W})$. Therefore $\mu(\mathcal{I}) > \mu(\mathcal{V})$. But $\mathcal{I}$ is a suboject of $\mathcal{V}$. This contradicts semistability of $\mathcal{V}$ by part (a).
    \item Let $M \vcentcolon =\text{max} \left\{ \mu(\mathcal{V}) , \, \mu(\mathcal{W}) \right\}$. Define $\mathcal{Q}$ to be the image of the composition $\mathcal{U}\,  \hookrightarrow \, \mathcal{P} \, \twoheadrightarrow \, \mathcal{W}$ in $\text{Vect}_{\overline{\lambda}}$. Then $\mathcal{Q}$ is a quotient bundle of $\mathcal{U}$. We get a short exact sequence $0 \longrightarrow \mathcal{K} \longrightarrow \mathcal{U} \longrightarrow \mathcal{Q} \longrightarrow 0$, where $\mathcal{K} \hookrightarrow \mathcal{V}$ and $\mathcal{Q} \hookrightarrow \mathcal{W}$. By part (a), we see that $\mu(\mathcal{K}) \leq \mu(\mathcal{V}) \leq M$. Similarly, $\mu(\mathcal{Q}) \leq \mu(\mathcal{W}) \leq M$. Therefore, we have
    \[ \mu(\mathcal{U}) = \frac{\text{rank}\, \mathcal{K}}{\text{rank} \, \mathcal{U}} \, \mu(\mathcal{K}) \, + \, \frac{\text{rank}\, \mathcal{Q}}{\text{rank} \, \mathcal{U}} \, \mu(\mathcal{Q}) \,  \leq \, \frac{\text{rank}\, \mathcal{K}}{\text{rank} \, \mathcal{U}} \, M \, + \,\frac{\text{rank}\, \mathcal{Q}}{\text{rank} \, \mathcal{U}} \, M \, = \, M\]
\end{enumerate}
\end{proof}

\begin{lemma} \label{lemma: existence maximal destibilizing subbundle}
Let $\mathcal{V}$ be a parabolic vector bundle in $\text{Vect}_{\overline{\lambda}}$. Let $\mu$ denote the maximal element of the set of all slopes of subbundles of $\mathcal{V}$, as in part (b) of Corollary \ref{coroll: boundedness degrees subbundles}. Then,
\begin{enumerate}[(a)]
\item There exists a unique maximal semistable subbundle $\mathcal{U}$ of $\mathcal{V}$ such that $\mu(\mathcal{U}) = \mu$ and $\mathcal{U}$ contains all other subbundles of slope $\mu$. We will call $\mathcal{U}$ the maximal destabilizing subbundle of $\mathcal{V}$.
\item If $\mathcal{U}$ is as in part (a), then all subobjects $\overline{\mathcal{W}}\,  \hookrightarrow \, \mathcal{V} \, / \, \mathcal{U}$ satisfy $\mu(\overline{\mathcal{W}}) < \mu$.
\item For $\mathcal{U}$ as in part (a), we have $\text{Hom}_{\text{Vect}_{\overline{\lambda}}}\left( \mathcal{U} , \; \mathcal{V} \, / \, \mathcal{U} \, \right)=0$.
\end{enumerate}
\end{lemma}
\begin{proof}
Uniqueness of such subbundle $\mathcal{U}$ is clear. It suffices to prove existence. Let $\mathcal{U}$ be a parabolic subbundle of $\mathcal{V}$ with $\mu(\mathcal{U}) = \mu$ that has maximal rank among the subbundles with slope $\mu$. By the definition of $\mu$, it is clear that $\mathcal{U}$ is semistable.

We claim that $\mathcal{U}$ satisfies the statement in part (b) of the proposition. Suppose that $\overline{\mathcal{W}}$ is a subobject of $\mathcal{V} \, / \, \mathcal{U}$ with slope $\mu(\overline{\mathcal{W}}) \geq \mu$. By the axioms of exact categories, there is a subobject $\mathcal{W}$ of $\mathcal{V}$ with $\mathcal{U} \subset \mathcal{W}$ and a short exact sequence $0 \longrightarrow \mathcal{U} \longrightarrow \mathcal{W} \longrightarrow \overline{\mathcal{W}} \rightarrow 0$.

By additivity of degree in short exact sequences (Lemma \ref{lemma: deg in ses}), we have 
  \[ \mu(\mathcal{W}) = \frac{\text{rank}\, \mathcal{U}}{\text{rank} \, \mathcal{W}} \, \mu(\mathcal{U}) \, + \, \frac{\text{rank}\, \overline{\mathcal{W}}}{\text{rank} \, \mathcal{W}} \, \mu(\overline{\mathcal{W}}) \,  \geq \, \mu\]
  contradicting maximality of the rank of $\mathcal{U}$.
  
  Now let $\mathcal{P}$ be any parabolic subbundle of $\mathcal{V}$ with slope $\mu(\mathcal{P}) = \mu$. By the definition of $\mu$, $\mathcal{P}$ is semistable. We claim that $\mathcal{P} \subset \mathcal{U}$. In order to see this, it suffices to show that the composition $\mathcal{P} \hookrightarrow \mathcal{V} \twoheadrightarrow \mathcal{V} \, / \, \mathcal{U}$ is 0. Let us show that $\text{Hom}_{\text{Vect}_{\overline{\lambda}}}\left( \mathcal{P} , \; \mathcal{V} \, / \, \mathcal{U} \, \right)=0$. This also implies part (c) as a special case.
  
  For the sake of contradiction, suppose that $\psi \in \text{Hom}_{\text{Vect}_{\overline{\lambda}}}\left( \mathcal{P} , \; \mathcal{V} \, / \, \mathcal{U} \, \right)$ is nonzero. Let $\mathcal{I}$ be the nontrivial image $\text{Im}(\psi)$ of $\mathcal{P}$ in $\mathcal{V} \, / \, \mathcal{U}$. $\mathcal{I}$ is a parabolic quotient bundle of $\mathcal{P}$. By part (b) of Lemma \ref{lemma: semistable bundles and ses} we have $\mu(\mathcal{I}) \geq \mu(\mathcal{P})$. Therefore $\mathcal{I}$ is a suboject of $\mathcal{V} \, / \, \mathcal{U}$ with slope $\mu(\mathcal{I}) \geq \mu$. We get a contradiction, since we have proven in the paragraph above that this is not possible.
\end{proof}

Now we are ready to prove the main proposition of this section.
\begin{prop} \label{prop: existence and uniqueness of HN filtrations}
Let $\mathcal{V}$ be a parabolic vector bundle in $\text{Vec}_{\overline{\lambda}}$. Then $\mathcal{V}$ admits a unique Harder-Narasimhan filtration.
\end{prop}
\begin{proof}
Let $\mathcal{U}_1$ be the maximal destabilizing subbundle of $\mathcal{V}$. Set $\mu_1 = \mu(\mathcal{U}_1)$. We can iterate this process by recursion, setting $\overline{\mathcal{U}_j}$ to be the maximal destabilizing subbundle of $\mathcal{V} \, / \, \mathcal{U}_{j-1}$ and letting $\mathcal{U}_j$ be the unique subbundle of $\mathcal{V}$ containing $\mathcal{U}_{j-1}$ such that $\overline{\mathcal{U}_j} = \mathcal{U}_j \, / \, \mathcal{U}_{j-1}$. Define $\mu_j \vcentcolon = \mu(\overline{\mathcal{U}_j})$. By part (b) of Lemma \ref{lemma: saturating subsheaves}, we see that $\mu_j > \mu_{j-1}$. Since the rank of $\mathcal{V}$ is finite, this process must terminate. Hence $\mathcal{U}_l = \mathcal{V}$ for some $l$. We get a filtration
\[ 0  \subset \mathcal{U}_1 \subset \cdots \, \subset \mathcal{U}_{l-1} \subset \mathcal{U}_l = \mathcal{V} \]
which by construction is a Harder-Narasimhan filtration.

For uniqueness, we can induct on the length $l$ of the filtration above. We are reduced to showing that for any Harder-Narasimhan filtration $0 = \mathcal{W}_0 \subset \mathcal{W}_1 \subset \cdots \, \subset \mathcal{W}_{m-1} \subset \mathcal{W}_m = \mathcal{V}$
we have that $\mathcal{W}_1$ is the maximal destabilizing subsheaf $\mathcal{U}_1$ of $\mathcal{V}$. By applying part (d) of Lemma \ref{lemma: semistable bundles and ses} several times, we conclude that we must have $\mu(\mathcal{W}_1) = \mu(\mathcal{U}_1) = \mu_1$. Therefore, by maximality of $\mathcal{U}_1$, we must have $\mathcal{U}_1 \supset \mathcal{W}_1$. Consider $\mathcal{U}_1 \, / \, \mathcal{W}_1$. This is a subbundle of $\mathcal{V} \, / \, \mathcal{W}_1$. Another sequence of applications of part (d) of Lemma \ref{lemma: semistable bundles and ses} shows that if $\mathcal{U}_1 \, / \, \mathcal{W}_1$ is nontrivial, then we must have $\mu\left( \mathcal{U}_1 \, / \, \mathcal{W}_1\right) < \mu_1$. This contradicts part (b) of Lemma \ref{lemma: semistable bundles and ses} applied to the semistable bundle $\mathcal{U}_1$. Hence $\mathcal{U}_1 \, / \, \mathcal{W}_1$ is trivial, and therefore $\mathcal{U}_1 = \mathcal{W}_1$.
\end{proof}
\end{subsection}

\begin{subsection}{Parabolic vector bundles in families} \label{subsection: parabolic vb families}
In this subsection we extend the theory of parabolic vector bundles to the relative setting.
\begin{defn} \label{defn: parabolic vb in family}
Let $S$ be a $k$-scheme. A parabolic vector bundle $\mathcal{V}$ of rank $n$ over $C \times S$ (with parabolic structure at $\{x_i\}$) consists of the data of:
\begin{enumerate}[(a)]
    \item A vector bundle $\mathcal{E}^{(0)}$ of rank $n$ on $C\times S$.
    \item For each $i \in I$, a chain of vector bundles $\mathcal{E}^{(0)} \,\subset \, \mathcal{E}_i^{(1)} \subset \cdots\; \subset \, \mathcal{E}_i^{(N_i)} = \mathcal{E}^{(0)}(x_i)$.
\end{enumerate}
such that for all $i \in I$ and all $1 \leq m \leq N_i$, the quotient $\mathcal{E}^{(m)}_i \, / \, \mathcal{E}_i^{(m-1)}$ is $S$-flat and the restriction $\mathcal{E}^{(m)}_i \, / \, \mathcal{E}_i^{(m-1)}|_{x_i \times S}$ is a vector bundle of constant rank on $x_i \times S$.

For each $i \in I$, the number $N_i$ will be called the chain length at $i$.
\end{defn}

\begin{remark}
Since $\mathcal{E}^{(m)}_i \, / \, \mathcal{E}_i^{(m-1)}|_{x_i \times S}$ is finitely presented and $S$-flat, it must be a vector bundle. The reason why we require $\mathcal{E}^{(m)}_i \, / \, \mathcal{E}_i^{(m-1)}|_{x_i \times S}$ to have constant rank is to simplify notation. Note that this always true locally on $S$.
\end{remark}
For a parabolic vector bundle $\mathcal{V}$ over $C \times S$, we will write
\[\mathcal{V} = \left[ \; \mathcal{E}^{(0)} \,\overset{a_i^{(1)}}{\subset} \, \mathcal{E}_i^{(1)} \overset{a_i^{(2)}}{\subset} \cdots\; \overset{a_{i}^{(N_i)}}{\subset} \, \mathcal{E}_i^{(N_i)}= \mathcal{E}^{(0)}(x_i) \;\right]_{i \in I} \]
to convey the information that $a_i^{(m)} = \text{rank} \, \mathcal{E}^{(m)}_i \, / \, \mathcal{E}_i^{(m-1)}|_{x_i \times S}$.
\begin{defn} \label{defn: pullback of parabolic vector bundles}
Let $S$ be a $k$-scheme and let $\mathcal{V} = \left[ \; \mathcal{E}^{(0)} \,\overset{a_i^{(1)}}{\subset} \, \mathcal{E}_i^{(1)} \overset{a_i^{(2)}}{\subset} \cdots\; \overset{a_{i}^{(N_i)}}{\subset} \, \mathcal{E}_i^{(N_i)}= \mathcal{E}^{(0)}(x_i) \;\right]_{i \in I}$ be a parabolic vector bundle over $C \times S$. For any $S$-scheme $f: T \rightarrow S$, the pullback of $\mathcal{V}$ is defined to be the parabolic vector bundle $f^{*}\mathcal{V}$ on $C \times T$ given by
\begin{gather*}f^{*}\mathcal{V} = \left[ \;(id_C \times f)^{*} \mathcal{E}^{(0)} \,\overset{a_i^{(1)}}{\subset} \, (id_C \times f)^{*}\mathcal{E}_i^{(1)} \overset{a_i^{(2)}}{\subset} \cdots\; \overset{a_{i}^{(N_i)}}{\subset} \, (id_C \times f)^{*}\mathcal{E}_i^{(N_i)}= (id_C \times f)^{*}\mathcal{E}^{(0)}(x_i) \;\right]_{i \in I}
\end{gather*}
It will often be the case that the choice of a map $f$ is understood and unambiguous, and then we will write $\mathcal{V}|_{C\times T}$ instead of $f^{*} \mathcal{V}$.
\end{defn}

\begin{remark}
It is necessary to know that $\mathcal{E}_i^{(m)} \, / \, \mathcal{E}_i^{(m-1)}$ is $S$-flat in order to be able to conclude that the inclusion $\mathcal{E}_i^{(m-1)} \subset \mathcal{E}_i^{(m)}$ remains a monomorphism after pulling back by $f$. Note that the ranks $a_i^{(m)}$ are preserved by pullbacks.
\end{remark}

\begin{defn} \label{defn: equivalence of types of vb in families}
Let $\mathcal{V} = \left[ \; \mathcal{E}^{(0)} \,\overset{a_i^{(1)}}{\subset} \, \mathcal{E}_i^{(1)} \overset{a_i^{(2)}}{\subset} \cdots\; \overset{a_{i}^{(N_i)}}{\subset} \, \mathcal{E}_i^{(N_i)}= \mathcal{E}^{(0)}(x_i) \;\right]_{i \in I}$
be a parabolic bundle over $C$. Let $S$ be a $k$-scheme and let $\mathcal{W} = \left[ \;\mathcal{F}^{(0)} \,\overset{b_i^{(1)}}{\subset} \, \mathcal{F}_i^{(1)} \overset{b_i^{(2)}}{\subset} \cdots\; \overset{b_{i}^{(M_i)}}{\subset} \,\mathcal{F}_i^{(M_i)}= \, \mathcal{F}^{(0)}(x_i) \;\right]_{i \in I}$ be a parabolic vector bundle over $C \times S$. We say that $\mathcal{W}$ is of type $\mathcal{V}$ if all of the following conditions are satisfied
\begin{enumerate}[(a)]
    \item $\text{rank}\, \mathcal{W} = \text{rank} \, \mathcal{V}$.
    \item For all $i \in I$, $N_i = M_i$ (they have the same chain lengths).
    \item For all $i \in I$ and $1 \leq m \leq N_i$, we have $a_i^{(m)} = b_i^{(m)}$.
\end{enumerate}
\end{defn}
\begin{remark}
The property of being equivalent to a given $\mathcal{V}$ is preserved by pullbacks.
\end{remark}
For the next two definitions, we fix a $k$-scheme $S$ and we let\\ \[\mathcal{V} = \left[ \; \mathcal{E}^{(0)} \,\overset{a_i^{(1)}}{\subset} \, \mathcal{E}_i^{(1)} \overset{a_i^{(2)}}{\subset} \cdots\; \overset{a_{i}^{(N_i)}}{\subset} \, \mathcal{E}_i^{(N_i)}= \mathcal{E}^{(0)}(x_i) \;\right]_{i \in I}\] 
\[\mathcal{W} = \left[ \;\mathcal{F}^{(0)} \,\overset{b_i^{(1)}}{\subset} \, \mathcal{F}_i^{(1)} \overset{b_i^{(2)}}{\subset} \cdots\; \overset{b_{i}^{(N_i)}}{\subset} \,\mathcal{F}_i^{(N_i)}= \, \mathcal{F}^{(0)}(x_i) \;\right]_{i \in I}\]
be two parabolic vector bundles over $C\times S$ with the same chain length $N_i$ at each $i \in I$.
\begin{defn} \label{defn: morphism of vb in family}
A morphism of parabolic Vector bundles $f : \mathcal{W} \rightarrow \mathcal{V}$ consists of of the data of a morphism of sheaves $f_i^{(m)}: \mathcal{F}_i^{(m)} \rightarrow \mathcal{E}_i^{(m)}$ for each $i \in I$ and $0 \leq m \leq N_i$ satisfying the following two conditions:
\begin{enumerate}[(a)]
\item $f_i^{(0)} = f_j^{(0)}$ for all $i, j \in I$.
\item For all $i \in I$ and $1 \leq m \leq N_i$, the following diagram commutes
\begin{figure}[H]
    \centering
\begin{tikzcd}
 \mathcal{E}_i^{(m-1)} \ar[r,  symbol= \subset] & \mathcal{E}_i^{(m)} \\
 \mathcal{F}_i^{(m-1)} \ar[r, symbol= \subset] \ar[u, "f_i^{(m-1)}"] & \mathcal{F}_i^{(m)} \ar[u, "f_i^{(m)}"]
\end{tikzcd}
\end{figure}
\end{enumerate}
\end{defn}

\begin{defn} \label{defn: parabolic subbundle in family}
Let $f: \mathcal{W} \rightarrow \mathcal{V}$ be a morphism of parabolic bundles.
\begin{enumerate}
\item We say that $\mathcal{V}$ is a parabolic quotient bundle of $\mathcal{W}$ (equivalently $f$ is an admissible epimorphism) if the morphism $f_i^{(m)}$ witnesses $\mathcal{E}_i^{(m)}$ as a quotient vector bundle of $\mathcal{F}_i^{(m)}$ for all $i \in I$ and $0 \leq m \leq N_i$.
\item We say that $\mathcal{W}$ is a parabolic subbundle of $\mathcal{V}$ (equivalently $f$ is an admissible monomorphism) if it satisfies the following two conditions:
    \begin{enumerate}[(a)]
        \item For each $i \in I$ and $0 \leq m \leq N_i$, we have $\mathcal{F}_i^{(m)} = \mathcal{E}_i^{(m)} \cap \mathcal{F}^{(0)}(x_i)$.
        \item The naive quotient (as a diagrams of quasicoherent sheaves)
        \[ \mathcal{Q} \vcentcolon = \left[ \mathcal{E}^{(0)} \, / \, \mathcal{F}^{(0)} \,\subset \, \mathcal{E}_i^{(1)} \, / \, \mathcal{F}_i^{(1)} \, \subset \cdots\; \subset \, \mathcal{E}_i^{(N_i)} \, / \, \mathcal{F}_i^{(N_i)} \,  = \, \left(\mathcal{E}^{(0)} \, / \, \mathcal{F}^{(0)}\right)(x_i)  \right]_{i \in I}\]
        is a parabolic vector bundle over $C \times S$.
    \end{enumerate}
\end{enumerate}
\end{defn}
We can use this definitions to impose the structure of an exact category on the class of parabolic vector bundles over $C\times S$ with a fixed set of chain lengths the structure. In this case we don't have kernels and images in general, because it is no longer true that a subsheaf of a locally free sheaf is locally free.

Let us fix a set of parabolic weights $\overline{\lambda}$ with chain length $N_i$ at $i \in I$. 
\begin{defn}
Let $S$ be a $k$-scheme. We write $\text{Vect}_{\overline{\lambda}}(S)$ to denote the exact category of parabolic vector bundles over $C \times S$ with chain length $N_i$ for each $i \in I$.
\end{defn}

\begin{remark}
Since the pullbacks of parabolic vector bundles preserve exact sequences, $\text{Vect}_{\overline{\lambda}}$ is a pseudofunctor from $k$-schemes into the 2-category of exact categories (with exact functors as 1-morphisms).
\end{remark}

\begin{lemma} \label{lemma: local constancy of deg}
Let $S$ be a scheme over $k$ and let $\mathcal{W}$ be a parabolic vector bundle in $\text{Vect}_{\overline{\lambda}}(S)$. Then, the function $f: S \longrightarrow \mathbb{R}$ given by $f(s) = \text{deg} \; \mathcal{W}\,|_{C_s}$ is locally constant in the Zariski topology.
\end{lemma}
\begin{proof}
Suppose that $\mathcal{W} = \left[ \; \mathcal{E}^{(0)} \,\overset{a_i^{(1)}}{\subset} \, \mathcal{E}_i^{(1)} \overset{a_i^{(2)}}{\subset} \cdots\; \overset{a_{i}^{(N_i)}}{\subset} \, \mathcal{E}_i^{(N_i)}= \mathcal{E}^{(0)}(x_i) \;\right]_{i \in I}$.
By the formula for the degree in Definition \ref{defn: deg of vector bundle}, we have
\[ f(s) = \text{deg} \; \mathcal{E}^{(0)}\, |_{C_s} + \sum_{i \in I} \left(n - \sum_{j=1}^{N_i} \lambda_i^{(j)} \, a_i^{(j)} \right) \]
The corollary follows, because the degree of the fibers of $\mathcal{E}^{(0)}$ is know to be a locally constant function.
\end{proof}

The following lemma tells us that the Harder-Narasimhan filtration behaves well under extension of scalars. The analogous result for classical vector bundles was proven in \cite{langton}. See also \cite{huybrechts.lehn}[Thm 1.3.7]. Our proof is a modification of the argument found there. 
\begin{lemma} \label{lemma: HN filtrations and scalar change}
Let $\mathcal{V}$ be a parabolic vector bundle over $C$ that belongs to $\text{Vect}_{\overline{\lambda}}$. Let $K$ be a field extension of $k$. If $\mathcal{V}$ is semistable, then $\mathcal{V}|_{C_K}$ is semistable.

In particular, the pullback of the Harder-Narasimhan filtration of $\mathcal{V}$ to $C_K$ is the Harder-Narasimhan filtration of $\mathcal{V}|_{C_K}$.
\end{lemma}
\begin{proof}
Let $\mathcal{V} = \left[ \; \mathcal{E}^{(0)} \,\overset{a_i^{(1)}}{\subset} \, \mathcal{E}_i^{(1)} \overset{a_i^{(2)}}{\subset} \cdots\; \overset{a_{i}^{(N_i)}}{\subset} \, \mathcal{E}_i^{(N_i)}= \mathcal{E}^{(0)}(x_i) \;\right]_{i \in I}$ be a semistable parabolic vector bundle over $C$ as in the statement.
Suppose that $\mathcal{V}|_{C_K}$ is not semistable. Let $\mathcal{W}_{K}$ be the maximal destabilizing parabolic subbundle as in Lemma \ref{lemma: existence maximal destibilizing subbundle}. It is clear by the formula in Definition \ref{defn: deg of vector bundle} that the slope of a parabolic vector bundle is preserved under extension of ground fields. We will try to descend $\mathcal{W}_K$ to a parabolic subbundle of $\mathcal{W}$ of $\mathcal{V}$ in order to contradict semistability of $\mathcal{V}$.

Suppose that $\mathcal{W}_K$ is given by $\mathcal{W}_{K} = \left[ \;\mathcal{F}^{(0)} \,\overset{b_i^{(1)}}{\subset} \, \mathcal{F}_i^{(1)} \overset{b_i^{(2)}}{\subset} \cdots\; \overset{b_{i}^{(N_i)}}{\subset} \,\mathcal{F}_i^{(N_i)}= \, \mathcal{F}^{(0)}(x_i) \;\right]_{i \in I}$. The definition of parabolic subbundle implies $\mathcal{F}_i^{(m)} = \mathcal{E}_i^{(m)}|_{C_K} \, \cap \, \mathcal{F}^{(0)}$. Hence $\mathcal{W}_K$ is completely determined by the vector subbundle $\mathcal{F}^{(0)} \subset \, \mathcal{E}^{(0)}|_{C_K}$. This vector subbundle gives us a $K$-point in $\text{Quot}_{\mathcal{E}^{(0)} / C / k}$ (see \cite{nitsure-quot} or Section \ref{section: parabolic Quot schemes} below). Since $\text{Quot}_{\mathcal{E}^{(0)} / C / k}$ is locally of finite type, this actually comes from a $K'$-point $\mathcal{F}^{(0)}_{K'} \subset \, \mathcal{E}^{(0)}|_{C_{K'}}$, where $K'$ is a subextension of $K \supset k$ that is finitely generated over $k$. Define $\mathcal{F}_{i, \, K'}^{(m)} \vcentcolon= \mathcal{F}^{(0)}_{K'} \, \cap \, \mathcal{E}_i^{(m)}|_{C_{K'}}$. 

Set $\mathcal{W}_{K'} = \left[ \;\mathcal{F}_{K'}^{(0)} \,\overset{b_i^{(1)}}{\subset} \, \mathcal{F}_{i, \, K'}^{(1)} \overset{b_i^{(2)}}{\subset} \cdots\; \overset{b_{i}^{(N_i)}}{\subset} \,\mathcal{F}_{i, \, K'}^{(N_i)}= \, \mathcal{F}_{K'}^{(0)}(x_i) \;\right]_{i \in I}$. This is a parabolic subbundle of $\mathcal{V}|_{C_{K'}}$, because $C_{K'}$-flatness of the subsheaves and the quotients can be checked fpqc-locally. $\mathcal{W}_{K'} \subset \mathcal{V}|_{C_{K'}}$ is the maximal parabolic subbundle with maximal slope. This is because there are more subbundles to check over $C_K$ than over $C_{K'}$. Hence we can assume without loss of generality that $K \supset k$ is a finitely generated field extension.

By enlarging the field if necessary, we can express $K \supset k$ as the composition of a finitely generated purely transcendental extension, followed by a finite Galois extension and a finite purely inseparable extension. We are thus reduced to analyze each of these three cases.
\begin{enumerate}[C 1]
    \item $K \supset k$ is a finite Galois extension
    Let $G = \text{Gal}(K\, / \, k)$. By the discussion in Appendix \ref{appendix: descent}, it suffices to show that for all $\sigma \in G$ we have $\sigma^*(\mathcal{W}_K) = \mathcal{W}_K$ as parabolic subbundles of $\mathcal{V}|_{C_K}$. From the formula for degree in Definition \ref{defn: deg of vector bundle}, it is easy to see that parabolic degree is preserved by elements of the Galois group. Since the rank is also preserved, $\sigma^*(\mathcal{W}_K)$ is still the unique maximal subbundle with maximal slope. Therefore $\mathcal{W}_K$ descends.
    \item $K = k\left(a^{\frac{1}{p}}\right)$ for some $a \in k$
    Let $\text{Der}$ be the set of $k$-derivations of $K$. For each $D \in \text{Der}$, we can extend $D$ to a $\mathcal{O}_{C}$-linear derivation $D_C$ of $\mathcal{O}_{C_K}$. This uniquely defines an $\mathcal{O}_{C}$-linear endomorphism of each $\mathcal{E}_i^{(m)}|_{C_K} = \mathcal{E}_i^{(m)} \otimes \mathcal{O}_{C_K}$ by acting on the second component of the tensor product. This endomorphism preserves the inclusion relations among the $\mathcal{E}_i^{(m)}$s, so it gives a $\mathcal{O}_C$-linear endomorphism $D_{\mathcal{V}}$ of $\mathcal{V}$.  This describes the canonical descent data of $\mathcal{V}|_{C_K}$ viewed as a diagram of quasicoherent sheaves, as explained in Appendix \ref{appendix: descent}. By the discussion in Appendix \ref{appendix: descent}, it follows that $\mathcal{W}_K \subset \mathcal{V}|_{C_K}$ descends to a parabolic subbundle of $\mathcal{V}$ if and only if $D_{\mathcal{V}}(\mathcal{W}_K) \subset \mathcal{W}_K$ for all $D \in \text{Der}$. This is equivalent to the composition $\phi: \mathcal{W}_K \xrightarrow{D_{\mathcal{V}}} \mathcal{V}|_{C_K} \twoheadrightarrow \mathcal{V}|_{C_K} \, / \, \mathcal{W}_K$ being $0$. But, by the Leibniz rule, this composition is actually $\mathcal{O}_{C_K}$-linear (because we are killing elements of $\mathcal{W}_K$). Hence $\phi$ is a genuine morphism of parabolic vector bundles. Since $\mathcal{W}_K$ is the maximal destabilizing sheaf, part (c) of Lemma \ref{lemma: existence maximal destibilizing subbundle} implies that $\text{Hom}\left(\mathcal{W}_K, \, \mathcal{V}|_{C_K} \, / \; \mathcal{W}_K\right) = 0$. Therefore we must have $\phi = 0$, as desired.
    \item $K = k(t)$
    The same trick with the $\text{Quot}$ scheme allows us to spread $\mathcal{W}_K$ to a parabolic subbundle $\mathcal{W}_U$ of $\mathcal{V}|_{C\times U}$, where $U$ is an open subset of $\mathbb{A}^1_{k}$. We can first extend as a set of chains of vector bundles. Then using opennes of the flat locus \cite[\href{https://stacks.math.columbia.edu/tag/0399}{Tag 0399}]{stacks-project}  we can decrease $U$ to ensure that all quotients in the chain are $U$-flat. Let $t \in U$ be a closed point of $U$. Consider the parabolic subbundle $\mathcal{W}_U|_{C_t}$ of $\mathcal{V}|_{C_t}$. By local constancy of degree (Lemma \ref{lemma: local constancy of deg}), $\mathcal{W}_U|_{C_t} \subset \mathcal{V}|_{C_t}$ implies that $\mathcal{V}|_{C_t}$ is not semistable.  Since the residue field $\kappa(t)$ is finite over $k$, we are reduced to the case of finite field extensions as in Cases 1,2 above.
\end{enumerate}
For the second part of the proposition, suppose that $0 = \mathcal{W}_0 \subset \mathcal{W}_1 \subset \cdots \, \subset \mathcal{W}_{l-1} \subset \mathcal{W}_l = \mathcal{V}$ is the Harder-Narasimhan filtration of $\mathcal{V}$. The first part of the proposition implies that $\mathcal{W}_j|_{C_K} \, / \, \mathcal{W}_{j-1}|_{C_K}$ is semistable for all $1 \leq j \leq l$. Since slopes are preserved by field extensions, we conclude that $0 = \mathcal{W}_0|_{C_K} \subset \mathcal{W}_1|_{C_K} \subset \cdots \, \subset \mathcal{W}_{l-1}|_{C_K} \subset \mathcal{W}_l|_{C_K} = \mathcal{V}|_{C_K}$ is a Harder-Narasimhan filtration of $\mathcal{V}|_{C_K}$.
\end{proof}

Next, we show that Harder-Narasimhan filtrations have no nontrivial first order deformations. We denote by $k[\epsilon] \vcentcolon = k[T] \, / \, (T^2)$ the ring of dual numbers. We denote by $\sigma: \text{Spec}(k) \rightarrow \text{Spec} \;k[\epsilon]$ the unique section of the structure map $p: \text{Spec} \;k[\epsilon] \rightarrow \text{Spec}(k)$.
\begin{prop} \label{prop: rigidity of HN filtration}
Let $\mathcal{V}$ be a parabolic vector bundle over $C$. Suppose that we are given a filtration $0 = \mathcal{W}_0 \subset \mathcal{W}_1 \subset \cdots \, \subset \mathcal{W}_{l-1} \subset \mathcal{W}_l = p^{*} \mathcal{V}$ of $p^{*} \mathcal{V}$ by parabolic subbundles.
Assume that the pulled-back filtration 
$0 = \sigma^{*}\mathcal{W}_0 \subset \sigma^{*}\mathcal{W}_1 \subset \cdots \, \subset \sigma^{*}\mathcal{W}_{l-1} \subset \sigma^{*}\mathcal{W}_l = \mathcal{V}$ over the closed point
is the Harder-Narasimhan filtration of $\mathcal{V}$. Then $\mathcal{W}_j = p^{*} \sigma^{*} \mathcal{W}_j$ for all $1 \leq j \leq l-1$.
\end{prop}
\begin{proof}
By induction on the length $l$ of the filtration, we are reduced to showing that $\mathcal{W}_1 = p^{*} \sigma^{*}\, \mathcal{W}_1$. Let us show that $\mathcal{W}_1 \subset p^{*} \sigma^{*}\, \mathcal{W}_1$. Observe that $p^{*} \sigma^{*}\, \mathcal{W}_1$ is a parabolic subbundle of $p^*\mathcal{V}$. It suffices to prove that the composition $f: \, \mathcal{W}_1 \, \hookrightarrow \, p^{*} \mathcal{V} \, \twoheadrightarrow \, p^{*}\mathcal{V} \, / \, p^{*}\sigma^{*} \, \mathcal{W}_1$
is $0$. Restricting to the closed point gives us the map $\sigma^{*}f: \, \sigma^{*}\mathcal{W}_1 \, \hookrightarrow \, \mathcal{V} \, \twoheadrightarrow \, \mathcal{V} \, / \, \sigma^{*} \, \mathcal{W}_1$, which is clearly $0$. Therefore, basic deformation theory tells us that $f = \epsilon \, \psi$ for some homomorphism of parabolic bundles $\psi \in \text{Hom}\left(\,\sigma^{*}\mathcal{W}_1, \; \mathcal{V} \, / \, \sigma^{*}\mathcal{W}_1 \,\right)$. But note that, by assumption, $\sigma^{*}\mathcal{W}_1$ is the maximal destabilizing parabolic subbundle of $\mathcal{V}$. Hence part (c) of Lemma \ref{lemma: existence maximal destibilizing subbundle} implies that $\text{Hom}\left(\,\sigma^{*}\mathcal{W}_1, \; \mathcal{V} \, / \, \sigma^{*}\mathcal{W}_1 \,\right) = 0$. We conclude that $f = 0$.

The same reasoning shows that $p^{*} \sigma^{*}\, \mathcal{W}_1 \subset \mathcal{W}_1$, thus concluding the proof.
\end{proof}
\end{subsection}
\end{section}

\begin{section}{Stacks of parabolic vector bundles} \label{section: stacks of parabolic vector bundles}
\begin{subsection}{The group scheme $\text{GL}\left(\mathcal{V}\right)$}
Let $J$ be a small category and let $X$ be a scheme. A diagram of sheaves on $X$ of shape $J$ is a functor $F$ from $J$ into the abelian category of quasicoherent sheaves $\text{QCoh}(X)$. Diagrams of sheaves on $X$ of shape $J$ form an abelian category, where the morphisms are natural transformations.  Given a morphism of schemes $f: Y \rightarrow X$ and a diagram of sheaves $F: J \rightarrow \text{QCoh}(X)$, the pullback $f^{*}F$ is defined to be the diagram of sheaves on $Y$ obtained by postcomposing $F$ with the functor $f^{*}: \text{QCoh}(X) \rightarrow \text{QCoh}(Y)$. As usual, we will write $F|_{Y}$ instead of $f^{*}F$ when the morphism $f$ is implicitly understood.

Let $\mathcal{V} = \left[ \; \mathcal{E}^{(0)} \,\overset{a_i^{(1)}}{\subset} \, \mathcal{E}_i^{(1)} \overset{a_i^{(2)}}{\subset} \cdots\; \overset{a_{i}^{(N_i)}}{\subset} \, \mathcal{E}_i^{(N_i)}= \mathcal{E}^{(0)}(x_i) \;\right]_{i \in I}$ be a parabolic vector bundle over $C$. We can interpret $\mathcal{V}$ as a diagram of sheaves on $C$. The source category $J_{\mathcal{V}}$ can be described as follows. $J_{\mathcal{V}}$ has $1+ \sum_{i \in I} N_i$ objects. There is a initial object $X^{(0)}$, and the rest of the objects can be labeled $X_i^{(m)}$ for $i \in I$ and $1 \leq m \leq N_i$. For each $i \in I$ and $1 \leq m_1 \leq m_2 \leq N_i$, there is a unique morphism $X_i^{(m_1)} \rightarrow X_i^{(m_2)}$. These, along with the initial morphims with source $X^{(0)}$, account for all morphisms in the category. Here is a picture of $J_{\mathcal{V}}$ (we fix an enumeration $I =\{i_1, i_2, \cdots, i_k\}$ for convenience):
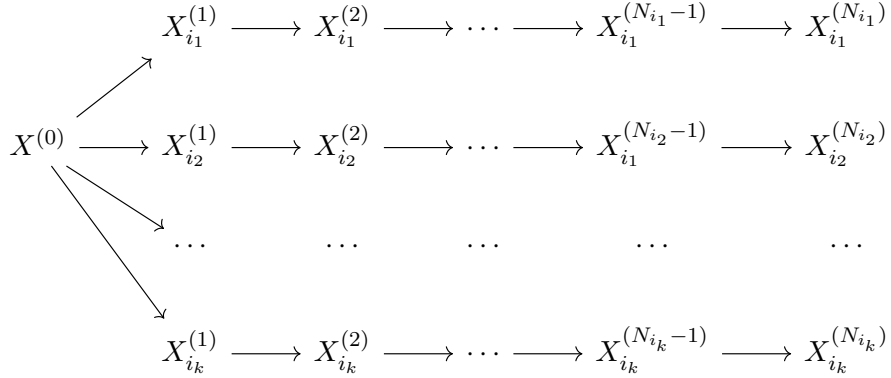
\begin{figure}[H]
\centering
\begin{tikzcd}
      & X_{i_1}^{(1)}\ar[r] & X_{i_1}^{(2)}\ar[r] & \cdots\ar[r] & X_{i_1}^{(N_{i_1}-1)} \ar[r]& X_{i_1}^{(N_{i_1})}\\
     X^{(0)} \ar[r] \ar[ur] \ar[dr] \ar[ddr] & X_{i_2}^{(1)}\ar[r] & X_{i_2}^{(2)}\ar[r] & \cdots\ar[r] & X_{i_1}^{(N_{i_2}-1)}\ar[r] & X_{i_2}^{(N_{i_2})}\\
     & \cdots & \cdots & \cdots & \cdots & \cdots\\
   & X_{i_k}^{(1)} \ar[r] & X_{i_k}^{(2)}\ar[r] & \cdots\ar[r] & X_{i_k}^{(N_{i_k}-1)}\ar[r] & X_{i_k}^{(N_{i_k})}
\end{tikzcd}
\end{figure}
This description is compatible with some of the definitions given in Section \ref{section: parabolic vector bundles curve}. For example, morphisms of parabolic vector bundles are the same as morphisms of diagrams of sheaves. Fixing the chain lengths amounts to fixing the shape $J$ of the diagram.

\begin{defn} \label{defn: group schemes of automorphisms}
Let $\mathcal{V}$ a parabolic vector bundle as above. Denote by $\text{GL}(\mathcal{V})$ the contravariant functor from the category of flat $C$-schemes into Set given as follows. For $f: T \rightarrow C$ a flat $C$-scheme, we set $\text{GL}(\mathcal{V})\, (T) \vcentcolon =  \text{Aut}( f^{*} \, \mathcal{V})$. Here the pullbacks and automorphisms are to be interpreted in the category of diagrams of sheaves of shape $J_{\mathcal{V}}$, as described in the paragraphs above.
\end{defn}

\begin{prop} \label{prop: groups schemes of automorphisms is smooth}
$\text{GL}(\mathcal{V})$ is represented by a smooth group scheme over $C$.
\end{prop}
\begin{proof}
By descent for morphisms of quasicoherent sheaves, we know that $\text{GL}(\mathcal{V})$ is a sheaf in the Zariski topology. In particular, it suffices to check the claim after passing to a Zariski cover of $C$. Let $\bigsqcup_{j \in J} U_j \rightarrow C$ be an open cover of $C$ such that for all $j \in J$ we have that $\mathcal{E}^{(0)}|_{U_j} \cong \mathcal{O}_{U_j}^{\oplus n}$. It suffices to show that the automorphism functor $\text{GL}(\mathcal{V}|_{U_j})$ of the diagram of sheaves $\mathcal{V}|_{U_j}$ is a smooth group scheme over $U_j$ for any $j \in J$. If $U_j$ does not contain any of the points of degeneration $\{x_i\}_{i \in I}$, then we have
\[ \text{GL}(\mathcal{V}|_{U_j}) \, \cong \, \text{GL}(\mathcal{E}^{(0)}|_{U_j}) \, \cong\, \text{GL}(\mathcal{O}_{U_j}^{\oplus n}) \, \cong \, \text{GL}_n \times U_j  ,\]
where $n = \text{rank} \, \mathcal{V}$. (Here we are abusing notation and writing $\mathcal{E}^{(0)}$ to denote the constant diagram with identities as morphisms). Hence we are done in this case.

Suppose that $U_j$ contains some of the points of degeneration $\{x_i\}_{i \in I}$. By passing to a finer Zariski cover, we can assume that $U_j$ contains only one such point of degeneration. Fix $i \in I$ such that $x_i \in U_j$. For all indexes $s \neq i$, the inclusions in the $s$th chain of vector bundles $\mathcal{E}^{(0)}|_{U_j} \,\overset{a_s^{(1)}}{\subset} \, \mathcal{E}_s^{(1)}|_{U_j} \overset{a_s^{(2)}}{\subset} \cdots\; \overset{a_{s}^{(N_s)}}{\subset} \, \mathcal{E}_s^{(N_s)}|_{U_j}= \mathcal{E}^{(0)}(x_s)|_{U_j}$ become isomorphisms. We conclude that the functor of automorphisms of the diagram $\mathcal{V}|_{U_j}$ is the same of the functor of automorphisms preserving the single chain of vector bundles \[\mathcal{E}^{(0)}|_{U_j} \,\overset{a_i^{(1)}}{\subset} \, \mathcal{E}_i^{(1)}|_{U_j} \overset{a_i^{(2)}}{\subset} \cdots\; \overset{a_{i}^{(N_i)}}{\subset} \, \mathcal{E}_i^{(N_s)}|_{U_j}= \mathcal{E}^{(0)}(x_i)|_{U_j}.\]

By making $U_j$ smaller, we can assume that $U_j$ is the spectrum of a Dedekind domain $R$ over $k$ and that the Cartier divisor $x_i$ is given by the vanishing of a nonzerodivisor $t \in R$. We can also assume that $\mathcal{E}_i^{(m)}|_{U_j} \cong \mathcal{O}_{U_j}^{\oplus n}$ for all $1 \leq m \leq N_i$. Now $\mathcal{E}^{(0)}(x_i)|_{U_j}$ is a free $R$-module of rank $n$, let's denote it by $F$. Similarly, for $1 \leq m \leq N_i -1$, let's write $F_m$ for the free $R$-module $\mathcal{E}^{(m)}_i|_{U_j}$. Our chain of vector bundles is given by a set of inclusions $t \,F \, \subset \, F_1 \, \subset \, \cdots \,\subset \, F_{N_i-1} \, \subset \, F$.

Observe that all maps in this chain become isomorphisms after inverting $t$. Hence we get a collection of canonical isomorphisms $F_l|_{U_j \setminus x_i} \xrightarrow{\sim} F|_{U_j \setminus x_i}$ for each $1 \leq l \leq N_i -1$. Consider the smooth group scheme $G$ over $R$ given by $G \vcentcolon = \text{GL}(t \, F) \times \,\prod_{l=1}^{N_i -1}\text{GL}(F_l) \, \times \text{GL}(F)$. We have a diagonal morphism $\overline{\Delta}: \text{GL}(F)|_{U_j \setminus x_i} \rightarrow G|_{U_j \setminus x_i}$ obtained by using the aforementioned canonical isomorphisms $F_l|_{U_j \setminus x_i} \xrightarrow{\sim} F|_{U_j \setminus x_i}$. To be precise, $\overline{\Delta}$ is given by the composition
\begin{equation*} 
\begin{aligned}
\text{GL}(F)|_{U_j \setminus x_i} \, \xrightarrow{\Delta} \,& \; \text{GL}(F)|_{U_j \setminus x_i} \times \,\prod_{l=1}^{N_i -1}\text{GL}(F)|_{U_j \setminus x_i} \, \times \text{GL}(F)|_{U_j \setminus x_i}\\ 
\xrightarrow{\sim} &  \; \text{GL}(t \, F)|_{U_j \setminus x_i} \times \,\prod_{l=1}^{N_i -1}\text{GL}(F_l)|_{U_j \setminus x_i} \, \times \text{GL}(F)|_{U_j \setminus x_i}
\end{aligned}
\end{equation*}

Let $S$ be a flat $R$-algebra. Automorphisms of the diagram of sheaves $\mathcal{V}|_{\text{Spec}(S)}$ are the same as automorphisms the chain of free modules $t \,F \otimes S\, \subset \, F_1 \otimes S \, \subset \, \cdots \,\subset \, F_{N_i-1} \otimes S \, \subset \, F \otimes S$. By $R$-flatness of $S$, we know that $t$ is still a nonzerodivisor in $S$. So the map $S \rightarrow S[\frac{1}{t}]$ is injective. By unraveling the definition, an automorphism of the chain of modules above is an element of $G(S)$ such that its localization in $G(S[\frac{1}{t}]) = G|_{U_j \setminus x_i}(S[\frac{1}{t}])$ belongs to the image of the diagonal map $\overline{\Delta}(S[\frac{1}{t}])$. These are exactly the $S$-points of the scheme theoretic closure of the composition $D: \, \text{GL}(F)|_{U_j \setminus x_i} \, \xrightarrow{\Delta} \, G|_{U_j \setminus x_i} \, \hookrightarrow \, G$. This scheme theoretic closure $\text{Im}(D)$ must be $R$-flat, because it is the Spec of a subalgebra of the torsion-free $R$-module $\mathcal{O}_{\text{GL}(F)|_{U_j \setminus x_i}}$. We conclude that the functor $\text{GL}(\mathcal{V})|_{U_j}$ is represented by an affine flat $U_j$-scheme. We are left to show smoothness.

$\text{Im}(D)$ is of finite type, because it is a closed subscheme of $G$. Also, we have seen that $\text{GL}(\mathcal{V})|_{U_j \setminus x_i} \, \cong \, \text{GL}_n \times (U_j \setminus x_i)$. We are left to show that the fiber $\text{Im}(D)|_{x_i}$ is smooth. Since taking scheme theoretic image of quasicompact maps commmutes with flat base-change, we can pass to the completion of the local ring at $x_i$ without altering the fiber $\text{Im}(D)|_{x_i}$. Therefore we can assume that $R$ is a complete discrete valuation with residue field $k$. By the Cohen structure theorem, we have $R = k\bseries{t}$.





Now we have put ourselves in a familiar situation, since the scheme theoretic closure is by definition the parahoric group scheme of $\text{GL}_n$ defined over $k\bseries{t}$ that preserves the chain of lattices $t\, F \, \subset \, F_1 \, \subset \, \cdots \, \subset \, F_{N_i-1} \, \, \subset \, F$. This group scheme is known to be smooth over $R$, see \cite[\S 3]{bruhat-tits}.
\end{proof}

\begin{remark}
We have seen in the proof above that $\text{GL}(\mathcal{V})$ is generically isomorphic to the constant group scheme $\text{GL}_n$ (via choosing a trivialization of $\mathcal{E}^{(0)}$ at the generic point). At the finitely many points of degeneration $\{x_i\}_{i \in I}$, the base-change to the completion of the local ring $\text{GL}(\mathcal{V})|_{\widehat{\mathcal{O}}_{C, \,  x_i}}$ is a parahoric group scheme of $\text{GL}_n$ in the sense of Bruhat-Tits \cite[Def. 5.2.6]{bruhat-tits-II-big}. $\text{GL}(\mathcal{V})$ is a special case of what is sometimes called a parahoric Bruhat-Tits group scheme in the literature, see  \cite{pappas-rapoport-moduli}, \cite{balaji-seshadri}, \cite{heinloth-uniformization}.
\end{remark}

We briefly recall the definition of the moduli stack of torsors for a group scheme such as $\text{GL}(\mathcal{V})$ above. Let $X$ be a scheme. Let $\mathcal{G} \rightarrow X$ be a relatively affine group scheme that is smooth over $X$. Let $\mathcal{P}$ be a $X$-scheme. A right $\mathcal{G}$-action on $\mathcal{P}$ consists of a map of $X$- schemes $a: \, \mathcal{P} \times_X \mathcal{G} \rightarrow \mathcal{P}$ such that for all $X$-schemes $T$ the map induced on $T$-points $a(T) : \, \mathcal{P}(T) \times \mathcal{G}(T) \rightarrow \mathcal{P}(T)$ is a right action of $\mathcal{G}(T)$ on the set $\mathcal{P}(T)$. A (right) $\mathcal{G}$-torsor over $X$ is a scheme $\pi: \mathcal{P} \rightarrow X$ together with a right $\mathcal{G}$-action that makes $\mathcal{P}$ a $\mathcal{G}$-torsor in the \`etale topology \cite[\href{https://stacks.math.columbia.edu/tag/0497}{Tag 0497}]{stacks-project}. This means that there is an \`etale cover $T \rightarrow X$ such that the base-change $\pi\times id_T: \mathcal{P}\times_X T \rightarrow T$ is $\mathcal{G}$-equivariantly isomorphic to the trivial $\mathcal{G}$-torsor $\mathcal{G} \times_X T$. An isomorphism between $\mathcal{G}$-torsors is an isomorphism of $X$-schemes that intertwines the $\mathcal{G}$-actions.

Let $\mathcal{G}$ be a relatively affine smooth group scheme over $C$.
\begin{defn} \label{defn: moduli of torsors smooth group scheme}
The moduli stack of $\mathcal{G}$-torsors $\text{Bun}_{\mathcal{G}}(C)$ is defined to be the pseudofunctor from $k$-schemes into groupoids given as follows. For every $k$-scheme $T$, we let
\[ \text{Bun}_{\mathcal{G}}(C)\, (T) \vcentcolon = \; \left\{ \begin{matrix} \text{groupoid of  $\mathcal{G} \times T$-torsors over $C\times T$} \end{matrix} \right\}  \]
\end{defn}
See Proposition 1 in \cite{heinloth-uniformization} for a proof of the following using the Artin criteria.
\begin{prop} \label{prop: algebraicity torsors over smooth group schemes}
$\text{Bun}_{\mathcal{G}}$ is a smooth algebraic stack over $k$. \qed
\end{prop}
\end{subsection}

\begin{subsection}{Moduli stacks of parabolic vector bundles}
Fix a parabolic vector bundle $\mathcal{V}$ over $C$. Let us define the stack of parabolic vector bundles of type $\mathcal{V}$.
\begin{defn} \label{defn: stack of parabolic vector bundles}
We define the pseudofunctor $\text{Bun}_{\mathcal{V}}$ from $k$-schemes into groupoids as follows. For any $k$-scheme $S$, we set
\[ \text{Bun}_{\mathcal{V}}(S) \; \vcentcolon = \; \left\{ \begin{matrix} \text{groupoid of  parabolic vector bundles $\mathcal{W}$}  \\   \text{of type $\mathcal{V}$ over $C \times S$} \end{matrix} \right\}  \]
For any morphism $f: S \rightarrow T$ of $k$-schemes, we set $\text{Bun}_{\mathcal{V}} (f)$ to be the pullback $f^{*}$ of parabolic vector bundles.
\end{defn}

For the next proposition, we write $\text{Bun}_{\text{GL}(\mathcal{V})}(C)$ to denote the moduli stack of torsors of $\text{GL}(\mathcal{V})$ over $C$. See Definition \ref{defn: moduli of torsors smooth group scheme} and \cite{heinloth-uniformization} Example (2) pg. 500.
\begin{prop} \label{prop: parabolic vb as parahoric torsors}
There is an equivalence of pseudofunctors $\text{Bun}_{\text{GL}(\mathcal{V})}(C) \xrightarrow{\sim} \text{Bun}_{\mathcal{V}}$.
\end{prop}
\begin{proof}
Suppose that $\mathcal{V} = \left[ \; \mathcal{E}^{(0)} \,\overset{a_i^{(1)}}{\subset} \, \mathcal{E}_i^{(1)} \overset{a_i^{(2)}}{\subset} \cdots\; \overset{a_{i}^{(N_i)}}{\subset} \, \mathcal{E}_i^{(N_i)}= \mathcal{E}^{(0)}(x_i) \;\right]_{i \in I}$. Recall that the group scheme $\text{GL}(\mathcal{V})$ acts on $\mathcal{V}$ by automorphisms of diagrams of sheaves. Let $T$ be a $k$-scheme. Let 
$\mathcal{P}$ be a $\text{GL}(\mathcal{V})\times T$-torsor over $C\times T$. We can use the natural action of $\text{GL}(\mathcal{V})$ on $\mathcal{V}$ in order to define an associated parabolic vector bundle $\mathcal{P} \times_{\text{GL}(\mathcal{V}) \times T} \, \mathcal{V}|_{C \times T}$ on $C \times T$. Let us describe how this is defined. Let $f: S \rightarrow C\times T$ be an \`etale cover of $C$ such that $f^{*} \mathcal{P}$ is isomorphic to the trivial torsor. Choosing such an isomorphims, the descent datum for $\mathcal{P}$ is given by an element $h \in \text{GL}(\mathcal{V})(S \times_{C\times T} S)$. Here $h$ is viewed as an automorphism of the trivial torsor $\text{GL}(\mathcal{V})|_S$ via left multiplication. Now $h$ gives us an automorphims of $\mathcal{V}|_{S \times_{C \times T} S}$, which we can interpret as a descent datum for the diagram of sheaves  $\mathcal{V}|_{S}$ relative to the cover $f: S \rightarrow C \times T$. We can use effectivity of \`etale descent for quasicoherent sheaves in order to obtain a diagram of sheaves $\mathcal{P} \times_{\text{GL}(\mathcal{V}) \times T} \, \mathcal{V}|_{C \times T}$ such that its canonical descent datum is given by $h$. Recall that being a vector bundle of a specified degree is a property of quasicoherent sheaves that can be checked after passing to a flat cover. We can use this fact to conclude that $\mathcal{P} \times_{\text{GL}(\mathcal{V}) \times T} \, \mathcal{V}|_{C \times T}$ is a parabolic vector bundle of type $\mathcal{V}$ over $C \times T$.

Let us now define a functor $F: \text{Bun}_{\text{GL}(\mathcal{V})}(C) \rightarrow \text{Bun}_{\mathcal{V}}$ as follows. For every $k$ scheme $T$ and for every $\text{GL}(\mathcal{V}) \times T$-torsor $\mathcal{P}$ on $C\times T$, we set $F(T) \, (\mathcal{P}) \vcentcolon = \mathcal{P} \times_{\text{GL}(\mathcal{V}) \times T} \mathcal{V}|_{C\times T}$.

We want to define a quasi-inverse functor $G$. Let $T$ be a $k$-scheme. Suppose that we are given a parabolic vector bundle $\mathcal{W}$ of type $\mathcal{V}$ over $C \times T$. We can define a presheaf of sets $\text{Iso}(\mathcal{V}, \, \mathcal{W})$ on the category of flat $C \times T$-schemes as follows. For every scheme $f: S \rightarrow C \times T$ that is flat over $C \times T$, we define $\text{Iso}(\mathcal{V}, \,\mathcal{W}) \, (S) \vcentcolon = \text{Iso} (f^{*} (\mathcal{V}|_{C \times T}), \, f^{*} \mathcal{W}) $. Here the isomorphisms are taken in the category of diagrams of sheaves on $S$. It is clear by definition that $\text{Iso}(\mathcal{V}, \, \mathcal{W})$ admits a right action of $\text{GL}(\mathcal{V})$ given by precomposition.

We claim that $\text{Iso}(\mathcal{V}, \, \mathcal{W})$ is in fact a $\text{GL}(\mathcal{V})\times T$-torsor over $C \times T$. By \`etale descent for morphisms of quasicoherent sheaves, we know that $\text{Iso}(\mathcal{V}, \, \mathcal{W})$ is a sheaf in the \`etale topology. Therefore, in order to prove the claim it suffices to show that there is an \`etale cover $f: S \rightarrow C \times T$ such that $f^{*} \mathcal{W} \cong f^{*} (\mathcal{V}|_{C \times T})$ as diagrams of sheaves.

Say $\mathcal{W} = \left[ \;\mathcal{F}^{(0)} \,\overset{a_i^{(1)}}{\subset} \, \mathcal{F}_i^{(1)} \overset{a_i^{(2)}}{\subset} \cdots\; \overset{a_{i}^{(N_i)}}{\subset} \,\mathcal{F}_i^{(N_i)}= \, \mathcal{F}^{(0)}(x_i) \;\right]_{i \in I}$. Let $\bigsqcup_{j \in J} U_j \rightarrow C\times T$ be an open cover of $C\times T$ such that for all $j \in J$ we have $\mathcal{E}^{(0)}|_{U_j} \cong \mathcal{O}_U^{\oplus n} \cong \mathcal{F}^{(0)}|_{U_j}$. By refining the cover, we can furthermore assume that each $U_j$ intersects at most one of the divisors $x_i \times T$ of $C \times T$.

Let's focus our attention on a single open subset $U_j$. If $U_j$ does not intersect any of the divisors $x_i \times T$, then we have a chain of isomorphisms 
\[\mathcal{V}|_{U_j} \cong \mathcal{E}^{(0)}|_{U_j} \cong \mathcal{O}_{U_j}^{\oplus n} \cong \mathcal{F}_i^{(0)}|_{U_j} \cong \mathcal{W}|_{U_j}\]
(Here we are abusing notation for the intermediate terms; they are constant diagrams of vector bundles with identities as morphisms). Hence we are done in this case.

Suppose now that $U_j$ intersects one of the divisors $x_i\times T$. By construction, it only intersects one such divisor. Fix $i \in I$ such that $U_j \, \cap \, (x_i \times T) \neq \emptyset$. Let us denote by $q_i: x_i \rightarrow C$ the closed immersion of the point $x_i$ into $C$. Let $\pi: U_j \rightarrow T$ be the structure morphism. Let $W$ denote the open subset of $T$ given by the inverse image of $U_j$ under the section $q_i \times Id_T: \, T \rightarrow C \times T$. Define $V \vcentcolon = \pi^{-1}(W)$. We can cover $U_j = V \, \bigcup \, (U_j \setminus x_i)$. Since $U_j \setminus x_i$ does not intersect any of the divisors of degeneration, we have seen above that 
\[\mathcal{V}|_{U_j \setminus x_i} \cong \mathcal{E}^{(0)}|_{U_j \setminus x_i} \cong \mathcal{O}_{U_j \setminus x_i}^{\oplus n} \cong \mathcal{F}_i^{(0)}|_{U_j \setminus x_i} \cong \mathcal{W}|_{U_j \setminus x_i}\]
It remains to check the claim on $V$. Let $D \hookrightarrow V$ be the divisor given by the intersection $V \, \cap \, (x_i \times T)$. By construction, we have a structure morphims $\pi: V \rightarrow W$ and $D$ is given by a section $s: W \rightarrow V$ of $\pi$. Since $V$ does not intersect any of the other divisors of degeneration, for all $u \neq i$ the chain of inclusions with index $u$ gives us canonical identifications $\mathcal{E}_u^{(m)}|_{V} \cong \mathcal{E}^{(0)}|_{V}$ for all $m$.  Similarly, there are canonical identifications $\mathcal{F}^{(m)}_u|_{V} \cong \mathcal{F}^{(0)}|_{V}$ for $u \neq i$. By choosing isomorphisms $\mathcal{E}^{(0)}|_{V} \cong \mathcal{F}^{(0)}|_{V} \cong \mathcal{O}_{V}^{\oplus n}$ we get compatible isomorphisms $\mathcal{E}_u^{(m)}|_{V} \cong \mathcal{F}_u^{(m)}|_{V} \cong \mathcal{O}_{V}^{\oplus n}$ for all indexes $u \neq i$ and $m$ via the aforementioned canonical identifications. Under these identifications, we see that $\mathcal{V}|_{V}$ and $\mathcal{W}|_{V}$ are determined by the following two chains of vector bundles on $V$ respectively:
\[  \mathcal{O}_{V}^{\oplus n}\,  \overset{a_i^{(1)}}{\subset} \, \mathcal{E}_i^{(1)}|_{V} \overset{a_i^{(2)}}{\subset} \cdots\; \overset{a_{i}^{(N_i)}}{\subset} \mathcal{O}_{V}(D)^{\oplus n} \]
\[ \mathcal{O}_{V}^{\oplus n}\,  \overset{a_i^{(1)}}{\subset} \, \mathcal{F}_i^{(1)}|_{V} \overset{a_i^{(2)}}{\subset} \cdots\; \overset{a_{i}^{(N_i)}}{\subset} \mathcal{O}_{V}(D)^{\oplus n} \]
The claim is reduced to the problem of finding (maybe after passing to an \`etale cover of $V$) an automorphism $\text{GL}_n(V)$ of $\mathcal{O}_{V}^{\oplus n}$ that sends one chain to the other. Consider the corresponding flags of vector bundles on $W$
\[ 0\,  \overset{a_i^{(1)}}{\subset} \; s^{*}\,\left(\mathcal{E}_i^{(1)}|_{V} \, / \,  \mathcal{O}_{V}^{\oplus n}\right) \; \overset{a_i^{(2)}}{\subset} \cdots\; \overset{a_{i}^{(N_i)}}{\subset} \; s^{*} \, \left(\mathcal{O}_{V}(D)^{\oplus n} \, / \,  \mathcal{O}_{V}^{\oplus n}\right) \]
\[ 0\,  \overset{a_i^{(1)}}{\subset} \; s^{*}\,\left(\mathcal{F}_i^{(1)}|_{V} \, / \,  \mathcal{O}_{M_l}^{\oplus n}\right) \; \overset{a_i^{(2)}}{\subset} \cdots\; \overset{a_{i}^{(N_i)}}{\subset} \; s^{*} \, \left(\mathcal{O}_{V}(D)^{\oplus n} \, / \,  \mathcal{O}_{V}^{\oplus n}\right) \]
These two flags represent two $W$-points of the partial flag variety $\text{GL}_{n} \, / \, P$, where $P$ is the standard parabolic corresponding to the partition $(a_i^{(m)})_{m}$ of $n$. Let us denote these points by $p_1, p_2 \in (\text{GL}_n \, / \, P) (W)$. We know that the quotient $\text{GL}_n \, / \, P$ represents the \`etale sheafification of the naive coset functor. So there is an \`etale cover $w: Y \rightarrow W$ and an element $g \in \text{GL}_n(Y)$ such that $g \cdot w^{*}p_1 = w^{*}p_2$. In other words, the element $g$ relates the following two flags:
\[ 0\,  \overset{a_i^{(1)}}{\subset} \; w^{*} \, s^{*}\,\left(\mathcal{E}_i^{(1)}|_{V} \, / \,  \mathcal{O}_{V}^{\oplus n}\right) \; \overset{a_i^{(2)}}{\subset} \cdots\; \overset{a_{i}^{(N_i)}}{\subset} \; w^{*} \, s^{*} \, \left(\mathcal{O}_{V}(D)^{\oplus n} \, / \,  \mathcal{O}_{V}^{\oplus n}\right) \]
\[ 0\,  \overset{a_i^{(1)}}{\subset} \; w^{*} \, s^{*}\,\left(\mathcal{F}_i^{(1)}|_{V} \, / \,  \mathcal{O}_{V}^{\oplus n}\right) \; \overset{a_i^{(2)}}{\subset} \cdots\; \overset{a_{i}^{(N_i)}}{\subset} \; w^{*} \, s^{*} \, \left(\mathcal{O}_{V}(D)^{\oplus n} \, / \,  \mathcal{O}_{V}^{\oplus n}\right) \]
The base-change $pr_2:Y \times_{W} V \rightarrow V$ is an \`etale cover. We have a structure morphism $pr_1: Y \times_{W} V \rightarrow Y$. By construction, the automorphism $pr_1^{*} \,g$ of $\mathcal{O}_{Y \times_{W} V}^{\oplus n}$ relates the two chains of vector bundles over $Y \times_{W} V$
\[ \mathcal{O}_{Y \times_{W} V}^{\oplus n}\,  \overset{a_i^{(1)}}{\subset} \, \mathcal{E}_i^{(1)}|_{Y \times_{W} V} \overset{a_i^{(2)}}{\subset} \cdots\; \overset{a_{i}^{(N_i)}}{\subset} \mathcal{O}_{Y \times_{W} V}(D)^{\oplus n}  \]
\[ \mathcal{O}_{Y \times_{W} V}^{\oplus n}\,  \overset{a_i^{(1)}}{\subset} \, \mathcal{F}_i^{(1)}|_{Y \times_{W} V} \overset{a_i^{(2)}}{\subset} \cdots\; \overset{a_{i}^{(N_i)}}{\subset} \mathcal{O}_{Y \times_{W} V}(D)^{\oplus n} \]
We conclude that $\mathcal{W}|_{Y \times_{W} V} \cong \mathcal{V}|_{Y \times_{W} V}$, and the claim is proven.

We can now define a quasi-inverse to $F$. We define $G : \text{Bun}_{\mathcal{V}} \rightarrow  \text{Bun}_{\text{GL}(\mathcal{V})}(C)$ as follows. For every $k$-scheme $T$ and every parabolic vector bundle $\mathcal{W}$ of type $\mathcal{V}$ over $C \times T$, we let $G(T) \, (\mathcal{W}) \vcentcolon =  \text{Iso}(\mathcal{V}, \,\mathcal{W})$.

Let us prove that $F$ and $G$ are quasi-inverse functors. Fix a $k$-scheme $T$. By construction, we have a canonical identification $FG(T) \, (\mathcal{V}|_{C \times T}) \cong \mathcal{V}|_{C \times T}$. It is also clear by construction that under this canonical identification we have $FG(T) \, (g) = g$ for any automorphism $g \in \text{Aut}(\mathcal{V}|_{C \times T})$. This remains true for $\mathcal{V}|_{S}$, where $S$ is any \`etale cover of $C \times T$. Technically our functor $G$ is not defined for such $S$, but the construction using the scheme $\text{Iso}(\mathcal{V}|_{S}, \, \mathcal{V}|_S)$ still makes sense. So we abuse notation and write $G$ for such construction.

Let $\mathcal{W}$ be any parabolic vector bundle of type $\mathcal{V}$ over $C \times T$. We know that there exists an \`etale cover $f: S \rightarrow C \times T$ such that $f^{*} \mathcal{W} \cong f^{*} ( \mathcal{V}|_{C \times T})$. Therefore, $\mathcal{W}$ can be described by some \`etale descent datum. In this case this is just an element of $\text{Aut}\left(\mathcal{V}|_{S \times_{C \times T} S}\right)$. We have seen that we have a canonical isomorphism $FG \, (\mathcal{V}|_{S \times_{C \times T} S}) \cong \mathcal{V}|_{S \times_{C \times T} S}$. Under this canonical identification, $FG(T)\, (\mathcal{W})$ is described by the same \`etale descent datum with respect to the  cover $S \rightarrow C\times T$ (since we said that $FG$ induces the identity on morphisms). So we can canonically identify  $\mathcal{W}$ and $FG(T)\, (\mathcal{W})$. The fact that $FG$ induces the identity on automorphisms of $\mathcal{V}|_{C \times T}$ implies the same for automorphims of $\mathcal{W}$ under the identification $FG(T) \, (\mathcal{W}) \, \cong \, \mathcal{W}$ we just constructed.
The argument for $GF(T) \cong id_{\text{Bun}_{\text{GL}(\mathcal{V})}(C)\, (T)}$ follows in a similar manner.
\end{proof}

\begin{coroll} \label{coroll: moduli of parabolic vb is algebraic}
$\text{Bun}_{\mathcal{V}}$ is a smooth algebraic stack over $k$.
\end{coroll}
\begin{proof}
This is an immediate consequence of Proposition \ref{prop: parabolic vb as parahoric torsors} and Proposition \ref{prop: algebraicity torsors over smooth group schemes}.
\end{proof}
\begin{remark}
It should be remarked that Corollary \ref{coroll: moduli of parabolic vb is algebraic} can be done directly by hand. One can describe a concrete atlas in terms of Quot schemes, just as is usually done for the moduli stack of vector bundles. We wanted to prove the isomorphism in Proposition \ref{prop: parabolic vb as parahoric torsors} in order to show explicitly how the moduli of parabolic vector bundles fits into the framework of moduli of torsors for Bruhat-Tits group schemes over $C$.
\end{remark}
\end{subsection}

\begin{subsection}{Harder-Narasimhan strata}
Fix a parabolic vector bundle $\mathcal{V}$ of rank $n$ over $C$. We will use the existence of the Harder-Narasimhan filtration in order to attach invariants to parabolic bundles of type $\mathcal{V}$. Let us first define the kind of invariants we will be working with.
 \begin{defn}
 A parabolic HN datum of rank $n$ is an $n$-tuple $(\mu_l)_{l=1}^{n}$ of real numbers $\mu_l$ such that $\mu_{l+1} \geq \mu_l$ for all $l$.
 \end{defn}
 We can define a partial order on the set of Harder-Narasimhan data of a given rank.
\begin{defn}
Let $P_1 = (\mu_l)_{l = 1}^n$ and $P_2 = (\nu_l)_{l=1}^n$ be two parabolic HN data of rank $n$. We say $P_1 \leq P_2$ if
\begin{enumerate}[(1)]
    \item $\sum_{l=1}^{n} \mu_l = \sum_{l=1}^{n} \nu_l$.
    \item For all $1 \leq k < n$, we have $\sum_{l=1}^{k} \mu_l \leq \sum_{l=1}^{k} \nu_l$.
\end{enumerate}
\end{defn}
For any parabolic bundle of type $\mathcal{V}$, we can associate a Harder-Narasimhan datum.
 \begin{defn} \label{defn: hn datum}
 Let $\mathcal{W}$ be a parabolic vector bundle in $\text{Vect}_{\overline{\lambda}}$. Suppose that it has Harder-Narasimhan filtration given by $0= \mathcal{W}_0 \subset \mathcal{W}_1 \subset \cdots \, \subset \mathcal{W}_{l} = \mathcal{W}$. Then, we define the Harder-Narasimhan datum $HN(\mathcal{W})$ associated to $\mathcal{W}$ to be the parabolic HN datum given by
 \begin{gather*} HN(\mathcal{W}) = ( \; \underbrace{\mu(\mathcal{W}_1) , \,\cdots, \,\mu(\mathcal{W}_1)}_{ \text{rank} \; \mathcal{W}_1 / 0 \, \text{times}}, \; \underbrace{\mu(\mathcal{W}_2/ \, \mathcal{W}_1) , \,\cdots, \,\mu(\mathcal{W}_2 / \, \mathcal{W}_1)}_{ \text{rank} \; \mathcal{W}_2 / \, \mathcal{W}_1 \, \text{times}}, \cdots, \, \underbrace{\mu(\mathcal{W}_l / \, \mathcal{W}_{l-1}) , \;\cdots, \;\mu(\mathcal{W}_l  / \, \mathcal{W}_{l-1})}_{ \text{rank} \; \mathcal{W}_l / \mathcal{W}_{l-1} \, \text{times}} \; )
 \end{gather*}
\end{defn}
We are now ready to define some relevant subfunctors of $\text{Bun}_{\mathcal{V}}$.
\begin{defn} \label{defn: HN strata}
Let $P$ be a HN-datum. We define $\text{Bun}_{\mathcal{V}}^{\leq P}$ to be the pseudofunctor from $k$-schemes into groupoids given as follows. For every $k$-scheme $S$,
\[ \text{Bun}_{\mathcal{V}}^{\leq P}(S) \; \vcentcolon = \; \left\{ \begin{matrix} \text{groupoid of } \mathcal{W} \in \text{Bun}_{\mathcal{V}}(S) \text{ such that} \\ HN\left( \mathcal{W} \, |_{C_s} \right) \leq P \; \; \text{for all } s\in S \end{matrix} \right\}  \]
\end{defn}

\begin{remark}
The locus of semistable bundles of a given degree $d$ is $\text{Bun}_{\mathcal{V}}^{\leq P}$, where $P$ is the the rank $n$ parabolic HN datum given by $P= \left( \frac{d}{n}, \frac{d}{n}, \cdots , \frac{d}{n}\right)$.
\end{remark}
Our main goal is to analyze the geometry of these subfunctors. It will turn out that all $\text{Bun}_{\mathcal{V}}^{\leq P}$s are quasicompact open substacks of $\text{Bun}_{\mathcal{V}}$. As we vary $P$, they form a stratification of $\text{Bun}_{\mathcal{V}}$. We will refer to the subfunctors $\text{Bun}_{\mathcal{V}}^{\leq P}$ collectively as Harder-Narasimhan strata.
\end{subsection}
\end{section}

\begin{section}{Parabolic Quot Schemes} \label{section: parabolic Quot schemes}
The main goal of this section is to prove the representability of the two functors described in Definition \ref{defn: second quot moduli} and Definition \ref{defn: parabolic quot filtration scheme}. These can be seen as parabolic analogues of the Quot scheme and the iterated Quot scheme respectively. The reader who is only interested in the main theorems described in the introduction can safely skip this section and refer back to it when needed.

\begin{defn} Let $\mathcal{Q}$ be a coherent sheaf on $C$. The Hilbert polynomial $HP(\mathcal{Q})$ of $\mathcal{Q}$ with respect to $\mathcal{O}(1)$ is the polynomial with rational coefficients whose value at an integer $n \in \mathbb{Z}$ is given by $HP(\mathcal{Q})(n) = \text{dim}_{k} \, H^0\left(\mathcal{Q}(n)\right) \, - \, \text{dim}_{k} \, H^1\left(\mathcal{Q}(n)\right)$.
\end{defn}

\begin{remark} \label{remark: riemann roch}
If $\mathcal{Q}$ is a vector bundle of rank $r$ and degree $d$, Riemann-Roch tells us that $HP(\mathcal{Q})(n) = (r \cdot d) \, n + r \cdot \chi(\mathcal{O_C})$. So the Hilbert polynomial contains precisely the information of the degree and the rank of the vector bundle.
\end{remark}
Let $S$ be a scheme of finite type over $k$. Let $\mathcal{E}$ be a coherent sheaf on $C\times S$ that is flat over $S$. Choose a polynomial $P \in \mathbb{Q}[T]$ with rational coefficients.
\begin{defn} We define $\text{Quot}_{\mathcal{E} / C\times S / S}^P$ to be the contravariant functor from $S$-schemes to sets given as follows. For any $S$-scheme $T$, $\text{Quot}_{\mathcal{E} / C\times S / S}^P (T)$ is the set of isomorphism classes of quotients $\mathcal{Q}$ of $\mathcal{E}|_{C \times T}$ such that
\begin{enumerate}[(a)]
    \item $\mathcal{Q}$ is flat over $T$
    \item Let $\mathcal{F}$ denote the kernel of $\mathcal{E}|_{C \times T} \twoheadrightarrow \mathcal{Q}$.  For all points $t \in T$, the Hilbert polynomial of the fiber $\mathcal{F}|_{C_t}$ is $P$.
\end{enumerate}
\end{defn}

In this context, we have the following classical theorem of Grothendieck. See \cite{nitsure-quot} for a proof.
\begin{thm} \label{thm: repr of classical quot scheme}
$\text{Quot}_{\mathcal{E} / C\times S / S}^P$ is represented by a projective $S$-scheme. \qed
\end{thm}

We want to define a variation of the Quot scheme that will serve our needs in latter sections. The set-up is as follows. We fix a finite set $\{x_i\}_{i \in I}$ of $k$-points in $C$. Let $S$ be a scheme of finite type over $k$. Let $\mathcal{V} = \left[ \; \mathcal{E}^{(0)} \,\overset{a_i^{(1)}}{\subset} \, \mathcal{E}_i^{(1)} \overset{a_i^{(2)}}{\subset} \cdots\; \overset{a_{i}^{(N_i)}}{\subset} \, \mathcal{E}_i^{(N_i)}= \mathcal{E}^{(0)}(x_i) \;\right]_{i \in I}$ be a parabolic vector bundle over $C\times S$ with parabolic structure at $\{x_i\}_{i \in I}$. We want to describe a functor that classifies subbundles of $\mathcal{V}$ with some given invariants. Let us first describe the invariants we will be looking at.
\begin{defn} \label{defn: parabolic quot datum}
A parabolic Quot datum of rank $k$ is a tuple $\left(P, \, b_i^{(m)}\right)$ where
\begin{enumerate}[(a)]
    \item $P \in \mathbb{Q}[T]$ is a polynomial with rational coefficients.
    \item For each $i \in I$ and $1 \leq m \leq N_i$, $b_i^{(m)}$ is a nonnegative integer. These are required to satisfy $\sum_{m =1}^{N_i} b_i^{(m)} = k$ for all $i \in I$.
\end{enumerate}
\end{defn}
We build our parabolic Quot scheme in stages. We start with the following moduli problem:
\begin{defn} \label{defn: first moduli quot}
Let $S$ and $\mathcal{V}$ as above. Fix $P \in \mathbb{Q}[T]$. We denote by $\text{Quot}_{\mathcal{V}}^P$ the subfunctor of $\text{Quot}_{\mathcal{E}^{(0)} / C\times S / S}^P$ given as follows. For a $S$-scheme $T$, define $\text{Quot}_{\mathcal{V}}^P(T)$ to be the set of isomorphism classes of quasicoherent quotients $\mathcal{Q}^{(0)}$ in $\text{Quot}_{\mathcal{E}^{(0)} / C\times S / S}^P(T)$ such that for all $i \in I$ and $0\leq m \leq N_i$, the sheaf 
$\mathcal{L}_i^{(m)}$ given as the pushout of the following diagram is $T$-flat.
\[ \xymatrix{
 \mathcal{Q}^{(0)}(x_i)  \ar[r] & \mathcal{L}_i^{(k)} \\
\mathcal{E}^{(0)}(x_i) \ar[u] \ar[r] & \mathcal{E}^{(0)}(x_i) \,/ \,\mathcal{E}_i^{(m)} \ar[u]} \]
\end{defn}

\begin{remark}
Since $\mathcal{L}_i^{(m)}$ is given as a pushout, its formation commutes with base-change. This is therefore a well defined functor.
\end{remark}

\begin{prop} \label{prop: representability of first quot moduli}
$\text{Quot}_{\mathcal{V}}^{P}$ is a finite disjoint union of locally closed subschemes of $\text{Quot}_{\mathcal{E}^{(0)} / C\times S / S}^P$.
\end{prop}
\begin{proof}
Let $\mathcal{Q}^{(0)}_{univ} \in \text{Coh}\left(C \times\text{Quot}_{\mathcal{E}^{(0)} / C\times S / S}^P\right)$ be the universal quotient of $\mathcal{E}^{(0)}$ on $C \times \text{Quot}_{\mathcal{E}^{(0)} / C\times S / S}^P$. Define $\mathcal{L}_{i, \, univ}^{(k)}$ to be the pushout diagram:
\[ \xymatrix{
 \mathcal{Q}^{(0)}_{univ}(x_i)  \ar[r] & \mathcal{L}_{i,\, univ}^{(m)} \\
\mathcal{E}^{(0)}(x_i)|_{C \times \text{Quot}_{\mathcal{E}^{(0)} / C\times S / S}^P} \ar[u] \ar[r] & \left( \mathcal{E}^{(0)}(x_i) \,/ \,\mathcal{E}_i^{(m)}\right) |_{C \times \text{Quot}_{\mathcal{E}^{(0)} / C\times S / S}^P}\ar[u]}  \]
By definition, $\text{Quot}_{\mathcal{V}}^{P} \longrightarrow \text{Quot}_{\mathcal{E}^{(0)} / C\times S / S}^P$ is the functor from $S$-schemes into set given as follows. For a $S$-scheme $T$, we have
\[ \text{Quot}_{\mathcal{V}}^{P} (T) \; = \; \left\{ \begin{matrix} \text{morphisms of $S$-schemes} \; f: T \rightarrow \text{Quot}_{\mathcal{E}^{(0)} / C\times S / S}^P \text{ such that} \\ (id_C \times f)^{*} \mathcal{L}_{i, \, univ}^{(m)}\; \text{is $T$-flat for all $i$ and $m$} \end{matrix} \right\}  \]
Now we can apply the existence of flattening stratifications (see \cite{nitsure-quot}) to the projective map $pr_2: \, C \times \text{Quot}_{\mathcal{E}^{(0)} / C\times S / S}^P \; \longrightarrow \; \text{Quot}_{\mathcal{E}^{(0)} / C\times S / S}^P$ and the coherent sheaves $\mathcal{L}_{i, \, univ}^{(m)}$ on $C \times \text{Quot}_{\mathcal{E}^{(0)} / C\times S / S}^P$. We conclude that this functor is represented by a finite disjoint union of locally closed subschemes of $\text{Quot}_{\mathcal{E}^{(0)} / C\times S / S}^P$.
\end{proof}

Suppose that we have $S$, $\mathcal{V}$, $T$ and $\mathcal{Q}^{(0)}$ as in Definition \ref{defn: first moduli quot}. We denote by $\mathcal{F}^{(0)}$ the kernel of the quotient $\mathcal{E}^{(0)} \longrightarrow \mathcal{Q}^{(0)}$. Furthermore, define $\mathcal{F}_i^{(m)} \vcentcolon= \mathcal{E}_i^{(m)} \cap \mathcal{F}^{(0)}(x_i)$ and $\mathcal{Q}_i^{(m)} \vcentcolon= \mathcal{E}_i^{(m)} \, / \, \mathcal{F}_i^{(m)}$. This means that we have the following commutative diagram with exact columns:
\begin{figure}[H]
\centering
\begin{tikzcd}
    0  & 0 & \cdots & 0 & 0 & 0\\
    \mathcal{Q}^{(0)} \ar [u] \ar[r, symbol= \subset] &  \mathcal{Q}_i^{(1)} \ar [u] \ar[r, symbol= \subset] & \cdots & \mathcal{Q}_i^{(Ni-2)} \ar [u] \ar[r, symbol= \subset] & \mathcal{Q}_i^{(N_i-1)} \ar [u] \ar[r, symbol= \subset] & \mathcal{Q}^{(0)}(x_i) \ar [u]\\
    \mathcal{E}^{(0)} \ar [u] \ar[r, symbol= \subset] &  \mathcal{E}_i^{(1)} \ar [u] \ar[r, symbol= \subset] & \cdots & \mathcal{E}_i^{(Ni-2)} \ar [u] \ar[r, symbol= \subset] & \mathcal{E}_i^{(N_i-1)} \ar [u] \ar[r, symbol= \subset] & \mathcal{E}^{(0)}(x_i) \ar [u]\\
    \mathcal{F}^{(0)} \ar [u] \ar[r, symbol= \subset] &  \mathcal{F}_i^{(1)} \ar [u] \ar[r, symbol= \subset] & \cdots & \mathcal{F}_i^{(Ni-2)} \ar [u] \ar[r, symbol= \subset] & \mathcal{F}_i^{(N_i-1)} \ar [u] \ar[r, symbol= \subset] & \mathcal{F}^{(0)}(x_i) \ar [u]\\
    0 \ar[u]  & 0 \ar[u] & \cdots & 0 \ar[u] & 0 \ar[u] & 0 \ar[u]
\end{tikzcd}
\end{figure}
By construction, we have a short exact sequence $0 \longrightarrow \, \mathcal{Q}_i^{(m)} \, \longrightarrow \, \mathcal{Q}^{(0)}(x_i) \, \longrightarrow \, \mathcal{L}_i^{(m)} \, \longrightarrow 0$. By the long exact sequence of Tors at each stalk, we conclude that $\mathcal{Q}_i^{(m)}$ is $T$-flat. The same reasoning applied to the short exact sequence $0 \longrightarrow \, \mathcal{F}_i^{(m)} \, \longrightarrow \, \mathcal{E}_i^{(m)} \, \longrightarrow \, \mathcal{Q}_i^{(m)} \, \longrightarrow 0$ shows that $\mathcal{F}_i^{(m)}$ is $T$-flat. In addition, the snake lemma gives us a short exact sequence
\[0 \longrightarrow \mathcal{F}^{(0)}(x_i) \, / \, \mathcal{F}_i^{(m)} \longrightarrow \mathcal{E}^{(0)}(x_i) \, / \,\mathcal{E}_i^{(m)} \longrightarrow \mathcal{L}_i^{(m)} \longrightarrow 0 \]
By $T$-flatness of $\mathcal{E}^{(0)}(x_i) \, / \, \mathcal{E}_i^{(m)}$ we conclude that $ \mathcal{F}^{(0)}(x_i) \, / \, \mathcal{F}_i^{(m)}$ is $T$-flat. Since all of these sheaves are $T$-flat, the formation of all of $\mathcal{F}_i^{(m)}$, $\mathcal{Q}_i^{(m)}$, $\mathcal{F}_i^{(m+1)} \, / \, \mathcal{F}_i^{(m)}$ and $\mathcal{Q}_i^{(m+1)} \, / \, \mathcal{Q}_i^{(m)}$ commutes with base-change on $T$. Hence the diagram above pulls back under base-change. This in turn shows that the following is a well-defined subfunctor.

\begin{defn} \label{defn: second quot moduli}
Fix $S$ and $\mathcal{V}$ as above. Let $\left(P, \, b_i^{(m)} \right)$ a parabolic Quot datum. We let $\text{Quot}_{\mathcal{V}}^{P, \, b_i^{(m)}}$ be the subfunctor of $\text{Quot}_{\mathcal{V}}^P$ defined as follows. For any $S$-scheme $T$, $\text{Quot}_{\mathcal{V}}^{P, \, b_i^{(m)}}(T)$ is the set of iso. classes of quotients $\mathcal{Q}^{(0)} \in \text{Quot}_{\mathcal{V}}^P$ satisfying
\begin{enumerate}[(a)]
    \item For all $i \in I$ and $1 \leq m \leq N_i$, $\mathcal{F}_i^{(m)} \, / \, \mathcal{F}_i^{(m-1)} |_{{x_i}\times T}$ is a vector bundle of rank $b_i^{(m)}$.
    \item $\mathcal{Q}_i^{(m)}$ and $\mathcal{F}_i^{(m)}$ are vector bundles on $C \times T$.
\end{enumerate}
Here we define $\mathcal{F}_i^{(m)} \vcentcolon= \mathcal{E}_i^{(m)}|_{C\times T} \cap \mathcal{F}^{(0)}(x_i)$ and $\mathcal{Q}_i^{(m)} \vcentcolon= \mathcal{E}_i^{(m)} \, / \, \mathcal{F}_i^{(m)}$, as explained in the paragraphs above.
\end{defn}

\begin{prop} \label{prop: representability of second quot moduli}
$\text{Quot}_{\mathcal{V}}^{P, \, b_i^{(m)}}$ is an open subscheme of $\text{Quot}_{\mathcal{V}}^P$. We call this the parabolic Quot scheme with parabolic Quot datum $\left(P, b_i^{(m)}\right)$.
\end{prop}
Before proving the proposition, let us state a lemma that will be used repeatedly.
\begin{lemma}[Fiberwise flatness criterion] \label{lemma: critere de platitude}
Let $B$ be a scheme and let $X$ be a scheme locally of finite presentation over $B$. Let $\mathcal{F} \in \text{QCoh}(X)$ be a sheaf that is locally of finite presentation as an $\mathcal{O}_X$-module. Suppose that $\mathcal{F}$ is flat over $B$. Let $x \in X$. Set $s = f(x)$.
If the fiber $\mathcal{F}|_{X_s}$ is $\mathcal{O}_{X_s}$-flat at $x$, then $\mathcal{F}$ is $\mathcal{O}_X$-flat at $x$.
\end{lemma}
\begin{proof}
This follows from \cite[\href{https://stacks.math.columbia.edu/tag/039C}{Tag 039C}]{stacks-project} by letting $X =Y$, $f = id_{X}$ and $B=S$.
\end{proof}

\begin{proof}[Proof of Proposition \ref{prop: representability of second quot moduli}]
Let $\mathcal{Q}^{(0)}_{univ} \in \text{Coh}(C\times\text{Quot}_{\mathcal{V}}^P)$ be the universal quotient of $\mathcal{E}^{(0)}$ on $C\times\text{Quot}_{\mathcal{V}}^P$. We denote by $\mathcal{F}^{(0)}_{univ}$ the kernel of the quotient $\mathcal{E}^{(0)}|_{C \times\text{Quot}_{\mathcal{V}}^P} \longrightarrow \mathcal{Q}^{(0)}_{univ}$. Furthermore, define $\mathcal{F}_{i, \,univ}^{(m)} \vcentcolon = \mathcal{E}_i^{(m)}|_{C \times\text{Quot}_{\mathcal{V}}^P} \, \cap \, \mathcal{F}^{(0)}_{univ}(x_i)$ and $\mathcal{Q}_{i, \, univ}^{(m)} \vcentcolon = \mathcal{E}_i^{(m)}|_{C \times\text{Quot}_{\mathcal{V}}^P} \, / \, \mathcal{F}_{i, \, univ}^{(m)}$.
By definition, the functor $\text{Quot}_{\mathcal{V}}^{P, \, b_i^{(m)}} \rightarrow \text{Quot}_{\mathcal{V}}^P$ is given as follows. For every $S$-scheme $T$,
\begin{gather*} \text{Quot}_{\mathcal{V}}^{P, \, b_i^{(m)}} (T) \; = \; \left\{ \begin{matrix} \text{morphisms of $S$-schemes} \; f: T \rightarrow \text{Quot}_{\mathcal{V}}^P \text{ such that}\\
\\
(1) \quad \text{For all $i$ and $m$ we have that }\;(q_i \times f)^{*} \, \mathcal{F}_{i, \, univ}^{(m)} \, / \, \mathcal{F}_{i, \, univ}^{(m-1)} \\
\text{is a vector bundle of rank $b_i^{(m)}$}\\
\\
(2) \quad \text{For all $i$ and $m$, we have that both} \; (id_C \times f)^{*} \mathcal{F}_{i, \, univ}^{(m)}\\
  \text{and} \; (id_C \times f)^{*} \mathcal{Q}_{i, \, univ}^{(m)} \; \text{are $C \times T$-flat} \end{matrix} \right\}
  \end{gather*}
  Let us define an intermediate functor $L$ given by imposing only the first condition. In other words, we let $L$ be the functor from $S$-schemes $T$ into sets given by
  \[ L (T) \; = \; \left\{ \begin{matrix} \text{morphisms of $S$-schemes} \; f: T \rightarrow \text{Quot}_{\mathcal{V}}^P \text{ such that} \\
\\
(1) \quad \text{For all $i$, $m$ we have that}\\
f^{*} \, (q_i \times id_{\text{Quot}_{\mathcal{V}}^P})^{*} \, \mathcal{F}_{i, \, univ}^{(m)} \, / \, \mathcal{F}_{i, \, univ}^{(m-1)} \; \text{is a vector bundle of rank $b_i^{(m)}$} \end{matrix} \right\}  \]
By the discussion before Definition \ref{defn: second quot moduli}, we already know that $(q_i \times id_{\text{Quot}_{\mathcal{V}}^P})^{*} \, \mathcal{F}_{i, \, univ}^{(m)} \, / \, \mathcal{F}_{i, \, univ}^{(m-1)}$ is a vector bundle.  Hence for any morphism $f: T \rightarrow \text{Quot}_{\mathcal{V}}^P$, we have that $f^{*}(q_i \times id_{\text{Quot}_{\mathcal{V}}^P})^{*} \, \mathcal{F}_{i, \, univ}^{(m)} \, / \, \mathcal{F}_{i, \, univ}^{(m-1)}$ is a vector bundle. Note that the condition for the rank can be checked at each point $t \in T$. In other words, $L$ can be rewritten as
 \[ L (T) \; = \; \left\{ \begin{matrix} \text{morphisms of $S$-schemes} \; f: T \rightarrow \text{Quot}_{\mathcal{V}}^P \text{ such that} \\
\\
(1) \quad \text{For all $i$, $m$ and $t \in T$ we have}\\
\text{rank}_{\kappa(t)} \, \left(f^{*}(q_i \times id_{\text{Quot}_{\mathcal{V}}^P})^{*} \, \mathcal{F}_{i, \, univ}^{(m)} \, / \, \mathcal{F}_{i, \, univ}^{(m-1)} \right)|_{t} \;= \; b_i^{(m)} \end{matrix} \right\}  \]
But we know that
\begin{gather*}\text{rank}_{\kappa(t)} \, \left(f^{*}(q_i \times id_{\text{Quot}_{\mathcal{V}}^P})^{*} \, \mathcal{F}_{i, \, univ}^{(m)} \, / \, \mathcal{F}_{i, \, univ}^{(m-1)} \right)|_{t} \; = \; \text{rank}_{\kappa(f(t))}\left((q_i \times id_{\text{Quot}_{\mathcal{V}}^P})^{*} \, \mathcal{F}_{i, \, univ}^{(m)} \, / \, \mathcal{F}_{i, \, univ}^{(m-1)} \right)|_{f(t)}
\end{gather*}
Denote by $U$ the set of points $s \in \text{Quot}_{\mathcal{V}}^P$ such that $\text{rank}_{\kappa(s)}\left((q_i \times id_{\text{Quot}_{\mathcal{V}}^P})^{*} \, \mathcal{F}_{i, \, univ}^{(m)} \, / \, \mathcal{F}_{i, \, univ}^{(m-1)} \right)|_{s} \, = \, b_i^{(m)}$ for all $i$ and $m$. Then $L$ is given by the set theoretic condition
\[ L (T) \; = \; \left\{ \begin{matrix} \text{morphisms of $S$-schemes} \; f: T \rightarrow \text{Quot}_{\mathcal{V}}^P \text{ such that} \\
\\
(1) \quad \text{For all $t \in T$ we have $f(t) \in U$}\\
 \end{matrix} \right\}  \]
Since $(q_i \times id_{\text{Quot}_{\mathcal{V}}^P})^{*} \, \mathcal{F}_{i, \, univ}^{(m)} \, / \, \mathcal{F}_{i, \, univ}^{(m-1)}$ is locally free, the rank of the fibers is a locally constant function on $\text{Quot}_{\mathcal{V}}^P$. In particular, we see that $U$ is an open subset of $\text{Quot}_{\mathcal{V}}^P$. It follows that the functor $L$ is represented by $U$ viewed as an open subscheme of $\text{Quot}_{\mathcal{V}}^P$.
 
 We can now go back to our functor $\text{Quot}_{\mathcal{V}}^{P, \, b_i^{(m)}}$. We see that it is given by
 \[ \text{Quot}_{\mathcal{V}}^{P, \, b_i^{(m)}} (T) \; = \; \left\{ \begin{matrix} \text{morphisms of $S$-schemes} \; f: T \rightarrow U \text{ such that} \\
\\
(2) \quad \text{For all $i$ and $m$, we have that both}\\
(id_C \times f)^{*} \mathcal{F}_{i, \, univ}^{(m)} \;
  \text{and} \; (id_C \times f)^{*} \mathcal{Q}_{i, \, univ}^{(m)} \; \text{are $C \times T$-flat} \end{matrix} \right\}  \]
  By the discussion before Definition \ref{defn: second quot moduli}, the sheaves $\mathcal{F}_{i, \, univ}^{(m)}$ and $\mathcal{Q}_{i, \, univ}^{(m)}$ are $\text{Quot}_{\mathcal{V}}^P$-flat. So for any morphism $f: T \rightarrow U$ the sheaves $(id_C \times f)^{*} \mathcal{F}_{i, \, univ}^{(m)}$ and $(id_C \times f)^{*} \mathcal{Q}_{i, \, univ}^{(m)}$ are $T$ flat. Hence we can apply Lemma \ref{lemma: critere de platitude} with $X = C\times T$ and $B = T$ to conclude that it suffices to check flatness for the fibers $C_t$ at each $t \in T$. So our fuctor can be described as:
 \[ \text{Quot}_{\mathcal{V}}^{P, \, b_i^{(m)}} (T) \; = \; \left\{ \begin{matrix} \text{morphisms of $S$-schemes} \; f: T \rightarrow U \text{ such that} \\
\\
(2) \quad \text{For all $i$ and $m$, for every $t \in T$ we have that both} \; \\
(id_C \times f)^{*} \mathcal{F}_{i, \, univ}^{(m)}|_{C_t} \; \text{and} \; (id_C \times f)^{*} \mathcal{Q}_{i, \, univ}^{(m)}|_{C_t} \; \text{are $C_t$ flat} \end{matrix} \right\}  \]
  Since flatness can be checked after faithfully flat base-change, we see that $(id_C \times f)^{*} \mathcal{Q}_{i, \, univ}^{(m)}|_{C_t}$ is flat if and only if $\mathcal{Q}_{i, \, univ}^{(m)}|_{C_{f(t)}}$ is flat. The same reasoning applies to $(id_C \times f)^{*} \mathcal{F}_{i, \, univ}^{(m)}|_{C_t}$. Let $V$ be the set of points $s \in U$ such that $\mathcal{Q}_{i, \, univ}^{(m)}|_{C_{s}}$ and $\mathcal{F}_{i, \, univ}^{(m)}|_{C_{s}}$ are $C_s$-flat for all $i$ and $m$. Then we have just seen that the functor can described as
  \[ \text{Quot}_{\mathcal{V}}^{P, \, b_i^{(m)}} (T) \; = \; \left\{ \begin{matrix} \text{morphisms of $S$-schemes} \; f: T \rightarrow U \text{ such that} \\
\\
(2) \quad \text{For all $t \in T$, we have that $f(t)\in V$} \end{matrix} \right\}  \]
We are reduced to check that $V$ is an open subset of $U$. For a fixed choice of indexes $i$ and $m$,  let $V_{i,m}^{\mathcal{Q}}$ be the set of points $s \in U$ such that $\mathcal{Q}_{i, \, univ}^{(m)}|_{C_{s}}$ is flat. We shall show that $V_{i,m}^{\mathcal{Q}}$ is open.

Suppose that $s \in V_{i,m}^{\mathcal{Q}}$. The sheaf $\mathcal{Q}_{i, \, univ}^{(m)}$ is $U$-flat (see the discussion before Definition \ref{defn: second quot moduli}). Hence we can apply Lemma \ref{lemma: critere de platitude} to the morphism $pr_2: \, C \times U \, \rightarrow U$  (here $X = C\times U$ and $B=U$) in order to conclude that $\mathcal{Q}_{i, \, univ}^{(m)}$ is flat at all points $x \in C \times U$ lying over $s$. Since the flatness locus of $\mathcal{Q}_{i, \, univ}^{(m)}$ is open in $C \times U$ \cite[\href{https://stacks.math.columbia.edu/tag/0399}{Tag 0399}]{stacks-project}, there is an open $N \subset C \times U$ containing the fiber $C_s$ such that $\mathcal{Q}_{i, \, univ}^{(m)}|_N$ is flat. Since the morphism $pr_2: \, C \times U \, \rightarrow U$ is projective, $pr_2$ is in particular closed. So we can find an open neighborhood $W$ of $s$ in $U$ such that $(pr_2)^{-1}(W) \subset N$. We conclude that $W \subset V_{i, \, m}^{\mathcal{Q}}$, which shows that $V_{i, \, m}^{\mathcal{Q}}$ is open.

The same argument applies verbatim to the locus $V_{i \, m}^{\mathcal{F}}$ of points $s \in U$ such that $\mathcal{F}_{i, \, univ}^{(m)}|_{C_{s}}$ is flat. Since $V = \bigcap_{i, \, m} \left(V_{i \, m}^{\mathcal{F}} \, \cap \, V_{i \, m}^{\mathcal{Q}}\right)$, we conclude that $V$ is open.
\end{proof}

\begin{coroll} \label{coroll: finiteness of parabolic Quot scheme}
$\text{Quot}_{\mathcal{V}}^{P, \, b_i^{(m)}}$ is a separated scheme of finite type over $S$.
\end{coroll}
\begin{proof}
This is an immediate consequence of Theorem \ref{thm: repr of classical quot scheme}, Proposition \ref{prop: representability of first quot moduli} and Proposition \ref{prop: representability of second quot moduli}.
\end{proof}

\begin{defn}
Let $\mathcal{V}$ is a parabolic vector bundle over $C$. Let $\mathcal{W}$ be a parabolic subbundle of $\mathcal{V}$. There is a unique Quot datum $\theta(\mathcal{W})$ such that $\mathcal{W}$ is a $k$-point of $\text{Quot}_{\mathcal{V}}^{\theta(\mathcal{W})}$. We call $\theta(\mathcal{W})$ the parabolic Quot datum associated to $\mathcal{W}$.
More explicitly, if we have $\mathcal{W} = \left[ \;\mathcal{F}^{(0)} \,\overset{b_i^{(1)}}{\subset} \, \mathcal{F}_i^{(1)} \overset{b_i^{(2)}}{\subset} \cdots\; \overset{b_{i}^{(N_i)}}{\subset} \,\mathcal{F}_i^{(N_i)}= \, \mathcal{F}^{(0)}(x_i) \;\right]_{i \in I}$, then $\theta(\mathcal{W}) = \left( HP(\mathcal{F}^{(0)}), \, b_i^{(m)} \right)$.
\end{defn}

\begin{remark} \label{remark: parabolic Quot invariants}
\begin{enumerate}[(i)]
    \item  Let $\mathcal{W}$ be a parabolic vector bundle over $C$. Recall that the rank and degree of a vector bundle can be recovered from its Hilbert polynomial via Riemann-Roch. By the degree formula for parabolic vector bundles, it follows that $\mu(\mathcal{W})$ is determined by $\theta(\mathcal{W})$.
    \item Let $T$ be a connected $k$-scheme. Let $\mathcal{W}$ be a parabolic vector bundle over $C \times T$. Since the degree and rank of a vector bundle are locally constant on families, it follows that $\theta(\mathcal{W}|_{C_t})$ is the same for all $t \in T$.
\end{enumerate}
\end{remark}
We now want to iterate our construction in order to parametrize filtrations of a parabolic vector bundle by parabolic subbundles with some given parabolic Quot data.

\begin{defn}
Let $S$ and $\mathcal{V}$ be as before. A filtration datum of length $l$ is a pair of tuples $\alpha = \left( \left( _j b_i^{(m)} \right)_{\substack{1 \leq j \leq l-1 \\ i \in I \\ 1 \leq m \leq N_i}}, \; \left(_j P \right)_{1 \leq j \leq l-1} \right)$, where
\begin{enumerate}[(a)]
    \item Each $_j b_i^{(m)}$ is a nonnegative integer.
    \item Each $_j P \in \mathbb{Q}[T]$ is a polynomial with rational coefficients.
\end{enumerate}
\end{defn}

\begin{defn} \label{defn: parabolic quot filtration scheme} Let $S$ and $\mathcal{V}$ as usual. Fix a filtration datum $\alpha$ of length $l$. Then, we define the functor $\text{Fil}_{\mathcal{V}}^{\alpha}$ over $S$ as follows. For any $S$-scheme $T$, $\text{Fil}^{\alpha}_{\mathcal{V}}(T)$ will denote the set of iso. classes of filtrations of $\mathcal{V}|_{T\times C}$ by parabolic subbundles $0 = _0\mathcal{V} \, \subset\, _1\mathcal{V} \, \subset \, _2 \mathcal{V} \, \subset \, \cdots \, \subset \, _l \mathcal{V} = \mathcal{V}$ such that if $_j\mathcal{V} = \left[ \; _j\mathcal{F}^{(0)} \,\subset \, _j\mathcal{F}_i^{(1)} \subset \cdots\; \subset \, _j\mathcal{F}_i^{(N_i)}=  \, _j \mathcal{F}^{(0)}(x_i) \;\right]_{i \in I}$, then we have:
\begin{enumerate}[(a)]
    \item For points $t \in T$ and all $1 \leq j \leq l-1$, the Hilbert Polynomial of $_j \mathcal{F}^{(0)} \, |_{C_{t}}$ is $_j P$.
    \item For all $1 \leq j \leq l-1$, $i \in I$ and $1 \leq m \leq N_i$, we have that $_j \mathcal{F}_i^{(m)} \, / \, _j \mathcal{F}_i^{(m-1)}\, |_{ T \times {x_i}}$ is a vector bundle of rank $_j b_i^{(m)}$.
\end{enumerate}
\end{defn}

\begin{prop} \label{prop: representability of filtration moduli}
$\text{Fil}^{\alpha}_{\mathcal{V}}$ is represented by a scheme that is separated and of finite type over $S$.
\end{prop}
\begin{proof}
We induct on the length $l$ of the filtration datum. The case $l =1$ is just a parabolic Quot scheme, so it follows from Corollary \ref{coroll: finiteness of parabolic Quot scheme}. Suppose that the theorem has been proved for filtration data of size $l-1$. Let $\alpha$ be as in the statement of the proposition. We can define $\beta \vcentcolon = \left( \left( _j b_i^{(m)} \right)_{\substack{2 \leq j \leq l-1 \\ i \in I \\ 1 \leq m \leq N_i}}, \; \left(_j P \right)_{2 \leq j \leq l-1} \right)$, which is a filtration datum of length $l-1$. By the induction hypothesis, we know that $Fil_{\mathcal{V}}^{\beta}$ is a scheme that is separated and of finite type over $S$.

There is a universal filtration $0 \subset \, \mathcal{W}_{2, \, univ} \, \subset \, \mathcal{W}_{3, \, univ} \, \subset \cdots\, \subset \, \mathcal{W}_{l, \, univ} = \mathcal{V}|_{C \times Fil^{\beta}_{\mathcal{V}}}$ by parabolic subbundles. Here each $\mathcal{W}_{j, \, univ}$ has the required associated parabolic data $\theta(\mathcal{W}_{j, \, univ}) = \left(_jP, \, _jb_i^{(m)}\right)$ for $2 \leq j \leq l-1$. We can define a map $Fil^{\alpha}_{\mathcal{V}} \rightarrow Fil^{\beta}_{\mathcal{V}}$ by forgetting the first element in the parabolic filtration. By tracing back the definitions, the functor $Fil^{\alpha}_{\mathcal{V}}$ is represented by the scheme $\text{Quot}_{\mathcal{W}_{2, \, univ}}^{_1P, \, _1b_i^{(m)}}$. Since we know that $\text{Quot}_{\mathcal{W}_{2, \, univ}}^{_1P, \, _1b_i^{(m)}} \longrightarrow Fil^{\beta}_{\mathcal{V}}$ is a separated morphism of finite type, we conclude that the composition $Fil^{\alpha}_{\mathcal{V}} \longrightarrow Fil^{\beta}_{\mathcal{V}} \longrightarrow S$is separated and of finite type.
\end{proof}

\begin{defn} \label{defn: filtration datum associated to filtration}
Let $\mathcal{V}$ be a parabolic vector bundle over $C$. Suppose that we have a filtration $\{_j\mathcal{W}\}$ of length $l$ $0 = \, _0\mathcal{W} \subset \, _1\mathcal{W} \subset \cdots \, \subset \, _{l-1}\mathcal{W} \subset \, _{l}\mathcal{W} = \mathcal{V}$.
We know that such filtration is a $k$-point of $\text{Fil}^{ \, \psi(\mathcal{W}_j)}_{\mathcal{V}}$ for some filtration datum $\psi(\mathcal{W}_j)$ of length $l$. We call this the filtration datum associated to the filtration $\{\mathcal{W}_j\}$.

More explicitly, if we have $_j\mathcal{W} = \left[ \;_j\mathcal{F}^{(0)} \,\overset{_jb_i^{(1)}}{\subset} \, _j\mathcal{F}_i^{(1)} \overset{_jb_i^{(2)}}{\subset} \cdots\; \overset{_jb_{i}^{(N_i)}}{\subset} \,_j\mathcal{F}_i^{(N_i)}= \, _j\mathcal{F}^{(0)}(x_i) \;\right]_{i \in I}$, then $\psi(\mathcal{W}_j) = \left( _jb_i^{(m)}, \, HP(_j\mathcal{F}^{(0)}) \right)$.
\end{defn}
\end{section}
\begin{section}{Openness of strata} \label{big section: openness of strata}
The goal of this section is to prove that each stratum $\text{Bun}_{\mathcal{V}}(C)^{\leq P}$ is open (Theorem \ref{thm: openess of strata}). This is achieved by first proving that $\text{Bun}_{\mathcal{V}}^{\leq P}$ is constructible (Proposition \ref{prop: constructibility of strata}), and then proving that $\text{Bun}_{\mathcal{V}}^{\leq P}$ is closed under generalization (Proposition \ref{prop: openness of locus of strata in families}). A treatment of the analogous result in the case of principal $G$-bundles over a curve of characteristic $0$ can be found in \cite{gurjar-nitsure}[\S2]. Our general approach is different from \cite{gurjar-nitsure} in that we prove constructibility of all strata directly without treating the openness of the semistable locus first.
\begin{subsection}{Constructibility of strata} \label{section: constructibility of strata}
The goal of this subsection is to prove the following proposition.
\begin{prop} \label{prop: constructibility of strata}
Let $T$ be a scheme of finite type over $k$. Let $\mathcal{V}$ be a parabolic vector bundle of rank $n$ over $C \times T$. Let $P = (\mu_1, \mu_2, \cdots, \mu_n)$ be a parabolic HN datum. Then, the subset of $T$ given by $T^{\leq P} = \left\{ t \in T \, \mid \, HN( \mathcal{V}|_{C_t}) \leq P \right\}$ is constructible in the Zariski topology.
\end{prop}
Before proving Proposition \ref{prop: constructibility of strata}, we need a couple of lemmas.
\begin{lemma} \label{lemma: existence destabilizing subbundle}
Let $\mathcal{V}$ be a parabolic vector bundle of rank $n$ over $C$. Let $P = ( \mu_1, \mu_2, \cdots, \mu_n)$ be a parabolic HN datum with $\sum_{j =1}^n \mu_j = \text{deg} \, \mathcal{V}$. Suppose that $HN(\mathcal{V}) \nleq P$. Then, there exists a parabolic subbundle $\mathcal{W}$ of $\mathcal{V}$ such that $\mu(\mathcal{W}) > \frac{1}{r} \sum_{l=1}^r \mu_l$, where $r$ is the rank of $\mathcal{W}$.
\end{lemma}
\begin{proof}
Suppose that $\mathcal{V}$ has Harder-Narasimhan filtration given by $0 \subset \mathcal{V}_0 \subset \mathcal{V}_1 \subset \cdots \, \subset \mathcal{V}_{l} = \mathcal{V}$. Recall that
\[ HN(\mathcal{V}) = ( \; \underbrace{\mu(\mathcal{V}_0) , \,\cdots, \,\mu(\mathcal{V}_0)}_{ \text{rank} \; \mathcal{V}_0 / 0 \, \text{times}}, \; \underbrace{\mu(\mathcal{V}_1/ \, \mathcal{V}_0) , \,\cdots, \,\mu(\mathcal{V}_1 / \, \mathcal{V}_0)}_{ \text{rank} \; \mathcal{V}_1 / \, \mathcal{V}_0 \, \text{times}}, \cdots, \, \underbrace{\mu(\mathcal{V}_l / \, \mathcal{V}_{l-1}) , \;\cdots, \;\mu(\mathcal{V}_l  / \, \mathcal{V}_{l-1})}_{ \text{rank} \; \mathcal{V}_l / \mathcal{V}_{l-1} \, \text{times}} \; ) \]
Let us write $HN(\mathcal{V}) = (\nu_1, \nu_2, \cdots, \nu_n)$. Then we know $\sum_{j=1}^n \nu_j = \text{deg} \, \mathcal{V} = \sum_{j=1}^n \mu_j$. By definition, $HN(\mathcal{V}) \nleq P$ implies that there exists $1 \leq m < n$ such that $\sum_{j=1}^m \nu_j > \sum_{j=1}^m \mu_j$. Let $m_0$ the maximal such $m$. Let $1 \leq q < l$ be such that $\text{rank} \, \mathcal{V}_q \leq m_0 < \text{rank}\, \mathcal{V}_{q+1}$.

Define $r \vcentcolon =\text{rank} \, \mathcal{V}_q $. We claim that we have $\sum_{j=1}^{r} \nu_j >\sum_{j=1}^{r} \mu_j$. Once we prove this claim we will be done, because we can then set $\mathcal{W} = \mathcal{V}_q$. 

In order to prove the claim, note that we have
\[ 0 < \sum_{j=1}^{m_0} \nu_j  - \sum_{j=1}^{m_0} \mu_j =  \left(\sum_{j=1}^{r} \nu_j - \sum_{j=1}^{r} \mu_j \right)  + \sum_{j=r+1}^{m_0} (\nu_j - \mu_j) \]
So it suffices to show that $\nu_j \leq \mu_j$ for all $r+1 \leq j \leq m_0$. For all $r+1 \leq j \leq m_0$ we have $\nu_j = \mu(\mathcal{V}_{q+1}) = \nu_{m_0+1}$. Using $\mu_j \geq \mu_{m_0+1}$, we are reduced to showing that $\nu_{m_0+1} \leq \mu_{m_0+1}$.

By maximality of $m_0$, we have $\sum_{j=1}^{m_0} \nu_j > \sum_{j=1}^{m_0} \mu_j$ and $\sum_{j=1}^{m_0+1} \nu_j \leq \sum_{j=1}^{m_0+1} \mu_j$. This means that we must have $\nu_{m_0+1} < \mu_{m_0+1}$, as desired.
\end{proof}

\begin{defn}
If $\mathcal{V}$, $P$ and $\mathcal{W}$ are as in Lemma \ref{lemma: existence destabilizing subbundle}, we say that $\mathcal{W}$ is a $P$-destabilizing subbundle of $\mathcal{V}$.
\end{defn}
For the next lemma, we need to talk about $P$-destabilizing subbundles for the fibers of a family. 
\begin{defn} Let $T$ be a $k$-scheme of finite type and let $\mathcal{V}$ be a parabolic vector bundle over $T \times_{k} C$. Let $P$ be a HN datum. For any topological point $t \in T$, we denote by $Dest_{t}^{P} (\mathcal{V})$ the set of $P$-destabilizing subbundles of $\mathcal{V}|_{C_t}$.
\end{defn}
\begin{lemma} \label{lemma: finiteness destabilizing family}
Let $T$ , $\mathcal{V}$ and $P$ as in the definition above. The set of all Hilbert polynomials
\begin{gather*} \left\{ HP\left(\mathcal{F}^{(0)}\right) \; \mid \; \exists t \in T, \; \exists \mathcal{W} = \left[ \; \mathcal{F}^{(0)} \,\overset{a_i^{(1)}}{\subset} \, \mathcal{F}_i^{(1)} \overset{a_i^{(2)}}{\subset} \cdots\; \overset{a_{i}^{(N_i)}}{\subset} \, \mathcal{F}_i^{(N_i)}= \mathcal{F}^{(0)}(x_i) \;\right]_{i \in I} \in Dest_{t}^{P}(\mathcal{V}) \right\}
\end{gather*}
is finite.
\end{lemma}
\begin{proof}
Let $\mathcal{V} = \left[ \; \mathcal{E}^{(0)} \,\overset{a_i^{(1)}}{\subset} \, \mathcal{E}_i^{(1)} \overset{a_i^{(2)}}{\subset} \cdots\; \overset{a_{i}^{(N_i)}}{\subset} \, \mathcal{E}_i^{(N_i)}= \mathcal{E}^{(0)}(x_i) \;\right]_{i \in I}$. Let $t \in T$. Recall (Remark \ref{remark: riemann roch}) that the Hilbert polynomial of a vector bundle on $C_t$ is completely determined by the degree and the rank of the bundle. For any 
\[\mathcal{W} = \left[ \; \mathcal{F}^{(0)} \,\overset{b_i^{(1)}}{\subset} \, \mathcal{F}_i^{(1)} \overset{b_i^{(2)}}{\subset} \cdots\; \overset{b_{i}^{(N_i)}}{\subset} \, \mathcal{F}_i^{(N_i)}= \mathcal{F}^{(0)}(x_i) \;\right]_{i \in I} \in Dest_t^{P}(\mathcal{V})\]
we know that $1 \leq \text{rank} \, \mathcal{F}^{(0)} \leq \text{rank} \, \mathcal{V}$. So there are only finitely many possible ranks for such $\mathcal{F}^{(0)}$. It therefore suffices to fix some rank $r$ and show that there are finitely many possible values of $\text{deg} \, \mathcal{F}^{(0)}$ for $\mathcal{W} \in \text{Dest}_t^P(\mathcal{V})$ of rank $r$ as above.

By definition of $P$-destabilizing bundle, we know that for such a $\mathcal{W}$ we have $\text{deg}\, \mathcal{W} >  \sum_{l=1}^r \mu_l$. By Lemma \ref{lemma: deg of parabolic vs regular vector bundles} we know that $deg \, \mathcal{W} \, \leq deg\left(\mathcal{F}^{(0)} \right) + r|I|$. Hence we must have $\text{deg} \, \mathcal{F}^{(0)} > \sum_{l=1}^r \mu_l - r|I|$. It therefore suffices to find a uniform upper bound for the possible values of $\text{deg} \, \mathcal{F}^{(0)}$. We argue by contradiction. Suppose that the degrees of all possible subbundles of $\mathcal{E}^{(0)}|_{C_t}$ for $t \in T$ are not uniformly bounded above. Since $H^{0}\left(C_t, \, \mathcal{E}^{(0)}|_{C_t}\right) \supset H^{0}\left(C_t, \, \mathcal{F}^{(0)}\right)$, Riemann-Roch for $\mathcal{F}^{(0)}$ implies that the dimension of $H^{0}\left(C_t, \, \mathcal{E}^{(0)}|_{C_t}\right)$ is not bounded above. This in turn implies that for each $m \geq 0$ the dimension of $H^{0}\left(C_t, \, \mathcal{E}^{(0)}(m)|_{C_t}\right)$ is not bounded. Serre vanishing and cohomology and base-change imply that there exists some $m>>0$ such that $H^{0}\left(C_t, \, \mathcal{E}^{(0)}(m)|_{C_t}\right)$ is the fiber of the pushforward $(\pi_T)_* \, \mathcal{E}^{(0)}(m)$ for all $t \in T$. Since $\pi_T$ is a projective morphism, the sheaf $(\pi_T)_* \, \mathcal{E}^{(0)}(m)$ is coherent on the Noetherian scheme $T$. This means that the dimension of the fibers of $(\pi_T)_* \, \mathcal{E}^{(0)}(m)$ are uniformly bounded, a contradition.
\end{proof}
We are now ready to prove the main proposition of this section.
\begin{proof}[Proof of Proposition \ref{prop: constructibility of strata}]
The question is local, so we can restrict to each (open) connected component of $T$. We can therefore assume that the degree of the fibers $\mathcal{V}|_{C_t}$ is the same for all $t \in T$. Let us denote this commond degree by $d$. We can assume without loss of generality that $\sum_{j=1}^n \mu_i =d$, because otherwise we would have that $T^{\leq P}$ is empty. In order to prove constructibility, we shall show that the complement $T^{\nleq P} \vcentcolon = |T| \setminus T^{\leq P}$ is constructible.

For each $t \in T$ and $\mathcal{W} \in \text{Dest}_t^P(\mathcal{V})$, we can associate a parabolic Quot datum $\theta(\mathcal{W})$ as in Definition \ref{defn: parabolic quot datum}. Define $\mathfrak{S} \vcentcolon = \left\{ \,\theta(\mathcal{W}) \; \mid \; \mathcal{W} \in \text{Dest}_t^P(\mathcal{V}) \; \text{for some $t \in T$} \, \right\}$. Lemma \ref{lemma: finiteness destabilizing family} tells us that the set of possible first coordinates of tuples $\left( HP(\mathcal{F}^{(0)}), (b_i^{(m)}) \right)$ in $\mathfrak{S}$ is finite. But there are only finitely many possible choices of nonnegative integers $b_i^{(m)}$, because $\sum_{m=1}^{N_i} b_i^{(m)} = \text{rank} \, \mathcal{W}$ for all $i \in I$ and the rank of $\mathcal{W}$ is always smaller than the rank of $\mathcal{V}$. Hence $\mathfrak{S}$ is finite.

The $T$-scheme $\underset{(P,\, b_i^{(m)}) \in \mathfrak{S}}{\sqcup} \text{Quot}_{\mathcal{V}}^{P, \, b_i^{(m)}}$ is of finite type over $T$, because it is a finite disjoint union of schemes of finite type over $T$. By Chevalley's Theorem, the set theoretic image of the structure morphism $\pi: \underset{(P,\, b_i^{(m)}) \in \mathfrak{S}}{\sqcup} \text{Quot}_{\mathcal{V}}^{P, \, b_i^{(m)}} \rightarrow T$ is constructible. We claim that this image is precisely $T^{\nleq P}$.

Let $t \in T^{\nleq P}$. By Lemma \ref{lemma: existence destabilizing subbundle}, $\mathcal{V}|_{C_t}$ admits a $P$-destabilizing subbundle $\mathcal{W}$ over $C_t$. By the definition of the parabolic Quot scheme, $\mathcal{W}$ represents a $t$-point of $\text{Bun}_{\mathcal{V}}^{\theta(\mathcal{W})}$. Therefore $t$ is in the image of the structure morphism $\pi : \underset{(P,\, b_i^{(m)}) \in \mathfrak{S}}{\sqcup} \text{Quot}_{\mathcal{V}}^{P, \, b_i^{(m)}} \rightarrow T$.

Conversely, let $s$ be a point of $\text{Bun}_{\mathcal{V}}^{P, \, b_i^{(m)}}$ for some $(P, \, b_i^{(m)}) \in \mathfrak{S}$. By definition, this means that $\mathcal{V}|_{C_s}$ has a subbundle $\mathcal{W}$ over $C_s$ of the form
\[\mathcal{W} = \left[ \; \mathcal{F}^{(0)} \,\overset{b_i^{(1)}}{\subset} \, \mathcal{F}_i^{(1)} \overset{b_i^{(2)}}{\subset} \cdots\; \overset{b_{i}^{(N_i)}}{\subset} \, \mathcal{F}_i^{(N_i)}= \mathcal{F}^{(0)}(x_i) \;\right]_{i \in I}\]
such that $HP(\mathcal{F}^{(0)}) = P$. But we know that $(P, \, b_i^{(m)}) \in \mathfrak{S}$. We can compute the slope of $\mathcal{W}$ from the parabolic Quot data $(P, \, b_i^{(m)})$, and we see that $\mathcal{W}$ must be a $P$-destabilizing subbundle of $\mathcal{V}|_{C_s}$. Therefore $HN(\mathcal{V}|_{C_s}) \nleq P$. Lemma \ref{lemma: HN filtrations and scalar change} implies that $HN(\mathcal{V}|_{C_s}) = HN(\mathcal{V}|_{C_{\pi(s)}})$, so we must also have $HN(\mathcal{V}|_{C_{\pi(s)}}) \nleq P$. We conclude that $\pi(s) \in T^{\nleq P}$.
\end{proof}
\end{subsection}

\begin{subsection}{Openness of strata} \label{section: opennness of strata}
The following proposition will be used to prove the main theorem of this section.
\begin{prop} \label{prop: strata closed under generalization}
Let $R$ be a discrete valuation ring over $k$. Let $\eta$ (resp. $s$) denote the generic (resp. special) point of $\text{Spec}(R)$. Let $\mathcal{V}$ be a parabolic vector bundle of rank $n$ over $C_R$. Let $P$ be a parabolic HN datum. If $HN \left(\mathcal{V}|_{C_s})\right) \leq P$, then $HN \left( \mathcal{V}|_{C_{\eta}}\right) \leq P$.
\end{prop}
Before proving Proposition \ref{prop: strata closed under generalization}, let's state a lemma that we will use in the proof.
\begin{lemma} \label{lemma: flatness criterion for subsheaf}
Let $R$ be a $k$-algebra. Let $\mathcal{E}$ be a vector bundle on $C_R$. Suppose that $\mathcal{F}$ is a coherent subsheaf of $\mathcal{E}$ such that the quotient $\mathcal{E} \, / \, \mathcal{F}$ is $R$-flat. Then, $\mathcal{F}$ is a vector bundle on $C_R$.
\end{lemma}
\begin{proof}
We want to show that $\mathcal{F}$ is $C_R$-flat. There is a short exact sequence  $0 \longrightarrow \mathcal{F} \longrightarrow \mathcal{E} \longrightarrow \mathcal{E} \, / \, \mathcal{F} \longrightarrow 0$. We can use the Tor long exact sequence at each stalk to conclude that $\mathcal{F}$ is $R$-flat. By Lemma \ref{lemma: critere de platitude} with $X = C_R$ and $B = \text{Spec} \, R$, it suffices to check flatness at each fiber. Choose a point $t \in \text{Spec} \, R$. Proving flatness of $\mathcal{F}|_{C_t}$ amounts to showing that $\mathcal{F}|_{C_t}$ is torsion-free, because $C_t$ is a regular curve. Since $\mathcal{E} \, / \, \mathcal{F}$ is $R$-flat, the fiber $\mathcal{F}|_{C_t}$ remains a subsheaf of $\mathcal{E}|_{C_t}$. Since $\mathcal{E}|_{C_t}$ is torsion-free, we conclude that $\mathcal{F}|_{C_t}$ is torsion-free.
\end{proof}

\begin{proof}[Proof of Proposition \ref{prop: strata closed under generalization}]
For the purpose of this proof, it will be convenient to fix the following notation. For every sheaf $\mathcal{F}$ over  $C_R$, we will write $^{\eta}\mathcal{F} \vcentcolon= \mathcal{F}|_{C_\eta}$ and $^{s}\mathcal{F} \vcentcolon= \mathcal{F}|_{C_s}$. We will use the same notation for parabolic vector bundles. 

Let $\mathcal{V} = \left[ \; \mathcal{E}^{(0)} \,\overset{a_i^{(1)}}{\subset} \, \mathcal{E}_i^{(1)} \overset{a_i^{(2)}}{\subset} \cdots\; \overset{a_{i}^{(N_i)}}{\subset} \, \mathcal{E}_i^{(N_i)}= \mathcal{E}^{(0)}(x_i) \;\right]_{i \in I}$ be a parabolic vector bundle over $C_R$. Suppose that $HN\left(^{\eta}\mathcal{V}\right) \nleq P$, where 
		\[^{\eta}\mathcal{V} = \left[ \; ^{\eta}\mathcal{E}^{(0)} \,\overset{a_i^{(1)}}{\subset} \, ^{\eta}\mathcal{E}_i^{(1)} \overset{a_i^{(2)}}{\subset} \cdots\; \overset{a_{i}^{(N_i)}}{\subset} \, ^{\eta}\mathcal{E}_i^{(N_i)}= \, ^{\eta}\mathcal{E}^{(0)}(x_i) \;\right]_{i \in I} \]
		By Lemma \ref{lemma: existence destabilizing subbundle}, there exists a $P$-destabilizing parabolic subbundle $^{\eta} \mathcal{W} \subset \, ^{\eta} \mathcal{V}$. Set
		\[^{\eta}\mathcal{W} = \left[ \; ^{\eta}\mathcal{F}^{(0)} \,\overset{b_i^{(1)}}{\subset} \, ^{\eta}\mathcal{F}_i^{(1)} \overset{b_i^{(2)}}{\subset} \cdots\; \overset{b_{i}^{(N_i)}}{\subset} \, ^{\eta}\mathcal{F}_i^{(N_i)}= \, ^{\eta}\mathcal{F}^{(0)}(x_i) \;\right]_{i \in I} \]
		Let $^{\eta} \mathcal{Q}^{(m)}_i \vcentcolon = \, ^{\eta}\mathcal{E}_i^{(m)} \, / \,^{\eta} \mathcal{F}_i^{(m)}$,
and set $^{\eta}\mathcal{L}_i^{(m)} \vcentcolon = \, ^{\eta}\mathcal{Q}^{(0)}(x_i) \, / \, ^{\eta}\mathcal{Q}^{(m)}_i$. Let $j$ denote the open immersion $j: C_{\eta} \hookrightarrow C_R$. Define $\mathcal{Q}_i^{(m)} \vcentcolon = \text{Im} \left( \;\mathcal{E}_i^{(m)} \xrightarrow{unit} j_* \, ^{\eta} \mathcal{E}_i^{(m)} \longrightarrow j_* \, ^{\eta} \mathcal{Q}_i^{(m)} \; \right)$. Also, define $\mathcal{F}_i^{(m)} \vcentcolon = \text{Ker} \left( \;\mathcal{E}_i^{(m)} \twoheadrightarrow \mathcal{Q}_i^{(m)} \; \right)$. We have the following commutative diagram with exact columns:
\begin{figure}[H]
\centering
\begin{tikzcd}
    0  & 0 & \cdots & 0 & 0 & 0\\
    \mathcal{Q}^{(0)} \ar [u] \ar[r, symbol= \subset] &  \mathcal{Q}_i^{(1)} \ar [u] \ar[r, symbol= \subset] & \cdots & \mathcal{Q}_i^{(Ni-2)} \ar [u] \ar[r, symbol= \subset] & \mathcal{Q}_i^{(N_i-1)} \ar [u] \ar[r, symbol= \subset] & \mathcal{Q}^{(0)}(x_i) \ar [u]\\
    \mathcal{E}^{(0)} \ar [u] \ar[r, symbol= \subset] &  \mathcal{E}_i^{(1)} \ar [u] \ar[r, symbol= \subset] & \cdots & \mathcal{E}_i^{(Ni-2)} \ar [u] \ar[r, symbol= \subset] & \mathcal{E}_i^{(N_i-1)} \ar [u] \ar[r, symbol= \subset] & \mathcal{E}^{(0)}(x_i) \ar [u]\\
    \mathcal{F}^{(0)} \ar [u] \ar[r, symbol= \subset] &  \mathcal{F}_i^{(1)} \ar [u] \ar[r, symbol= \subset] & \cdots & \mathcal{F}_i^{(Ni-2)} \ar [u] \ar[r, symbol= \subset] & \mathcal{F}_i^{(N_i-1)} \ar [u] \ar[r, symbol= \subset] & \mathcal{F}^{(0)}(x_i) \ar [u]\\
    0 \ar[u]  & 0 \ar[u] & \cdots & 0 \ar[u] & 0 \ar[u] & 0 \ar[u]
\end{tikzcd}
\caption{Diagram 3}
\label{diagram: 3}
\end{figure}
Define $\mathcal{L}_i^{(m)} \vcentcolon = \mathcal{Q}^{(0)}(x_i) \, / \, \mathcal{Q}_i^{(m)}$. Since $j$ is affine, the pushforward functor $j_*$ is exact on quasicoherent sheaves. This implies that $\mathcal{L}_i^{(m)} \subset j_* \, ^{\eta} \mathcal{L}_i^{(m)}$, and hence $\mathcal{L}_i^{(m)}$ is $R$-torsion free. Therefore $\mathcal{L}_i^{(m)}$ is $R$-flat, because $R$ is a discrete valuation ring.

We can use the short exact sequence,
\[ 0 \longrightarrow \mathcal{F}^{(0)}(x_i) \, / \, \mathcal{F}_i^{(m)} \longrightarrow \mathcal{E}^{(0)}(x_i) \, / \, \mathcal{E}_i^{(m)} \longrightarrow  \mathcal{L}_i^{(m)} \longrightarrow 0 \]
and the flatness of $\mathcal{E}^{(0)}(x_i) \, / \, \mathcal{E}_i^{(m)}$ to see that $\mathcal{F}^{(0)}(x_i) \, / \, \mathcal{F}_i^{(m)}$ is $R$-flat for all $i$, $m$. Now we can use the short exact sequence
\[ 0 \longrightarrow  \mathcal{F}_i^{(m)} \, / \, \mathcal{F}_i^{(m-1)} \longrightarrow \mathcal{F}^{(0)}(x_i) \, / \, \mathcal{F}_i^{(m-1)} \longrightarrow \mathcal{F}^{(0)}(x_i) \, / \, \mathcal{F}_i^{(m)} \longrightarrow 0 \]
to conclude that $\mathcal{F}_i^{(m)} \, / \, \mathcal{F}_i^{(m-1)}$ is $R$- flat for all $i$ and $m$. A similar argument shows that $\mathcal{Q}_i^{(m)} \, / \, \mathcal{Q}_i^{(m-1)}$ is $R$ flat. Therefore, Diagram \ref{diagram: 3} remains exact and with horizontal arrows given by inclusions after any base-change with respect to $R$.

Since $\mathcal{Q}_i^{(m)} \subset j_* \, ^{\eta}\mathcal{Q}_i^{(m)}$, we have that $\mathcal{Q}_i^{(m)}$ is $R$-torsion free. This implies that $\mathcal{Q}_i^{(m)}$ is $R$-flat, because $R$ is a discrete valuation ring. By construction we have the short exact sequence
\[ 0 \longrightarrow \mathcal{F}_i^{(m)} \longrightarrow \mathcal{E}_i^{(m)} \longrightarrow \mathcal{Q}_i^{(m)} \longrightarrow 0 \] 
By Lemma \ref{lemma: flatness criterion for subsheaf} we conclude that $\mathcal{F}_i^{(m)}$ is a vector bundle on $C_R$. The discussion above implies that $\mathcal{W} \vcentcolon = \left[ \;\mathcal{F}^{(0)} \,\overset{b_i^{(1)}}{\subset} \, \mathcal{F}_i^{(1)} \overset{b_i^{(2)}}{\subset} \cdots\; \overset{b_{i}^{(N_i)}}{\subset} \,\mathcal{F}_i^{(N_i)}= \, \mathcal{F}^{(0)}(x_i) \;\right]_{i \in I}$ is a parabolic vector bundle on $C_R$. We don't necessarily know that $\mathcal{W}$ is a parabolic subbundle of $\mathcal{V}$, because we have no control on the quotients $\mathcal{Q}_i^{(m)}$. We do however know that $\mathcal{Q}_i^{(m)}$ $R$-flat. Therefore, the base-change to the special fiber $^s \mathcal{W} \hookrightarrow \, ^s \mathcal{V}$ remains a monomorphism.
	
We can now appeal to Lemma \ref{lemma: saturating subsheaves} to obtain a subbundle $^s \mathcal{W}^{sat}$ of $^s\mathcal{V}$ containing the subobject $^s \mathcal{W}$ and satisfying $\mu(^s\mathcal{W}^{sat}) \geq \mu(^s\mathcal{W})$. By local constancy of degree (Lemma \ref{lemma: local constancy of deg}) we have $ \mu(^s\mathcal{W}) = \mu(^{\eta}\mathcal{W})$. Hence $\mu(^s\mathcal{W}^{sat}) \geq \, \mu(^{\eta}\mathcal{W})$. We conclude that $^s\mathcal{W}^{sat}$ is $P$-destabilizing, because $\text{rank} \, ^s\mathcal{W}^{sat} = \text{rank} \, ^{\eta}\mathcal{W}$ and $^{\eta}\mathcal{W}$ is $P$-destabilizing. This implies that $HN\left(^s\mathcal{V}\right) \nleq P$, as desired.
\end{proof}

\begin{remark} \label{remark: equality degree properness proof}
Suppose that $\mu(^s\mathcal{W}^{sat}) = \mu(^s\mathcal{W})$ in the proof above. Then part (b) of Lemma \ref{lemma: saturating subsheaves} implies that $^s\mathcal{\mathcal{W}} = \, ^s\mathcal{W}^{sat}$. In this case $^s\mathcal{Q}_i^{(m)}$ is a vector bundle. Applying Lemma \ref{lemma: critere de platitude} to the $R$-flat sheaf $\mathcal{Q}_i^{(m)}$, we see that $\mathcal{Q}_i^{(m)}$ is a vector bundle over $C_R$. Therefore $\mathcal{W} \subset \mathcal{V}$ is a parabolic subbundle over $C_R$ in this case.
\end{remark}

\begin{prop} \label{prop: openness of locus of strata in families}
Let $T$ be a scheme of finite type over $k$. Let $\mathcal{V}$ a parabolic vector bundle over $C \times S$. Let $P$ be a HN datum. The set $T^{\leq P} \vcentcolon = \left\{ t \in T \, \mid \, HN( \mathcal{V}|_{C_t}) \leq P \right\}$ is Zariski open in $T$.
\end{prop}
\begin{proof}
By Proposition \ref{prop: constructibility of strata}, we know that $T^{\leq P}$ is constructible. In order to prove that it is open we are left to show that it is closed under generalization. Let $p, t \in T$ be two topological points with $t \in \overline{\{p\}}$. Assume that $t \in T^{\leq P}$. By \cite[\href{https://stacks.math.columbia.edu/tag/054F}{Tag 054F}]{stacks-project}, there exists a discrete valuation $k$-algebra $R$ and a morphism $f: \text{Spec}(R) \rightarrow T$ such that $f(\eta) = p$ and $f(s) = t$. (Here we are using the same notation as in Proposition \ref{prop: strata closed under generalization}, so $s$ is the special point and $\eta$ is the generic point). By Lemma \ref{lemma: HN filtrations and scalar change}, we know that $HN(\mathcal{V}|_{C_s}) = HN(\mathcal{V}|_{C_t})$. Since $t \in T^{\leq P}$, we have $HN \left(\mathcal{V}|_{C_t}\right) \leq P$. Hence $HN \left(\mathcal{V}|_{C_s}\right) \leq P$. We can apply Proposition \ref{prop: strata closed under generalization} to deduce that $HN (\mathcal{V}|_{C_{\eta}}) \leq P$. By Lemma \ref{lemma: HN filtrations and scalar change} again, we have $HN(\mathcal{V}|_{C_p}) = (\mathcal{V}|_{C_{\eta}}) \leq P$. So $p \in T^{\leq P}$, as desired.
\end{proof}
\begin{remark}
In particular, Proposition \ref{prop: openness of locus of strata in families} implies that parabolic HN data are upper semicontinuous.
\end{remark}

\begin{thm} \label{thm: openess of strata}
Let $\mathcal{V}$ be a parabolic vector bundle over $C$. Let $P$ be a parabolic HN datum. Then $\text{Bun}_{\mathcal{V}}^{\leq P}$ is an open substack of $\text{Bun}_{\mathcal{V}}$.
\end{thm}
\begin{proof}
The theorem amounts to showing that for any $k$-scheme $T$ and morphism $T \rightarrow \text{Bun}_{\mathcal{V}}$, the base-change $\text{Bun}_{\mathcal{V}}^{\leq P} \times_{\text{Bun}_{\mathcal{V}}} T \; \rightarrow T$ is represented by an open subscheme of $T$. Corollary \ref{coroll: moduli of parabolic vb is algebraic} tells us that $\text{Bun}_{\mathcal{V}}$ is locally of finite type over $k$. So we can assume that $T$ is of finite type over $k$.

A map $T \rightarrow \text{Bun}_{\mathcal{V}}$ is by definition a parabolic vector bundle $\mathcal{W}$ of type $\mathcal{V}$ on $C\times T$. We can describe the base-change $\text{Bun}_{\mathcal{V}}^{\leq P} \times_{\text{Bun}_{\mathcal{V}}} T$ as follows. For any scheme $U$ we have
\[ \text{Bun}_{\mathcal{V}}^{\leq P} \times_{\text{Bun}_{\mathcal{V}}} T \, (U) \; = \; \left\{ \begin{matrix} \text{morphisms} \; f: U \rightarrow T \text{ such that} \\ \text{ for all $u \in U$}, \; HN\left(\mathcal{W} \, |_{C_u}\right)\, \leq P \end{matrix} \right\} \]
For any $u \in U$, Lemma \ref{lemma: HN filtrations and scalar change} implies that $HN\left(\mathcal{W} \, |_{C_u}\right) = HN\left(\mathcal{W} \, |_{C_{f(u)}}\right)$. We conclude that the functor $\text{Bun}_{\mathcal{V}}^{\leq P} \times_{\text{Bun}_{\mathcal{V}}} T$ can be alternatively described by the set theoretic condition
\[ \text{Bun}_{\mathcal{V}}^{\leq P} \times_{\text{Bun}_{\mathcal{V}}} T \, (U) \; = \; \left\{ \begin{matrix} \text{morphisms} \; f: U \rightarrow T \text{ such that} \\ \text{ for all $u \in U$, $f(u) \in T^{\leq P}$} \end{matrix} \right\}  \]
By Proposition \ref{prop: openness of locus of strata in families}, the set $T^{\leq P}$ is open. Hence the functor is represented by $T^{\leq P}$ viewed as an open subscheme of $T$.
\end{proof}

The analogue of Theorem \ref{thm: openess of strata} in the context of vector bundles without parabolic structure when $P=0$ is classical. This amounts to proving that the semistable locus of a family of vector bundles is open. See e.g. \cite{huybrechts.lehn}[Prop. 2.3.1].
\end{subsection}
\end{section}
\begin{section}{Quasicompactness of strata} \label{section: quasicompactness of strata}
\begin{subsection}{Harder-Narasimhan stratification for the moduli of vector bundles} \label{subsection: classical hn stratication}
We recall the Harder-Narasimhan stratification for the classical moduli stack of vector bundles on a curve.
\begin{defn}
The moduli stack of vector bundles of rank $n$ over $C$ is the pseudofunctor $\text{Bun}_{\text{GL}_n}(C)$ from $k$-schemes into groupoids given as follows. For each $k$-scheme $T$, we define $ \text{Bun}_{\text{GL}_n}(C)\, (T) \vcentcolon = \; \left\{ \begin{matrix} \text{groupoid of  rank $n$ vector bundles over $C\times T$} \end{matrix} \right\}$.
\end{defn}
It is well known that $\text{Bun}_{\text{GL}_n}(C)$ is a smooth algebraic stack over $k$. This is just a special case of Proposition \ref{prop: algebraicity torsors over smooth group schemes}.

Recall that the slope of a rank $n$ vector bundle $\mathcal{E}$ over $C$ is defined to be $\mu(\mathcal{E}) = \frac{1}{n} \, \text{deg}\, \mathcal{E}$. A vector bundle is called semistable if all nontrivial subbundles $\mathcal{F} \subset \mathcal{E}$ satisfy $\mu(\mathcal{F}) \leq \mu(\mathcal{E})$. For any given vector bundle $\mathcal{E}$ over $C$, there exists a filtration $0 = \mathcal{E}_0 \,\subset \, \mathcal{E}_1 \, \subset \,\cdots\, \subset \, \mathcal{E}_l = \mathcal{E}$ by subbundles such that $\mathcal{E}_j\, / \,\mathcal{E}_{j-1}$ is semistable for all $j$, and such that for all $1 \leq j \leq l-1$ we have $\mu\left(\mathcal{E}_j\, / \,\mathcal{E}_{j-1} \right) > \mu\left(\mathcal{E}_{j+1}\, / \,\mathcal{E}_{j} \right)$. There is a unique filtration satisfying these conditions; it is called the Harder-Narasimhan filtration.

By a classical HN datum of length $n$ we mean a tuple of $n$ rational numbers $(\mu_1, \mu_2, \cdots, \mu_n) \in \mathbb{Q}^n$ satisfying $\mu_1 \geq \mu_2 \geq \cdots \geq \mu_n$. To any vector bundle $\mathcal{E}$ on $C$ we can associate a HN datum $HN(\mathcal{E})$ given as follows. Let $0 = \mathcal{E}_0 \,\subset \, \mathcal{E}_1 \, \subset \,\cdots\, \subset \, \mathcal{E}_l = \mathcal{E}$
be the Harder-Narasimhan filtration of $\mathcal{E}$. Then, we set
\begin{gather*} HN(\mathcal{E}) = ( \; \underbrace{\mu(\mathcal{E}_1) , \,\cdots, \,\mu(\mathcal{E}_1)}_{ \text{rank} \; \mathcal{E}_1 / 0 \, \text{times}}, \; \underbrace{\mu(\mathcal{E}_2/ \, \mathcal{E}_1) , \,\cdots, \,\mu(\mathcal{E}_2 / \, \mathcal{E}_1)}_{ \text{rank} \; \mathcal{E}_2 / \, \mathcal{E}_1 \, \text{times}}, \cdots, \, \underbrace{\mu(\mathcal{E}_l / \, \mathcal{E}_{l-1}) , \;\cdots, \;\mu(\mathcal{E}_l  / \, \mathcal{E}_{l-1})}_{ \text{rank} \; \mathcal{E}_l / \mathcal{E}_{l-1} \, \text{times}} \; )
\end{gather*}
If $\mathcal{E}$ has rank $n$, then $HN(\mathcal{E}) \in \left(\frac{1}{n!} \mathbb{Z} \right)^n$.

Let $P_1 = (\mu_l)$ and $P_2 = (\nu_l)$ be two classical HN-data of rank $n$. We say that $P_1 \leq P_2$ if the following two conditions are satisfied
\begin{enumerate}[(a)]
    \item $\sum_{l=1}^{n} \mu_l = \sum_{l=1}^{n} \nu_l$.
    \item For all $1 \leq m < n$, we have $\sum_{l=1}^{m} \mu_l \leq \sum_{l=1}^{m} \nu_l$.
\end{enumerate}

\begin{defn}
Let $Q$ be a classical HN datum of rank $n$. We define $\text{Bun}_{\text{GL}_n}^{\leq Q}(C)$ to be the pseudofunctor from $k$-schemes into groupoids given as follows. For any $k$-scheme $T$,
\[  \text{Bun}_{\text{GL}_n}^{\leq Q}(C)\, (T) \vcentcolon = \; \left\{ \begin{matrix} \text{groupoid of  rank $n$ vector bundles $\mathcal{E}$ over $C\times T$ such that}\\
\text{for all $t \in T$, we have $HN(\mathcal{E}|_{C_t}) \leq Q$}\end{matrix} \right\}  \]
\end{defn}

The following is a well known result. See e.g. \cite{behrend-thesis} or \cite{schieder-compactifications} for proofs in the generality of reductive algebraic groups.
\begin{thm} \label{thm: quasicompactness of classical HN strata}
Let $Q$ be a classical HN datum. The subfunctor $\text{Bun}_{\text{GL}_n}^{\leq Q}(C)$ is a quasicompact open substack of $\text{Bun}_{\text{GL}_n}(C)$. \qed
\end{thm}
\end{subsection}

\begin{subsection}{Quasicompactness of parabolic Harder-Narasimhan strata}
Let us now return to our parabolic setting. We have the following natural map of stacks.
\begin{defn}
We define $Forget : \text{Bun}_{\mathcal{V}} \longrightarrow \text{Bun}_{\text{GL}_n} (C)$ to be the map of stacks given as follows. Suppose $T$ is a $k$-scheme and $\mathcal{W} = \left[ \; \mathcal{F}^{(0)} \,\overset{a_i^{(1)}}{\subset} \, \mathcal{F}_i^{(1)} \overset{a_i^{(2)}}{\subset} \cdots\; \overset{a_{i}^{(N_i)}}{\subset} \, \mathcal{F}_i^{(N_i)}= \mathcal{F}^{(0)}(x_i) \;\right]_{i \in I}$ is a parabolic vector bundle in $\text{Bun}_{\mathcal{V}}(T)$. We set $\phi(T) \,(\mathcal{W}) = \mathcal{F}^{(0)}$.
\end{defn}
\begin{prop} \label{prop: properness of forgetful map}
$Forget: \text{Bun}_{\mathcal{V}} \longrightarrow \text{Bun}_{\text{GL}_{n}}(C)$ is schematic and proper.
\end{prop}
In order to prove this proposition, we first recall the definition and properties of generalized flag varieties.
\begin{defn} \label{generalized flag variety}
Let $S$ be a scheme and let $\mathcal{G}$ be a locally free sheaf of constant rank on $S$. Let $(a^{(l)})_l$ be a tuple of nonnegative integers $a^{(1)}, a^{(2)}, \cdots, a^{(N)}$. The generalized flag variety $Flag^{(a^{(l)})_l}(\mathcal{G})$ is defined to be the functor from $S$-schemes to sets given as follows. Let $X$ be an $S$-scheme. Then,
\begin{gather*} Flag^{(a^{(l)})_l}(\mathcal{G})(X) = \; \left\{  \begin{matrix} \text{filtrations of $\mathcal{G}|_{X}$ by vector subbundles} \; \left[ \; 0 \, = \mathcal{P}^{(0)} \, \subset \, \mathcal{P}^{(1)} \subset \cdots\; \subset \mathcal{P}^{(N)} = \mathcal{G}|_{X} \;  \right] \\ 
\\  
\text{such that} \; \mathcal{P}^{(l)} \, / \, \mathcal{P}^{(l-1)} \; \text{is locally free of rank $a^{(l)}$ for every $1 \leq l \leq N$} \end{matrix} \right\} \end{gather*}
The map on morphisms is given by pulling back the filtrations.
\end{defn}

\begin{prop} \label{prop: representability of flag varieties}
Let $S$ be a $k$-scheme and $\mathcal{G}$ a vector bundle of constant rank $n$ on $S$. Let $(a^{(l)})_l$ be a tuple of nonnegative integers. $Flag^{(a^{(l))})_l}(\mathcal{G})$ is represented by a proper scheme over $S$.
\end{prop}
\begin{proof}
Zariski descent for quasicoherent sheaves implies that $Flag^{(a^{(l))})_l}(\mathcal{G})$ is a sheaf in the Zariski topology. In particular, it suffices to check representability by a proper scheme after passing to a Zariski cover of $S$. We can assume that $\mathcal{G}$ is trivial by passing to an open cover. Then $Flag^{(a^{(l))})_l}(\mathcal{G})$ is the base-change of the classical partial flag variety of flags of type $(a^{(l))})_l$ in the vector space $k^n$. This is known to be projective over $k$.
\end{proof}

\begin{proof}[Proof of Proposition \ref{prop: properness of forgetful map}]
Suppose that $\mathcal{V} = \left[ \; \mathcal{E}^{(0)} \,\overset{a_i^{(1)}}{\subset} \, \mathcal{E}_i^{(1)} \overset{a_i^{(2)}}{\subset} \cdots\; \overset{a_{i}^{(N_i)}}{\subset} \, \mathcal{E}_i^{(N_i)}= \mathcal{E}^{(0)}(x_i) \;\right]_{i \in I}$. Let $T$ be a $k$-scheme and $\mathcal{E}$ a vector bundle of rank $n$ on $C \times T$. In order to ease notation, let us write $L \vcentcolon = T \times_{\text{Bun}_{\text{GL}_n}(C)} \text{Bun}_{\mathcal{V}}$. We want to show that $L$ is represented by a proper $T$-scheme. By definition, $L$ is the functor from $T$-schemes into groupoids given as follows. If $S$ is a $T$-scheme, then
\begin{gather*} L(S) = \; \left\{  \begin{matrix} \text{parabolic vector bundles} \; \mathcal{W} = \left[ \; \mathcal{F}^{(0)} \,\overset{a_i^{(1)}}{\subset} \, \mathcal{F}_i^{(1)} \overset{a_i^{(2)}}{\subset} \cdots\; \overset{a_{i}^{(N_i)}}{\subset} \, \mathcal{F}_i^{(N_i)}= \mathcal{F}^{(0)}(x_i) \;\right]_{i \in I} \\  \text{ over $C \times S$ of type $\mathcal{V}$}, \text{ with an isomorphism} \; \phi: \mathcal{E}|_{C \times S} \xrightarrow{\sim} \mathcal{F}^{(0)} \end{matrix} \right\}   \end{gather*}
The isomorphism $\phi$ is part of the data. Automorphisms in the groupoid are required to be compatible with the isomorphism $\phi$. So they must be the identity on $\mathcal{F}^{(0)}$. This in turn implies that they must be the identity as a morphism of parabolic vector bundles. Therefore, $L$ is naturally equivalent to the sheaf of sets:
\begin{gather*} L(S) = \; \left\{  \begin{matrix} \text{parabolic vector bundles} \; \mathcal{W} = \left[ \; \mathcal{E}|_{C\times S} \,\overset{a_i^{(1)}}{\subset} \, \mathcal{F}_i^{(1)} \overset{a_i^{(2)}}{\subset} \cdots\; \overset{a_{i}^{(N_i)}}{\subset} \, \mathcal{F}_i^{(N_i)}= \mathcal{E}(x_i)|_{C\times S} \;\right]_{i \in I} \\  \text{ over $C \times S$ of type $\mathcal{V}$} \end{matrix} \right\} 
\end{gather*}
We claim that this sheaf is represented by the $T$-scheme $\prod_{i \in I} \text{Flag}^{(a_i^{(l)})_l}\left(\mathcal{E}(x_i) \, / \, \mathcal{E}  \,|_{x_i \times T}\right)$. By Proposition \ref{prop: representability of flag varieties}, this claim concludes the proof.

In order to prove the claim, we present two natural transformations that are inverse to each other. We start with $f: L \rightarrow \prod_{i \in I} \text{Flag}^{(a_i^{(l)})_l}\left(\mathcal{E}(x_i) \, / \, \mathcal{E} \,|_{x_i \times T} \right)$ given as follows. For a $T$-scheme $S$, and a parabolic bundle $\mathcal{W} = \left[ \; \mathcal{E}|_{C\times S} \,\overset{a_i^{(1)}}{\subset} \, \mathcal{F}_i^{(1)} \overset{a_i^{(2)}}{\subset} \cdots\; \overset{a_{i}^{(N_i)}}{\subset} \, \mathcal{F}_i^{(N_i)}= \mathcal{E}(x_i)|_{C\times S} \;\right]_{i \in I}$ in $L(S)$, we define
\begin{gather*}f\left(\mathcal{W} \right) \; \vcentcolon= \; \left( \; 0 \,\overset{a_i^{(1)}}{\subset} \, \left(\mathcal{F}_i^{(1)} \, / \, \mathcal{E}|_{C\times S}\right)|_{x_i \times S} \overset{a_i^{(2)}}{\subset} \cdots\; \overset{a_{i}^{(N_i)}}{\subset} \, \left(\mathcal{F}_i^{(N_i)} \, / \, \mathcal{E}|_{C\times S}\right)|_{x_i \times S}= \left(\mathcal{E}(x_i) \, / \, \mathcal{E} \right)|_{x_i \times S} \; \right)_{i \in I}
\end{gather*}
Let $q_i: x_i \rightarrow C$ denote the inclusion of the closed point $x_i$ into $C$. The inverse $g: \prod_{i \in I} \text{Flag}^{(a_i^{(l)})_l}\left(\mathcal{E}(x_i) \, / \, \mathcal{E} \,|_{x_i \times T} \right) \rightarrow L$ is defined as follows. Let $S$ be a $T$-scheme. Let $(fl_i)_{i \in I}$ be at tuple of flags in $\prod_{i \in I} \text{Flag}^{(a_i^{(l)})_l}\left(\mathcal{E}(x_i) \, / \, \mathcal{E}  \,|_{x_i \times T}\right)(S)$. Suppose that $fl_i$ is given by a chain of vector bundles over $S$
\[ fl_i = \; \left[ \; 0 \,\overset{a_i^{(1)}}{\subset} \, \mathcal{P}_i^{(1)} \overset{a_i^{(2)}}{\subset} \cdots\; \overset{a_{i}^{(N_i)}}{\subset} \mathcal{P}_i^{(N_i)} = \left(\mathcal{E}(x_i) \, / \, \mathcal{E} \right)|_{x_i \times S} \;  \right] \]
We define $g( \,(fl_i)_{i \in I} \, ) = \left[ \; \mathcal{E}|_{C\times S} \,\overset{a_i^{(1)}}{\subset} \, \mathcal{F}_i^{(1)} \overset{a_i^{(2)}}{\subset} \cdots\; \overset{a_{i}^{(N_i)}}{\subset} \, \mathcal{F}_i^{(N_i)}= \mathcal{E}(x_i)|_{C\times S} \;\right]_{i \in I}$, where 
\begin{gather*}\mathcal{F}^{(l)}_i \vcentcolon = \text{Ker} \left( \; \mathcal{E}(x_i)|_{C \times S} \xrightarrow{unit} \, (q_{i}\times id_S)_* (q_i\times id_S)^{*} \mathcal{E}(x_i)|_{C\times S} \, \twoheadrightarrow \, (q_i \times id_S)_* \left(\mathcal{E}(x_i)|_{x_i \times S} \, / \, \mathcal{P}^{(l)}_i\right) \; \right)
\end{gather*}
These functors are inverses of each other by construction. The only nontrivial thing to check is that $g$ well defined (i.e. $g(\,(fil_i)_{i \in I}\,)$ is a parabolic bundle). In order to show this, we need to check that the $\mathcal{F}_i^{(l)}$s are vector bundles on $C \times S$. Each sheaf $\mathcal{F}^{(l)}_i$ is locally finitely presented over $C \times S$ and flat over $S$, because it fits into the short exact sequence $0 \rightarrow \mathcal{E}|_{C \times S} \rightarrow  \mathcal{F}^{(l)}_i \rightarrow (q_i \times id_S)_*\left( \mathcal{P}^{(l)}_i \right)\,   \rightarrow 0$ and the other two sheaves in the sequence are finitely presented over $C \times S$ and flat over $S$. By Lemma \ref{lemma: critere de platitude}, it suffices to show that for every point $s \in S$ we have that $\mathcal{F}_i^{(l)}|_{C_s}$ is flat over $C_s$. Since $C_s$ is a smooth curve, this is equivalent to proving that $\mathcal{F}_i^{(l)}|_{C_s}$ is torsion free. Since $\mathcal{E}(x_i)|_{C \times S} \, / \, \mathcal{F}^{(l)}_i \, \cong \, (q_i \times id_S)_* \left( \mathcal{E}|_{x_i \times S} \, / \, \mathcal{P}_i^{(l)} \right)$ is $S$-flat, the inclusion $\mathcal{F}_i^{(l)} \subset \mathcal{E}(x_i)|_{C \times S}$ remains a monomorphism after passing to the fiber $C_s$. Hence $\mathcal{F}_i^{(l)}|_{C_s}$ is a subsheaf of the locally free sheaf $\mathcal{E}(x_i)|_{C_s}$. We conclude that $\mathcal{F}_i^{(l)}|_{C_s}$ is torsion free, as desired.
\end{proof}

We want to leverage the quasicompactness of the classical Harder-Narasimhan strata for the moduli of vector bundles in order to prove quasicompactness of HN strata for parabolic vector bundles. As a first step, we need to understand how the map $Forget$ interacts with the strata in the source and target. This is achieved by the following lemma.
\begin{lemma} \label{lemma: finiteness of HN under forgetful map}
Let $P$ be a parabolic HN datum. There exists a finite set $F(P) \subset \frac{1}{n!} \mathbb{Z}$ of classical HN data of rank $n$ such that for all fields $K \supset k$ and all parabolic vector bundles $\mathcal{W} \in \text{Bun}_{\mathcal{V}}^{ \leq P} (K)$, the classical HN datum of the vector bundle $Forget(\mathcal{W})$ belongs to $F(P)$.
\end{lemma}
\begin{proof}
Suppose that $P = (\mu_1, \mu_2, \cdots, \mu_n)$. Let $K \supset k$ be a field extension. Let $\mathcal{W} = \left[ \; \mathcal{F}^{(0)} \,\overset{a_i^{(1)}}{\subset} \, \mathcal{F}_i^{(1)} \overset{a_i^{(2)}}{\subset} \cdots\; \overset{a_{i}^{(N_i)}}{\subset} \, \mathcal{F}_i^{(N_i)}= \mathcal{F}^{(0)}(x_i) \;\right]_{i \in I}$ be a parabolic vector bundle of type $\mathcal{V}$ over $C_K$. Let $HN(\mathcal{W}) = (\nu_1, \nu_2, \cdots , \nu_n)$. Suppose that $HN(\mathcal{W}) \leq P$. This means in particular that $\mu_1 \geq \nu_1$. Let $0 = \, _0\mathcal{F}^{(0)} \subset \, _1\mathcal{F}^{(0)} \, \subset \, _2\mathcal{F}^{(0)} \, \subset \cdots \, \subset \, _l\mathcal{F}^{(0)} = \mathcal{F}^{(0)}$ be the (classical) Harder-Narasimhan filtration for $\mathcal{F}^{(0)}$. Suppose that the classical Harder-Narasimhan datum for $\mathcal{F}^{(0)}$ is given by $HN(\mathcal{F}^{(0)}) = (\xi_1, \xi_2, \cdots, \xi_n)$.

Define $_1\mathcal{F}_i^{(m)} \vcentcolon = \, _1\mathcal{F}^{(0)}(x_i) \, \cap \, \mathcal{F}_i^{(m)}$. The diagram of sheaves $_1\mathcal{W}$ given by
\[_1\mathcal{W} = \left[ \; _1\mathcal{F}^{(0)} \,\overset{a_i^{(1)}}{\subset} \, _1\mathcal{F}_i^{(1)} \overset{a_i^{(2)}}{\subset} \cdots\; _1\overset{a_{i}^{(N_i)}}{\subset} \, _1\mathcal{F}_i^{(N_i)}= _1\mathcal{F}^{(0)}(x_i) \;\right]_{i \in I}\]
is a parabolic subbundle of $\mathcal{W}$, because any subsheaf of a locally free sheaf on $C_K$ is locally free. By the proof of Proposition \ref{prop: existence and uniqueness of HN filtrations}, we know that $\nu_1$ is the maximal slope among subbundles of $\mathcal{W}$. In particular, $\mu(_1\mathcal{W}) \leq \nu_1 \leq \mu_1$. By Lemma \ref{lemma: deg of parabolic vs regular vector bundles}, we know that $\mu(_1\mathcal{F}^{(0)}) \leq \mu(_1\mathcal{W})$. So we have $\mu_1 \geq \xi_1$.

Next we will bound $\xi_n$. By Lemma \ref{lemma: deg of parabolic vs regular vector bundles}, we have $\text{deg} \, \mathcal{F}^{(0)} \geq \text{deg} \, \mathcal{W} - n|I|$. By additivity of degree for vector bundles in short exact sequences, we have $\text{deg}\, \mathcal{F}^{(0)} = \sum_{j=1}^{n} \xi_j$. Similarly, additivity of degree for parabolic bundles gives $\text{deg} \, \mathcal{W} = \sum_{j=1}^n \nu_j$. Replacing these in the inequality above yields $\sum_{j=1}^{n} \xi_j \geq \sum_{j=1}^n \nu_j  -n|I|$. But we know that $\xi_1 \geq \xi_2 \geq \cdots \geq \xi_n$. Therefore, $\sum_{j=1}^{n} \xi_j \leq (n-1) \xi_1 + \xi_n$. Also, we know that $HN(\mathcal{W}) \leq P$. This means in particular that $\sum_{j=1}^n \nu_j = \sum_{j=1}^n \mu_j $. Hence the inequality above becomes $(n-1) \xi_1 + \xi_n \geq \sum_{j=1}^n \mu_j -n|I|$. 

We have seen that $\xi_1 \leq \mu_1$. So we can rearrange the last inequality to obtain $\xi_n \geq \sum_{j=1}^n \mu_j -n|I|-(n-1)\mu_1$. Hence, we get the uniform bounds
\[ \mu_1 \geq \xi_1 \geq \xi_2 \geq \cdots \geq \xi_n \geq \sum_{j=1}^n \mu_j -n|I|-(n-1)\mu_1.\]
Since the $\xi_j$s must lie in the lattice $\frac{1}{n!} \mathbb{Z}$, there are finitely many possibilities for $HN(\mathcal{F}^{(0)}) = (\xi_1, \xi_2, \cdots , \xi_n)$.
\end{proof}

\begin{coroll} \label{coroll: factoring of forgetful map}
Let $P$ a parabolic HN datum. There exists finitely many classical HN data ${Q_j}$ such that the restriction $Forget: \text{Bun}_{\mathcal{V}}^{\leq P} \longrightarrow \text{Bun}_{\txt{GL}_n}(C)$ factors through the open inclusion $\cup_j \text{Bun}_{\text{GL}_n}^{\leq Q_j} (C) \hookrightarrow \text{Bun}_{\text{GL}_n}(C)$.
\end{coroll}
\begin{proof}
We can take the finite set $\{Q_j\}$ to be $F(P)$ as in Lemma \ref{lemma: finiteness of HN under forgetful map} above. It suffices to show that for any $k$-scheme $T$ and a vector bundle $\mathcal{E}$ of rank $n$ on $C \times T$, the fiber product $Forget: \text{Bun}_{\mathcal{V}}^{\leq P} \times_{\text{Bun}_{\text{GL}_n}(C)} \, T \, \rightarrow T$ factors through the open subset $\cup_j \text{Bun}_{\text{GL}_n}^{\leq Q_j} (C) \times_{\text{Bun}_{\text{GL}_n}(C)} \, T$ of $T$.

In order to show that it factors through an open subset, it suffices to show that it does so set theoretically. So we can check it at the level of points. We have to show that if $t \in T$ is such that $\mathcal{E}|_{C_t}$ does not satisfy $HN(\mathcal{E}|_{C_t}) \leq Q_j$ for some $Q_j \in F(P)$, then there is no field extesion $K \supset \kappa(t)$ and parabolic vector bundle $\mathcal{W} \in \text{Bun}_{\mathcal{V}}^{\leq P}(K)$ satisfying $Forget(\mathcal{W}) \cong \mathcal{E}|_{C_K}$. This is just a restatement of Lemma \ref{lemma: finiteness of HN under forgetful map}, because the classical HN type of a vector bundle is preserved under scalar extension (see \cite{langton} for a proof of this last fact).
 \end{proof}
\begin{thm} \label{thm: quasicompactness of strata}
For any parabolic HN datum $P$, the open substack $\text{Bun}_{\mathcal{V}}^{\leq P}$ is quasicompact.
\end{thm}
\begin{proof}
By Corollary \ref{coroll: factoring of forgetful map}, there are finitely many classical HN data $Q_j$ such that the map $Forget$ factors as $\text{Bun}_{\mathcal{V}}^{\leq P} \longrightarrow \cup_j \text{Bun}_{\text{GL}_n}^{\leq Q_j} (C)$. By Proposition \ref{prop: properness of forgetful map}, this map is the composition of an open immersion and a proper schematic map. Since everything is locally of finite type over $k$, this map is quasicompact. Now Theorem \ref{thm: quasicompactness of classical HN strata} tells us that $\cup_j \text{Bun}_{\text{GL}_n}^{\leq Q_j} (C)$ is a finite union of quasicompact open substacks of $\text{Bun}_{\text{GL}_n} (C)$. Therefore $\cup_j \text{Bun}_{\text{GL}_n}^{\leq Q_j} (C)$ is itself a quasicompact algebraic stack. We conclude that $\text{Bun}_{\mathcal{V}}^{\leq P}$ is quasicompact.
\end{proof}
\end{subsection}
\end{section}

\begin{section}{Completeness of strata} \label{section: completeness of strata}
In this section we deal with discrete valuation rings $R$ over $k$. We always write $\eta$ (resp. $s$) for the generic (resp. special) point of $\text{Spec}(R)$. We will also work with the base-change $C_R$. We denote by $j: C_{\eta} \hookrightarrow C_R$ the open immersion of the generic fiber, and $\iota: C_s \hookrightarrow C_R$ the closed immersion of the special fiber.
\begin{defn} \label{defn: completeness of stacks}
Let $\mathfrak{X}$ be a stack of finite type over $k$. We say that $\mathcal{X}$ is complete over $k$ if for all complete discrete valuation $k$-algebras $R$, it satisfies the following lifting criterion. For every map $f: \eta \rightarrow \mathfrak{X}$ there exists a morphism $g: \text{Spec} \, R \rightarrow \mathfrak{X}$ making the following diagram commute
\[\xymatrix{
 \text{Spec} \, R \ar[dr]^{g} \\
\eta \ar[u] \ar[r]_{f} & \mathfrak{X} } \] 
\end{defn}
 The goal of this section is to show Theorem \ref{thm: completeness of strata}, which states that all the HN-strata $\text{Bun}_{\mathcal{V}}^{\leq P}$ are complete over $k$. Let us briefly describe some of the work done in this direction. Mehta and Seshadri \cite{mehta-seshadri} prove completeness of the moduli of parabolic vector bundles that are semistable of parabolic degree $0$. On the other hand, Heinloth \cite[Remark 3.20]{heinloth-hilbertmumford} notes that the valuative criterion for universal closedness \cite[\href{https://stacks.math.columbia.edu/tag/0CLK}{Tag 0CLK}]{stacks-project} can be proven for the semistable locus of the moduli stack of torsors for a parahoric Bruhat-Tits group scheme (under some mild tameness conditions on the generic fiber). Recall that the semistable locus consits of the union of the minimal open strata. In our case these are given by $\text{Bun}_{\mathcal{V}}^{\leq P}$, where $P=(r, r, r, \cdots , r)$ for some real number $r$. We note that this definition of completeness is slightly stronger than the valuative criterion for universal closedness, because we do not require passing to an extension of the fraction field. 

 We extend the proof in \cite{mehta-seshadri} so that it applies to unstable strata of arbitrary degree. First we need a series of lemmas.
\begin{lemma} \label{lemma: finiteness of HN datum less than}
Let $Q$ be a parabolic HN datum of rank $n$. Consider the set $B_Q$ of parabolic HN data $P$ satisfying
\begin{enumerate}[(a)]
    \item $P \leq Q$
    \item $P = HN( \mathcal{W})$ for some $\mathcal{W}$ of rank $n$ in $\text{Vect}_{\overline{\lambda}}$
\end{enumerate}
The set $B_Q$ is finite.
\end{lemma}
\begin{proof}
Set $Q = (\xi_1, \xi_2, \cdots, \xi_n)$. Choose $P = ( \mu_1, \mu_2, \cdots , \mu_n) \in B_Q$. Define the finite set of real numbers
\[ X \vcentcolon = \left\{ -\sum_{i \in I} \sum_{m =1}^{N_i} b_i^{(m)} \lambda_i^{(m)} \; \mid \; 0 \leq b_i^{(m)} \leq n \text{\; for $i \in I$ and $1 \leq m \leq N_i$} \right\} \]
Let $\mathcal{W}$ in $\text{Vect}_{\overline{\lambda}}$ such that $HN(\mathcal{W}) = P$. Fix $1 \leq k \leq n$. By definition, we have that $\mu_k = \mu(\mathcal{P})$ for some parabolic subbundle $\mathcal{P} \subset \mathcal{W}$. If we have
\[\mathcal{P} = \left[ \; \mathcal{F}^{(0)} \,\overset{a_i^{(1)}}{\subset} \, \mathcal{F}_i^{(1)} \overset{a_i^{(2)}}{\subset} \cdots\; \overset{a_{i}^{(N_i)}}{\subset} \, \mathcal{F}_i^{(N_i)}= \mathcal{F}^{(0)}( x_i) \;\right]_{i \in I}\]
then by definition $\mu_k = \frac{1}{\text{rank}\, \mathcal{P}} \left(\text{deg} \; \mathcal{F}^{(0)} + \sum_{i \in I} \left(n - \sum_{j=1}^{N_i} \lambda_i^{(j)} \, b_i^{(j)} \right) \, \right)$. So $\mu_k$ belongs to the discrete subset of real numbers $\frac{1}{n!}\left(\mathbb{Z} + X \right)$. Since $P \leq Q$, we have $\mu_1 \leq \xi_1$. By definition, we must also have $\sum_{j =1}^n \xi_j = \sum_{j = 1}^n \mu_j$. Since $\mu_j \leq \mu_1$ for all $j$, we know that $\sum_{j = 1}^{n-1} \mu_j \leq (n-1)\mu_1 \leq (n-1) \xi_1$. So we get the inequality
\[ \sum_{j =1}^n \xi_j = \sum_{j = 1}^n \mu_j \, \leq  \, (n-1) \xi_1 + \mu_n \]
We conclude that we have the uniform bounds $\sum_{j =1}^n \xi_j - (n-1) \xi_1  \, \leq \,  \mu_k \, \leq \, \xi_1$ for all $k$. Since $\mu_k$ belongs to a discrete set, this shows that there are finitely many possible values for each $\mu_k$.
\end{proof}

In order to prove completeness, we need to be able to extend a given parabolic vector bundle on $C_{\eta}$. This is what the next lemma achieves.
\begin{lemma} \label{lemma: extending parabolic discrete valuation}
Let $R$ be a discrete valuation $k$-algebra. Let $\mathcal{P}$ be a parabolic vector bundle over $C_{\eta}$. There exists a parabolic vector bundle $\mathcal{W}$ over $C_R$ such that $\mathcal{W}|_{C_{\eta}} \cong \mathcal{P}$.
\end{lemma}
\begin{proof}
Suppose that $\mathcal{P} = \left[ \; \mathcal{P}^{(0)} \,\overset{a_i^{(1)}}{\subset} \, \mathcal{P}_i^{(1)} \overset{a_i^{(2)}}{\subset} \cdots\; \overset{a_{i}^{(N_i)}}{\subset} \, \mathcal{P}_i^{(N_i)}= \mathcal{P}^{(0)}(x_i) \;\right]_{i \in I}$. It is well known that exists a vector bundle $\mathcal{F}^{(0)}$ over $C_R$ such that $\mathcal{F}^{(0)}|_{C_{\eta}} \cong \mathcal{P}^{(0)}$ (for example see Proposition 6 in \cite{langton}). Define $\mathcal{F}^{(m)}_i$ to be the kernel
\[ \mathcal{F}^{(m)}_i \vcentcolon = \text{Ker} \, \left( \mathcal{F}^{(0)}(x_i) \xrightarrow{unit} j_* \mathcal{P}^{(0)}( x_i) \twoheadrightarrow j_* \left(\mathcal{P}^{(0)}( x_i) \, / \, \mathcal{P}^{(m)}_i \right) \,  \right) \]
We have that $\mathcal{F}^{(0)}( x_i) \, / \, \mathcal{F}_i^{(m)} \, \subset \, j_* \left(\mathcal{P}^{(0)}(x_i) \, / \, \mathcal{P}^{(m)}_i \right)$ is $R$-torsion free, and hence $R$-flat. By Lemma \ref{lemma: flatness criterion for subsheaf}, $\mathcal{F}_i^{(m)}$ is a vector bundle on $C_R$. Also, we have a short exact sequence
\[ 0 \longrightarrow \,  \mathcal{F}^{(m)}_i \, / \,  \mathcal{F}^{(m-1)}_i \, \longrightarrow \,  \mathcal{F}^{(0)}( x_i) \, / \, \mathcal{F}^{(m-1)}_i \, \longrightarrow \,  \mathcal{F}^{(0)}( x_i) \, / \, \mathcal{F}^{(m)}_i  \, \longrightarrow 0 \]
By the $R$-flatness of $\mathcal{F}^{(0)}(x_i) \, / \, \mathcal{F}^{(m-1)}_i$ and $\mathcal{F}^{(0)}( x_i) \, / \, \mathcal{F}^{(m)}_i$ plus the Tor long exact sequence at each stalk we see that $\mathcal{F}^{(m)}_i \, / \, \mathcal{F}^{(m-1)}_i$ is $R$-flat. We conclude that $\mathcal{W} \vcentcolon = \left[ \; \mathcal{F}^{(0)} \,\overset{a_i^{(1)}}{\subset} \, \mathcal{F}_i^{(1)} \overset{a_i^{(2)}}{\subset} \cdots\; \overset{a_{i}^{(N_i)}}{\subset} \, \mathcal{F}_i^{(N_i)}= \mathcal{F}^{(0)}( x_i) \;\right]_{i \in I}$ is a parabolic vector bundle over $C_R$. By construction $\mathcal{W}|_{C_{\eta}} \cong \mathcal{P}$.
\end{proof}

\begin{lemma} \label{lemma: HN datum decreases ses}
Let $\mathcal{W}$ be a parabolic vector bundle in $\text{Vect}_{\overline{\lambda}}$. Let $\mathcal{U}$ be the maximal destabilizing parabolic subbundle of $\mathcal{W}$. Set $\mathcal{Q} \vcentcolon = \mathcal{W} \, / \, \mathcal{U}$. Suppose that there is a parabolic vector bundle $\mathcal{W}_1$ fitting in a short exact sequence $0 \longrightarrow \mathcal{Q} \longrightarrow \mathcal{W}_1 \longrightarrow \mathcal{U} \longrightarrow 0$. Then, $HN(\mathcal{W}_1) \leq HN(\mathcal{W})$.
\end{lemma}
\begin{proof}
Set $P = (\mu_1, \cdots, \mu_n) \vcentcolon = HN(\mathcal{W})$. Suppose for the sake of contradiction that $HN(\mathcal{W}_1) \nleq P$. The short exact sequence above show that $\text{deg}\, \mathcal{W}_1 = \text{deg} \, \mathcal{W} = \sum_{j = 1}^n \mu_j$. So we can use Lemma \ref{lemma: existence destabilizing subbundle} to conclude that there is a $P$-destabilizing parabolic subbundle $\mathcal{M}$ of $\mathcal{W}_1$. Let us denote by $\mathcal{L}$ the image of the composition $\mathcal{M} \rightarrow \mathcal{W}_1 \twoheadrightarrow \mathcal{U}$. We have a short exact sequence  $0 \longrightarrow \mathcal{Q} \, \cap \, \mathcal{M} \longrightarrow \mathcal{M} \longrightarrow \mathcal{L} \longrightarrow 0$ of parabolic vector bundles. Since $\mathcal{U}$ is semistable and $\mathcal{L} \subset \mathcal{U}$ is a parabolic subbundle, we have $\mu(\mathcal{L}) \leq \mu(\mathcal{U})$. We conclude that $\deg(\mathcal{L}) \leq \text{rank} \, \mathcal{L} \cdot \mu_1 = \sum_{j =1}^{\text{rank} \, \mathcal{L}} \mu_j$.

We can also look at the decompostion of $\mathcal{Q} \, \cap \, \mathcal{M}$ obtained by intersecting with the Harder-Narasimahn filtration of $\mathcal{Q}$. By a similar argument we conclude that $\text{deg} \, (\mathcal{Q} \, \cap \, \mathcal{M}) \leq \sum_{j = \text{rank}\, \mathcal{U} +1}^{\text{rank} \, (\mathcal{Q} \, \cap \, \mathcal{M}) + \text{rank}\, \mathcal{U}} \mu_j$. We can put these together to obtain
\[ \text{deg}\, \mathcal{M} \leq  \sum_{j =1}^{\text{rank} \, \mathcal{L}} \mu_j \, + \, \sum_{j = \text{rank}\, \mathcal{U} +1}^{\text{rank} \, (\mathcal{Q} \, \cap \, \mathcal{M}) + \text{rank}\, \mathcal{U}} \mu_j \]
Since the $\mu_j$ are decreasing, the right hand side is less than or equal to $\sum_{j =1}^{\text{rank} \, \mathcal{M}} \mu_j$. By definition, this contradicts the fact that $\mathcal{M}$ is $P$-destabilizing.
\end{proof}





\begin{prop} \label{prop: lifting criterion for strata}
Let $R$ be a complete discrete valuation ring over $k$. Let $P$ be a HN datum. Let $\mathcal{P}$ be a parabolic vector bundle of rank $n$ over $C_{\eta}$. Suppose that $HN (\mathcal{P}) = P$. Then there exists a parabolic bundle $\mathcal{W}$ over $C_R$ such that $\mathcal{W}|_{C_{\eta}} \cong \mathcal{P}$ and $HN \left(\mathcal{W}|_{C_s} \right) = P$.
\end{prop}
\begin{proof}
Suppose $\mathcal{P}$ has Harder-Narasimhan filtration $0 \subset \,  \mathcal{P}_1 \,  \subset \, \cdots \, \subset \, \mathcal{P}_l = \mathcal{P}$. By Lemma \ref{lemma: extending parabolic discrete valuation}, there exists a parabolic vector bundle $\mathcal{W}$ such that $\mathcal{W}|_{C_{\eta}} \cong \mathcal{P}$. Let $\overline{\mathcal{U}}$ be the maximal destabilizing parabolic subbundle of $\mathcal{W}|_{C_s}$. Write $\overline{\mathcal{Q}} \vcentcolon = \mathcal{W}|_{C_s} \, / \, \overline{\mathcal{U}}$. We have a map $g$ of diagrams of sheaves given by the composition $g: \, \mathcal{W} \dhxrightarrow{unit} \iota_* \mathcal{W}|_{C_s} \twoheadrightarrow  \iota_* \overline{\mathcal{Q}}$.

Define $\mathcal{W}_1 \vcentcolon = \text{Ker}(g)$. We have a short exact sequence of diagrams of sheaves $0 \longrightarrow \mathcal{W}_1 \longrightarrow \mathcal{W} \longrightarrow \iota_* \overline{\mathcal{Q}} \longrightarrow 0$.
The pullback $\mathcal{W}_1|_{C_{\eta}} \hookrightarrow \mathcal{W}|_{C_{\eta}}$ becomes an isomorphism. In order to pass to the special fiber we use the Tor long exact sequence for each short exact sequence of sheaves in the diagrams. We get an exact sequence of diagrams of sheaves over the special fiber $0 \longrightarrow \overline{\mathcal{Q}} \longrightarrow \mathcal{W}_1|_{C_s} \longrightarrow \mathcal{W}|_{C_s} \longrightarrow \overline{\mathcal{Q}} \longrightarrow 0$. To see this recall that $\text{Tor}^1_{\mathcal{O}_s} ( \mathcal{O}_s, \, \mathcal{M})$ is given by the $R$-torsion submodule of $\mathcal{M}$ for any $R$-module $M$. This can be applied to the stalks of the short exact sequence of sheaves above.

We therefore get short exact sequences
\[ 0 \longrightarrow \overline{\mathcal{Q}} \longrightarrow \mathcal{W}_1|_{C_s} \longrightarrow \overline{\mathcal{U}} \longrightarrow 0  \]
\[ 0 \longrightarrow \overline{\mathcal{U}} \longrightarrow \mathcal{W}|_{C_s} \longrightarrow \overline{\mathcal{Q}} \longrightarrow 0  \]
Using this we see that $\mathcal{W}_1|_{C_s}$ is a vector bundle over $C_s$. All the sheaves in the diagram $\mathcal{W}_1 \hookrightarrow \mathcal{W}$ are $R$-torsion free. Therefore they are $R$-flat. We can then use Lemma \ref{lemma: critere de platitude} and an argument similar to the one in Lemma \ref{lemma: extending parabolic discrete valuation} to conclude that $\mathcal{W}_1$ is a parabolic vector bundle. Lemma \ref{lemma: HN datum decreases ses} implies that  $HN(\mathcal{W}_1|_{C_s}) \leq HN(\mathcal{W}|_{C_s})$.

Now we can replace $\mathcal{W}$ with $\mathcal{W}_1$ and iterate the same construction. By Lemma \ref{lemma: finiteness of HN datum less than}, the Harder-Narasimhan datum $HN(\mathcal{W}|_{C_S})$ can't decresase indefinitely. So it must eventually stabilize. Therefore, we can assume without loss of generality that $HN(\mathcal{W}_1|_{C_s}) = HN(\mathcal{W}|_{C_s})$ in the construction above. Let $\overline{\mathcal{U}}_1$ be the maximal destabilizing parabolic subbundle of $\mathcal{W}_1|_{C_s}$. Since $HN(\mathcal{W}_1|_{C_s}) = HN(\mathcal{W}|_{C_s})$, we must have that $\overline{\mathcal{U}}$ and $\overline{\mathcal{U}}_1$ have the same rank and slope. By Lemma \ref{lemma: existence maximal destibilizing subbundle}, we see that $\overline{\mathcal{U}}_1$ maps isomorphically to $\overline{\mathcal{U}}$ under the surjection $\mathcal{W}_1 \twoheadrightarrow \overline{\mathcal{U}}$. So the short exact sequence $0 \longrightarrow \overline{\mathcal{Q}} \longrightarrow \mathcal{W}_1|_{C_s} \longrightarrow \overline{\mathcal{U}} \longrightarrow 0$ splits. Hence we see that $\mathcal{W}_1|_{C_s} \cong \overline{\mathcal{U}} \oplus \overline{\mathcal{Q}}$. 

Repeating this construction for $\mathcal{W}_1$, we can obtain a parabolic suboject $\mathcal{W}_2 \hookrightarrow \mathcal{W}_1$ such that $\mathcal{W}_2|_{C_s} \cong \overline{\mathcal{U}} \oplus \overline{\mathcal{Q}}$. We can iterate this construction to get an infinite chain of parabolic subojects
\[ \cdots \, \mathcal{W}_n \, \xhookrightarrow{f_n} \, \mathcal{W}_{n-1} \, \xhookrightarrow{f_{n-1}} \, \cdots \, \mathcal{W}_1 \, \xhookrightarrow{f_1} \mathcal{W}  \]
such that $\mathcal{W}_n|_{C_s} \cong \overline{\mathcal{U}} \oplus \overline{\mathcal{Q}}$. The pullback $\iota^{*} f_n: \mathcal{W}_n|_{C_s} \rightarrow \mathcal{W}_{n-1}|_{C_s}$ induces an isomorphims between the corresponding maximal destabilizing parabolic subbundles $\overline{\mathcal{U}}$. The proof of Lemma 3.5 in \cite{mehta-seshadri} now implies that there exists a parabolic subbundle $\mathcal{U}$ of $\mathcal{W}$ over $C_R$ such that $\mathcal{U}|_{C_s}$ is the maximal destabilizing parabolic subbundle $\overline{\mathcal{U}}$ of $\mathcal{W}|_{C_s}$. By local constancy of slope we know that $\mu(\mathcal{U}|_{C_{\eta}}) = \mu(\overline{\mathcal{U}})$. But $\mathcal{U}|_{C_{\eta}}$ is a parabolic subbundle of $\mathcal{P}$. Hence the definition of the maximal destabilizing parabolic subbundle $\mathcal{P}_1$ of $\mathcal{P}$ implies that $\mu(\mathcal{U}|_{C_{\eta}}) \leq \mu(\mathcal{P}_1)$. On the other hand, we have that $HN(\mathcal{P}) \leq HN(\mathcal{W}|_{C_s})$ by Proposition \ref{prop: strata closed under generalization}. So we must actually have $\mu(\overline{U}) = \mu(\mathcal{P}_1)$. The definition of the maximal destabilizing parabolic subbundle $\mathcal{P}_1$ of $\mathcal{P}$ now implies that $\text{rank} \, \mathcal{U}|_{C_{\eta}} \leq \text{rank} \, \mathcal{P}_1$. We can use again the fact that $HN(\mathcal{P}) \leq HN(\mathcal{W}|_{C_s})$ to conclude that we must actually have $\text{rank} \, \mathcal{U}|_{C_{\eta}} = \text{rank} \, \overline{\mathcal{U}} =  \text{rank} \, \mathcal{P}_1$. Hence $\mathcal{U}|_{C_{\eta}} = \mathcal{P}_1$, and so the subbundle $\mathcal{U}$ restricts to the maximal destabilizing parabolic subbundle at both the special and the generic fiber.

Consider now the parabolic quotient $\mathcal{M} \vcentcolon = \mathcal{W} \, / \, \mathcal{U}$. We can repeat the same construction as above to obtain a parabolic subobject $\mathcal{K} \hookrightarrow \mathcal{M}$ such that
\begin{enumerate}
    \item $\mathcal{K}|_{C_{\eta}} \cong \mathcal{P} \, / \, \mathcal{P}_1$.
    \item There exists a parabolic subbundle $\mathcal{L}$ of $\mathcal{K}$ over $C_R$ that restricts to the maximal destabilizing parabolic subbundle at both the special and the generic fiber.
\end{enumerate}
We can consider the corresponding parabolic suboject $\mathcal{W}^1 \hookrightarrow \mathcal{W}$ containing $\mathcal{U}$ such that $\mathcal{W}^1 \, / \, \mathcal{U} = \mathcal{M}$. Similarly we can lift $\mathcal{L}$ to a subbundle $\mathcal{U}^1$ of $\mathcal{W}^1$ containing $\mathcal{U}$. By construction the two step filtration $\mathcal{U} \subset \mathcal{U}^1 \subset \mathcal{W}^1$ restricts to the first two elements of the Harder-Narasimhan filtration at both the special and the generic fiber. Now we can repeat the construction with $\left(\mathcal{W}^1, \, \mathcal{U}^1\right)$ in place of $\left( \mathcal{W}, \, \mathcal{U} \right)$.

Since the length of the Harder-Narasimhan filtration of $\mathcal{P}$ is finite, this process must eventually stop. We end up with a parabolic vector bundle $\mathcal{W}$ over $C_R$ such that $\mathcal{W}|_{C_{\eta}} \cong \mathcal{P}$, and $\mathcal{W}$ admits a filtration by parabolic subbundles $0 \subset \mathcal{U} \subset \mathcal{U}_1 \subset \cdots \subset \mathcal{W}^l$ that restricts to the Harder-Narasimhan filtration at both the special and the generic fiber. By local constancy of slope, this implies that $HN(\mathcal{W}^l|_{C_s}) = HN(\mathcal{P}) = P$.
\end{proof}

\begin{thm} \label{thm: completeness of strata}
Let $P$ be a parabolic HN datum. The moduli stack $\text{Bun}_{\mathcal{V}}^{\leq P}$ is complete over $k$.
\end{thm}
\begin{proof}
Let $R$ be a complete discrete valuation ring over $k$. Let $\mathcal{P}: \eta \rightarrow \text{Bun}_{\mathcal{V}}^{\leq P}$ be a parabolic vector bundle over $C_{\eta}$ such that $HN(\mathcal{P}) \leq P$. By Proposition \ref{prop: lifting criterion for strata}, there exists a parabolic vector bundle $\mathcal{W}$ of type $\mathcal{V}$ over $C_R$ such that $HN(\mathcal{W}|_{C_s}) = HN(\mathcal{P}) \leq P$ and $\mathcal{W}|_{C_{\eta}} \cong \mathcal{P}$. This means that $\mathcal{W}$ represents a morphism $\mathcal{W}: \text{Spec} \, R \rightarrow \text{Bun}_{\mathcal{V}}^{\leq P}$ lifting $\mathcal{P}: \eta \rightarrow \text{Bun}_{\mathcal{V}}^{\leq P}$.
\end{proof}
\end{section}

\begin{section}{Harder-Narasimhan filtrations in families} \label{section: harder-narasimhan filtrations in families}
In this section we use the filtration schemes $Fil^{\alpha}_{\mathcal{V}}$ defined in Section \ref{section: parabolic Quot schemes} in order to give the structure of an algebraic stack to the locus $\text{Bun}^{=P}_{\mathcal{V}}$ for each HN datum $P$. Since we have proved that conjecture holds for parabolic vector bundles in arbitrary characteristic (Proposition \ref{prop: rigidity of HN filtration}), we have established the result in its strongest form that $\text{Bun}_{\mathcal{V}}^{=P} \hookrightarrow \text{Bun}_{\mathcal{V}}^{\leq P}$ is a closed immersion for every field $k$. See Theorem \ref{thm: representability of moduli of vb with fixed P} below. This idea of schematic Harder-Narasimhan stratifications first appeared in \cite{nitsurehnsheaves} for the moduli of sheaves. Gurjar and Nitsure proved an analogous result for the moduli stack of principal $G$-bundles on a curve over a field of characteristic $0$ \cite{gurjar-nitsure}, and later for higher dimensional varieties \cite{nitsuregurjar2, gurjar2020hardernarasimhan}.

We start with a technical definition.
\begin{defn} \label{defn: filtration datum of type}
Let $T$ be a $k$-scheme. Let $\mathcal{W}$ be a parabolic vector bundle over $C \times T$. Let $\alpha$ be a filtration datum, and let $P$ be a HN datum. We write $HN_{\mathcal{W}}(\alpha) = P$ if for all points $x \in Fil^{\alpha}_{\mathcal{W}}$ we have
\begin{enumerate}[(a)]
    \item $HN(\mathcal{W}|_{C_x}) = P$
    \item $x$ represents the Harder-Narasimhan filtration of $\mathcal{W}|_{C_x}$.
\end{enumerate}
\end{defn}
\begin{lemma} \label{lemma: filtration datum of type}
Let $T$ be a connected $k$-scheme. Let $\mathcal{W}$ be a parabolic vector bundle over $C \times T$. Let $P$ be a HN datum. Fix a point $t \in T$. Suppose that $HN(\mathcal{W}|_{C_t}) = P$. Let $0 = \mathcal{W}_0 \subset \mathcal{W}_1 \subset \cdots \, \subset \mathcal{W}_{l-1} \subset \mathcal{W}_l = \mathcal{W}|_{C_t}$
denote the Harder-Narasimhan filtration of $\mathcal{W}|_{C_t}$. Let $\alpha = \psi(\mathcal{W}_j)$ denote the associated filtration datum, as in Definition \ref{defn: filtration datum associated to filtration}. Then, we have $HN_{\mathcal{W}}(\alpha) = P$.
\end{lemma}
\begin{proof}
Suppose that $\alpha = \left( \left( _j b_i^{(m)} \right)_{\substack{1 \leq j \leq l-1 \\ i \in I \\ 1 \leq m \leq N_i}}, \; \left(_j P \right)_{1 \leq j \leq l-1} \right)$. Choose $x \in Fil^{\alpha}_{\mathcal{W}}$. Let $0 = \mathcal{P}_0 \subset \mathcal{P}_1 \subset \cdots \, \subset \mathcal{P}_{l-1} \subset \mathcal{P}_l = \mathcal{W}|_{C_x}$ be the filtration of $\mathcal{W}|_{C_x}$ corresponding to $x$. We know that $\theta(\mathcal{P}_j) = \theta(\mathcal{W}_j) = (_jb_i^{(m)}, \, _jP)$ for all $1 \leq j \leq l-1$. By connectedness of $T$, we know that $\theta(\mathcal{P}_{l}) = \theta(\mathcal{W}|_{C_x}) = \theta(\mathcal{W}|_{C_t}) = \theta(\mathcal{W}_l)$ (see Remark \ref{remark: parabolic Quot invariants}). Since the rank and slope of a parabolic vector bundle can be determined from its associated parabolic Quot datum, we conclude that $\mu(\mathcal{P}_j) = \mu(\mathcal{W}_j)$ and $\text{rank} \, \mathcal{P}_j = \text{rank} \, \mathcal{W}_j$ for all $1 \leq j \leq l$. Since $\{\mathcal{W}_j\}$ is a Harder-Narasimhan filtration, we must have that $\{\mathcal{P}_j\}$ is a Harder-Narasimhan filtration. We have $HN(\mathcal{W}|_{C_x}) = HN(\mathcal{W}|_{C_t}) = P$, because both filtrations have the same ranks and slopes.
\end{proof}

\begin{prop} \label{prop: schematic HN filtrations}
Let $\mathcal{V}$ be a parabolic vector bundle of rank $n$ over $C$. Let $P=( \mu_1, \mu_2, \cdots , \mu_n)$ be a parabolic HN datum. Let $\alpha$ be filtration datum. Fix a scheme $S$ of finite type over $k$. Let $\mathcal{W} : \,  S \longrightarrow \text{Bun}_{\mathcal{V}}^{\leq P}$ be a parabolic vector bundle. Suppose that $HN_{\mathcal{W}}(\alpha) = P$. Then, the structure map $\pi: \, Fil^{\alpha}_{\mathcal{W}} \longrightarrow S$ is a closed immersion.
\end{prop}
\begin{proof}
We will use \cite[\href{https://stacks.math.columbia.edu/tag/04XV}{Tag 04XV}]{stacks-project}, which says that a morphism that is universally injective, unramified and proper is a closed immersion. By Proposition \ref{prop: representability of filtration moduli}, we know that  $\pi: \, Fil^{\alpha}_{\mathcal{W}} \longrightarrow S$ is of finite type. Let $t \in S$ such that the fiber $\pi^{*} (t) = \text{Fil}^{\alpha}_{\mathcal{W}} \times_{S} t$ is nonempty. By definition, the $\kappa(t)$-scheme $\pi^{*}(t) = \text{Fil}^{\alpha}_{\mathcal{W}|_{C_t}}$ classifies filtrations of $\mathcal{W}|_{C_t}$ with filtration datum $\alpha$. A point $x \in \pi^{*}(t)$ would give us a filtration of $\mathcal{W}|_{C_{\kappa(x)}}$ with associated HN datum $P$. By Lemma \ref{lemma: HN filtrations and scalar change}, we know that $HN(\mathcal{W}|_{C_{\kappa(x)}}) = HN(\mathcal{W}_{C_t}) \leq P$. Hence, we must have $HN(\mathcal{W}|_{C_{\kappa(x)}}) = P$. So the filtration must be the Harder-Narasimhan filtration of $\mathcal{W}|_{C_{\kappa(x)}}$. By the uniqueness of the Harder-Narasimhan filtration for any field extension (Lemma \ref{lemma: HN filtrations and scalar change} again),  $\pi^{*}(t)$ has a single point and this point is defined over $\kappa(t)$. So $\pi^{*}(t)$  must be of the form $\text{Spec}(A)$ for some local artinian $\kappa(t)$-algebra $A$ with residue field $\kappa(t)$. This shows that $\pi$ radicial. Hence $\pi$ is universally injective.

We claim that $A = \kappa(t)$. In order to see this, it suffices to show that $\Omega_{A \, / \,\kappa(t)}^1 = 0$. We have that $\Omega_{A \, / \,\kappa(t)}^1$ is dual to the tangent space $\text{Hom}_{\kappa(t)} \left( \text{Spec}(\kappa(t)[\epsilon]),  \, \pi^{*}(t) \right)$. By Proposition \ref{prop: rigidity of HN filtration}, the Harder-Narasimhan filtration has no notrivial first order deformations. This means that $\text{Hom}_{\kappa(t)} \left( \text{Spec}(\kappa(t)[\epsilon]),  \, \pi^{*}(t) \right)= 0$. Therefore $A = \kappa(t)$. We conclude that $\pi$ is unramified.

 We are left to show that $\pi$ is proper. We will use the valuative criterion for properness. The argument is similar to the proof of \ref{prop: strata closed under generalization}. We use the same notation as in that proof (look there for any piece of unexplained notation). Let $R$ be a discrete valuation $k$-algebra with generic point $\eta$ and special point $s$. Let $f: \text{Spec}\; R \rightarrow S$. The fiber product $\text{Fil}^{\alpha}_{\mathcal{W}} \times_S \text{Spec}\; R$ is by definition the $R$-scheme $\pi_R : \text{Fil}^{\alpha}_{\mathcal{W}|_{C_R}} \rightarrow \text{Spec} \; R$. To simplify notation, let us set $\mathcal{U} \vcentcolon = \mathcal{W}|_{C_R}$. Suppose that we have a section $p: \eta \rightarrow \text{Fil}^{\alpha}_{\mathcal{U}}$ of $\pi_R$ defined over $\eta$. We want to show that $p$ extends to a unique section over $\text{Spec} \; R$. Uniqueness follows from the fact that $\text{Fil}^{\alpha}_{\mathcal{U}}$ is a separated $R$-scheme (Proposition \ref{prop: representability of filtration moduli}). We are left to show existence.

By hypothesis, we know that $HN(^{\eta}\mathcal{U}) , HN(^s\mathcal{U}) \leq P$. The point $p$ is a filtration of $^{\eta}\mathcal{U}$ by parabolic subbundles $0 = \, ^{\eta}\mathcal{W}_0 \, \subset \, ^{\eta}\mathcal{W}_1 \, \subset \, \cdots \, \subset \, ^{\eta}\mathcal{W}_l = \, ^{\eta}\mathcal{U}$
with filtration datum $\alpha$. Set $k_j \vcentcolon = 
\text{rank} ^{\eta}\mathcal{W}_j$. Arguing as in the proof of Proposition \ref{prop: strata closed under generalization}, we get set of inclusions $0 = \, \mathcal{W}_0 \, \hookrightarrow \, \mathcal{W}_1 \, \hookrightarrow \, \cdots \, \hookrightarrow \, \mathcal{W}_l = \, \mathcal{U}$ of parabolic subobjects over $C_R$.
These are not parabolic subbundles a priori. As in the proof of Proposition \ref{prop: strata closed under generalization}, passing to the special fiber preserves the subobject inclusions $0 = \, ^s\mathcal{W}_0 \, \hookrightarrow \, ^s\mathcal{W}_1 \, \hookrightarrow \, \cdots \, \hookrightarrow \, ^s\mathcal{W}_l = \, ^s\mathcal{U}$.

Using Lemma \ref{lemma: saturating subsheaves}, we get a filtration of $^s\mathcal{U}$ by parabolic subbundles
\[0 = \, ^s\mathcal{W}^{sat}_0 \, \subset \, ^s\mathcal{W}_1^{sat} \, \subset \, \cdots \, \subset \, \mathcal{W}_l^{sat} = \, ^s\mathcal{U}\]
with $\mu( ^s\mathcal{W}_j) \leq \mu( ^s\mathcal{W}^{sat}_j)$ for all $1 \leq j \leq l$. By local constancy of degree we have $\mu(^s\mathcal{W}_j) = \mu(^{\eta}\mathcal{W}_j) $. Since $HN(\alpha) = P$, we have $\mu(^s\mathcal{W}_j) = \mu_{k_j}$. So $\mu_{k_j} \leq \mu( ^s\mathcal{W}^{sat}_j)$. We know that $HN(^s\mathcal{U}) \leq P$, hence we must have $\mu_{k_j} = \mu( ^s\mathcal{W}^{sat}_j)$. By Remark \ref{remark: equality degree properness proof}, this implies that $\mathcal{W}_j$ is actually a parabolic subbundle of $\mathcal{U}$ over $C_R$. The filtration $0 = \, \mathcal{W}_0 \, \subset \, \mathcal{W}_1 \, \subset \, \cdots \, \subset \, \mathcal{W}_l = \, \mathcal{U}$ gives us a section of $\pi_R$ over $\text{Spec}\; R$, as desired.
\end{proof}

\begin{defn} \label{defn: relative HN filtration}
Let $S$ be a $k$-scheme and $\mathcal{V}$ be a parabolic vector bundle on $C\times S$. A relative Harder-Narasimhan filtration (over $S$) is a filtration $0 = \mathcal{W}_0 \subset \mathcal{W}_1 \subset \cdots \, \subset \mathcal{W}_{l-1} \subset \mathcal{W}_l = \mathcal{V}$
by parabolic subbundles $\mathcal{W}_j$ over $C \times S$ such that for all $s \in S$, the restriction $0 = \mathcal{W}_0|_{C_s} \subset \mathcal{W}_1|_{C_s} \subset \cdots \, \subset \mathcal{W}_{l-1}|_{C_s} \subset \mathcal{W}_l|_{C_s} = \mathcal{V}|_{C_s}$ to the fiber $C_s$ is the Harder-Narasimhan filtration of $\mathcal{V}|_{C_s}$.
\end{defn}

\begin{defn} \label{defn: moduli of parabolic vb with fixed P}
Let $\mathcal{V}$ be a parabolic vector bundle over $C$. Let $P$ be a HN datum. $\text{Bun}_{\mathcal{V}}^{=P}$ denotes the functor from $k$-schemes into groupoids given as follows. For a $k$-scheme $S$, $\text{Bun}_{\mathcal{V}}^{=P}(S)$ is the groupoid of parabolic bundles over $C\times S$ of type $\mathcal{V}$ along with a relative Harder-Narasimhan filtration (over $C \times S$) such that for all $s \in S$, we have $HN(\mathcal{V}|_{C_s}) = P$.
\end{defn}

We will denote by $\iota: \text{Bun}_{\mathcal{V}}^{=P} \rightarrow \text{Bun}_{\mathcal{V}}$ the morphism given by forgetting the relative Harder-Narasimhan filtration.
\begin{thm} \label{thm: representability of moduli of vb with fixed P}
Let $P = (\mu_1, \mu_2, \cdots, \mu_n)$ be a parabolic HN datum. The morphism $\iota: \text{Bun}_{\mathcal{V}}^{=P} \rightarrow \text{Bun}_{\mathcal{V}}$ exhibits $\text{Bun}_{\mathcal{V}}^{=P}$ as a locally closed substack of $\text{Bun}_{\mathcal{V}}$.
\end{thm}
\begin{proof}
By definition $\iota: \text{Bun}_{\mathcal{V}}^{=P} \rightarrow \text{Bun}_{\mathcal{V}}$ factors through the open substack $\text{Bun}_{\mathcal{V}}^{\leq P}$. We shall show that $\iota: \text{Bun}_{\mathcal{V}}^{=P} \rightarrow \text{Bun}_{\mathcal{V}}^{\leq P}$ is a closed immersion.

Let $T$ be a connected scheme of finite type over $k$. Let $\mathcal{W}: T \rightarrow \text{Bun}_{\mathcal{V}}^{\leq P}$ be a parabolic vector bundle over $C \times T$. Define the set $T^{=P} \vcentcolon = \left\{ \; t \in T \; \mid \; HN(\mathcal{W}|_{C_t}) = P \right\}$. Recall that for any $t \in T$ and any filtration $\{\mathcal{W}_{j}^{t}\}$ of $\mathcal{W}|_{C_t}$, we have an associated filtration datum $\psi(\mathcal{W}^{t}_{j})$ as in Definition \ref{defn: filtration datum associated to filtration}. In particular, we can do this with the Harder-Narasimhan filtration $HN\text{Fil} \,\mathcal{W}|_{C_t}$ of $\mathcal{W}|_{C_t}$. 

Let's define a set $\mathfrak{S} = \left\{ \; \psi\left( \,HN\text{Fil} \, \mathcal{W}|_{C_t}\, \right) \; \mid \; t \in T^{=P} \right\}$. We claim that $\mathfrak{S}$ is a finite set. Let $t \in T^{=P}$ and let $0 = \mathcal{W}_0^t \subset \mathcal{W}_1^t \subset \cdots \, \subset \mathcal{W}_{l-1}^t \subset \mathcal{W}_l^t = \mathcal{W}|_{C_t}$ be the Harder-Narasimhan filtration of $\mathcal{W}|_{C_t}$. Suppose that $_j\mathcal{W}^t = \left[ \; _j\mathcal{F}^{(0)} \,\overset{_jb_i^{(1)}}{\subset} \, _j\mathcal{F}_i^{(1)} \overset{_jb_i^{(2)}}{\subset} \cdots\; \overset{_jb_{i}^{(N_i)}}{\subset} \, _j\mathcal{F}_i^{(N_i)}= \, _j\mathcal{F}^{(0)}(x_i) \;\right]_{i \in I}$. We know that $HN(\mathcal{W}|_{C_t}) =P$. Therefore the length $l$ of the filtration and the rank $k_j$ of each $\mathcal{W}_j^t$ do not depend on $t$. Since there is a finite number of possibilities for the numbers $_jb_i^{(m)}$, it suffices to show that for each $j$ there is a finite number of possibilities for $\text{deg} \, _j\mathcal{F}^{(0)}$. We know that $\text{deg} \, _j\mathcal{W} = k_j \, \mu_{k_j}$ is fixed. By Lemma \ref{lemma: deg of parabolic vs regular vector bundles}, we have 
\[k_j \, \mu_{k_j} - \, k_j \, |I| \leq \text{deg} \, _j\mathcal{F}^{(0)} \leq \, k_j \, \mu_{k_j} \]
Hence there are finitely many possible choices for the integer $\text{deg} \, _j\mathcal{F}^{(0)}$. We conclude that the set $\mathfrak{S}$ is finite.

Consider the $T$-scheme $X \vcentcolon = \sqcup_{\alpha \in \mathfrak{S}} \, \text{Fil}_{\mathcal{V}}^{\alpha}$. By construction, $X$ represents the fiber product $\text{Bun}_{\mathcal{V}}^{= P} \times_{\text{Bun}_{\mathcal{V}}^{\leq P}} T$. It therefore suffices to show that $X \rightarrow T$ is a closed immersion. For every $\alpha \in \mathfrak{S}$, we have $HN_{\mathcal{W}}(\alpha) = P$ by Lemma \ref{lemma: filtration datum of type}. By Proposition \ref{prop: schematic HN filtrations}, $X$ is a finite disjoint union of closed subschemes of $T$. It therefore suffices to show that if $\alpha \neq \beta$ are two filtration data in $\mathfrak{S}$, then the closed subschemes $\text{Fil}_{\mathcal{V}}^{\alpha}$ and $\text{Fil}_{\mathcal{V}}^{\beta}$ are set theoretically disjoint inside $T$. But we can't have a point $t \in T$ in the intersection, since that would violate the uniqueness of the Harder-Narasimhan filtration for $\mathcal{W}|_{C_t}$.
\end{proof}

The result that the relative Harder-Narasimhan stacks are locally closed substacks does not hold for $\text{Bun}_{\mathcal{G}}$ for exceptional parahoric Bruhat-Tits group schemes $\mathcal{G}$ over $C$ when the characteristic of $k$ is small. This is because of the failure of Behrend's conjecture, even in the case when the group scheme is reductive (see \cite{heinloth-behrends-conjecture}).

We end this section by stating an explicit corollary of Theorem \ref{thm: representability of moduli of vb with fixed P}. A weaker version of this corollary in the context of coherent sheaves without parabolic structure can be found in \cite{huybrechts.lehn}[Thm. 2.3.2].
\begin{coroll}
Fix a parabolic vector bundle $\mathcal{V}$ over $C$. Let $S$ be a $k$-scheme and let $\mathcal{W}$ be a parabolic vector bundle of type $\mathcal{V}$ over $C\times S$. There is a surjective morphism $\bigsqcup_{P} S^P \rightarrow S$ where
\begin{enumerate}[(i)]
    \item The $P$ run over the set of all HN data of rank $n$.
    \item Each $S^P$ is a locally closed subscheme of $S$.
    \item $\mathcal{W}|_{S^P}$ admits a relative Harder-Narasimhan filtration.
    \item For all $s \in S^P$, we have $HN(\mathcal{W}|_{C_{s}})=P$.
\end{enumerate}
\end{coroll}
\begin{proof}
The parabolic vector bundle $\mathcal{W}$ is the same a as morphism $S \rightarrow \text{Bun}_{\mathcal{V}}$. Theorem \ref{thm: representability of moduli of vb with fixed P} implies that there is a surjective map of stacks $\bigsqcup_{P} \text{Bun}_{\mathcal{V}}^{=P} \, \rightarrow \text{Bun}_{\mathcal{V}}$, where the $P$ run over the set of HN data of rank $\text{rank} \, \mathcal{V}$. Moreover, each $\text{Bun}_{\mathcal{V}}^{=P}$ is a locally closed substack of $\text{Bun}_{\mathcal{V}}$.

Define $S^P$ to be the locally closed subscheme of $S$ given by $S^P \vcentcolon = S \times_{\text{Bun}_{\mathcal{V}}} \text{Bun}_{\mathcal{V}}^{=P}$. We have a fiber product diagram
\[\xymatrix{
 \bigsqcup_{P} S^P  \ar[r] \ar[d] & \bigsqcup_{P} \text{Bun}_{\mathcal{V}}^{=P} \ar[d] \\
 S \ar[r] & \text{Bun}_{\mathcal{V}} } \] 
Since the map $\bigsqcup_{P} \text{Bun}_{\mathcal{V}}^{=P} \, \rightarrow \text{Bun}_{\mathcal{V}}$ is surjective, we conclude that the base-change $\bigsqcup_{P} S^P \rightarrow S$ is surjective as well. Properties $(iii)$ and $(iv)$ follow directly from the definition of $\text{Bun}_{\mathcal{V}}^{=P}$ above.
\end{proof}
\end{section}

\appendix
\begin{section}{Descent for finite covers} \label{appendix: descent}
For the convenience of the reader, we give an explicit description of descent data for certain finite flat covers as done in \cite{grothendieck-descente}[B.3]. All sheaves considered are quasicoherent.

Let $\pi: Y \rightarrow X$ be a faithfully-flat finite morphims of schemes such that $\pi_*(\mathcal{O}_Y)$ is a locally free $\mathcal{O}_X$-module. We want to explicitly describe descent data for quasicoherent sheaves on $Y$ relative to $X$. Since $\pi$ is affine, the functor $\pi_*$ sets up an equivalence between $\text{QCoh}(Y)$ and the category of quasicoherent sheaves on $X$ with a $\pi_*(\mathcal{O}_Y)$-module structure. Consider the fiber product:
\[\xymatrix{
 Y \times_X Y \ar[r]^{\; \; \; \; \; \;pr_1} \ar[d]^{pr_2} & Y \ar[d]^{\pi}\\
Y \ar[r]^{\pi} & X } \] 
By the same reasoning, a quasicoherent sheaf on $Y \times_X Y$ is the same as a quasicoherent $\pi_*(\mathcal{O}_Y) \underset{\mathcal{O}_X}{\otimes} \pi_*(\mathcal{O}_Y)$ -module.

Given a $\pi_*(\mathcal{O}_Y)$-module $\mathcal{F}$, we can obtain two different $\pi_*(\mathcal{O}_Y) \underset{\mathcal{O}_X}{\otimes} \pi_*(\mathcal{O}_Y)$ -modules via pullback:
\[ pr_1^{*}(\mathcal{F}) = \mathcal{F} \underset{\mathcal{O}_X}{\otimes} \pi_*(\mathcal{O}_Y) \]
\[pr^{*}_2(\mathcal{F}) = \pi_*(\mathcal{O}_Y) \underset{\mathcal{O}_X}{\otimes} \mathcal{F}\]

A descent datum for $\mathcal{F}$ would correspond to an homorphism of $\pi_*(\mathcal{O}_Y) \underset{\mathcal{O}_X}{\otimes} \pi_*(\mathcal{O}_Y)$ -modules $\phi: pr_2^{*}(\mathcal{F}) \rightarrow pr_1^{*}(\mathcal{F})$ satisfying a cocycle condition.

By assumption $\pi_*(\mathcal{O}_Y)$ is locally free as a $\mathcal{O}_X$-module, so it is $\mathcal{O}_X$-dualizable. Using this fact plus the tensor-Hom adjunction, we get a chain of isomorphisms:
\begin{equation*}
\begin{split}
    \text{Hom}_{\mathcal{O}_X} \left(pr_2^{*}(\mathcal{F}), \, pr_1^{*}(\mathcal{F})\right)  & = \, \text{Hom}_{\mathcal{O}_X} \left(\pi_*(\mathcal{O}_Y) \underset{\mathcal{O}_X}{\otimes} \mathcal{F}, \; \; \mathcal{F} \underset{\mathcal{O}_X}{\otimes} \pi_*(\mathcal{O}_Y)\right) \\
    & \cong \,\text{Hom}_{\mathcal{O}_X} \left(\pi_*(\mathcal{O}_Y) \underset{\mathcal{O}_X}{\otimes} \pi_*(\mathcal{O}_Y)^{\vee} \underset{\mathcal{O}_X}{\otimes} \mathcal{F}, \; \; \; \mathcal{F} \underset{\mathcal{O}_X}{\otimes} \pi_*(\mathcal{O}_Y)\right)\\
    & \cong \, \text{Hom}_{\mathcal{O}_X} \left(\,\text{End}_{\mathcal{O}_X} \,\pi_*(\mathcal{O}_Y),  \, \; \text{End}_{\mathcal{O}_X} \, \mathcal{F} \,\right)
    \end{split}
\end{equation*}
What is the image of the set of $\pi_*(\mathcal{O}_Y) \underset{\mathcal{O}_X}{\otimes} \pi_*(\mathcal{O}_Y)$- homomorphisms? Tracing back the chain isomorphisms above, one can check that this is $\text{Hom}_{\pi_*(\mathcal{O}_Y) \underset{\mathcal{O}_X}{\otimes} \pi_*(\mathcal{O}_Y)} \left(\,\text{End}_{\mathcal{O}_X} \,\pi_*(\mathcal{O}_Y),  \, \; \text{End}_{\mathcal{O}_X} \, \mathcal{F} \,\right)$. Here $\pi_*(\mathcal{O}_Y) \underset{\mathcal{O}_X}{\otimes} \pi_*(\mathcal{O}_Y)$ acts on each $\text{End}$ group by $(s \otimes t) \cdot f(x) \vcentcolon = tf(sx)$ for any sections $s, t$ and $x$ of the corresponding sheaves over a given open set.

We are left to describe the cocycle condition for a given $\pi_*(\mathcal{O}_Y) \underset{\mathcal{O}_X}{\otimes} \pi_*(\mathcal{O}_Y)\,$-homomorphims $\phi: pr_2^{*}(\mathcal{F}) \rightarrow pr_1^{*}(\mathcal{F})$. Recall that the cocycle condition amounts to the commutativity of the following diagram:
\[ \xymatrix{
\pi_*(\mathcal{O}_Y) \underset{\mathcal{O}_X}{\otimes} \pi_*(\mathcal{O}_Y) \underset{\mathcal{O}_X}{\otimes} \mathcal{F} \ar[r]^{id \otimes \phi} \ar[d]^{id \otimes \text{swap}} & \pi_*(\mathcal{O}_Y) \underset{\mathcal{O}_X}{\otimes} \mathcal{F} \underset{\mathcal{O}_X}{\otimes} \pi_*(\mathcal{O}_Y) \ar[d]^{\phi \otimes id}\\
\pi_*(\mathcal{O}_Y) \underset{\mathcal{O}_X}{\otimes} \mathcal{F} \underset{\mathcal{O}_X}{\otimes} \pi_*(\mathcal{O}_Y) \ar[r]^{\phi \otimes id} & \mathcal{F} \underset{\mathcal{O}_X}{\otimes} \pi_*(\mathcal{O}_Y) \underset{\mathcal{O}_X}{\otimes} \pi_*(\mathcal{O}_Y) }  \]
Tracing back the construction of the isomorphim 
\[\psi: \text{Hom}_{\pi_*(\mathcal{O}_Y) \underset{\mathcal{O}_X}{\otimes} \pi_*(\mathcal{O}_Y)} \left(pr_2^{*}(\mathcal{F}), \, pr_1^{*}(\mathcal{F})\right) \, \xrightarrow{\sim} \, \text{Hom}_{\pi_*(\mathcal{O}_Y) \underset{\mathcal{O}_X}{\otimes} \pi_*(\mathcal{O}_Y)} \left(\,\text{End}_{\mathcal{O}_X} \,\pi_*(\mathcal{O}_Y),  \, \; \text{End}_{\mathcal{O}_X} \, \mathcal{F} \,\right)\]
we see that the cocycle condition corresponds to $\psi(\phi)$ being an homomorphism of unital algebras.

Let us give two examples when there is a nice explicit presentation of the unital algebra $\text{End}_{\mathcal{O}_X} \,\pi_*(\mathcal{O}_Y)$.
\begin{example}[Galois descent]
Suppose that $\pi: Y \rightarrow X$ is a Galois cover with Galois group $G$. By definition, $G$ is a group of $\mathcal{O}_X$-linear automorphisms of $\pi_*(\mathcal{O}_Y)$. Then it can be checked that we have
\[  \text{End}\left(\,\pi_*(\mathcal{O}_Y)\, \right) = \bigoplus_{g \in G} \left(1 \otimes \pi_*(\mathcal{O}_Y) \right) \,g \]
Since everything is $\mathcal{O}_X$ coherent, it suffices to check this equality on fibers, This reduces the claim to Galois theory for separable field extensions.

By definition of the $\pi_*(\mathcal{O}_Y) \underset{\mathcal{O}_X}{\otimes} \pi_*(\mathcal{O}_Y)$-module structure on $\text{End}_{\mathcal{O}_X} \, \pi_*(\mathcal{O}_Y)$, we have $(s \otimes 1) \cdot g = (1 \otimes g(s)) \cdot g$ for any $g \in G$ and any section $s$ of $\pi_*(\mathcal{O}_Y)$ on an open. So in this case descent data correspond to homomorphisms of groups $\psi: G \rightarrow \text{Aut}_{\mathcal{O}_X}\, \mathcal{F}$ such that $\psi(g)$ is $g$-semilinear for all $g \in G$.
\end{example}

\begin{example}[Inseparable descent/ Frobenius descent]
Let us suppose that $X$ is a scheme over a finite field $\kappa$ of characteristic $p$ (in other words, $p \, \mathcal{O}_X = 0$). Let $\pi: Y \rightarrow X$ as above. Suppose that $\pi_*(\mathcal{O}_Y)^p \subset \mathcal{O}_X$, so the cover is of "exponent 1". Furthermore, suppose that $\pi_*(\mathcal{O}_Y)$ locally admits a $p$-basis over $\mathcal{O}_X$ (for Noetherian schemes this is known to be equivalent to $\Omega^1_{Y\, / \,X}$ being locally free, see \cite{tyc-p-basis}).

Let $\mathcal{D}$ be the $\pi_*(\mathcal{O}_Y)$-module of $\mathcal{O}_X$-linear derivations of $\pi_*(\mathcal{O}_Y)$. Here the $\pi_*(\mathcal{O}_Y)$-module structure comes from the action of the second coordinate  in $\pi_*(\mathcal{O}_Y) \underset{\mathcal{O}_X}{\otimes} \pi_*(\mathcal{O}_Y)$ when we view $\mathcal{D}$ as a subset of $\text{End}_{\mathcal{O}_X}\, \pi_*(\mathcal{O}_Y)$. By assumption, we know that $\mathcal{D} = (\Omega^1_{Y \, / \, X})^{\vee}$ is locally free. Observe that $\mathcal{D}$ has the structure of a restricted p-Lie algebra over $\mathcal{O}_X$. We will denote by $[ -, - ]$ and $(-)^{(p)}$ the p-Lie algebra operations in $\mathcal{D}$.

We claim that $\text{End}_{\mathcal{O}_X}\, \pi_*(\mathcal{O}_Y)$ is generated as a $\pi_*(\mathcal{O}_Y) \underset{\mathcal{O}_X}{\otimes} \pi_*(\mathcal{O}_Y)$-algebra by $\mathcal{D}$ subject to the relations
\[ (1 \otimes s) \cdot \delta \, = \, (1 \otimes \delta(s)) \cdot \, id  \, + \, (s \otimes 1) \cdot \delta \]
\[ \delta \gamma - \gamma \delta \, = \, [\delta, \, \gamma] \]
\[ \delta^p \, = \, \delta^{(p)}\]
Here $\delta$ and $\gamma$ are sections of $\mathcal{D}$ and $s$ is a section of $\pi_*(\mathcal{O}_Y)$. The claim is a local statement. Since everything is $\mathcal{O}_X$ coherent, we can check it on fibers. We are reduced to the case of a finite algebra of exponent 1 over a field. This can be done by hand by using the result for field extensions (\cite{jacobson-algebraII} Theorem 8.45) and a computation for the ring of dual numbers.

Let $\mathcal{F}$ be a quasicoherent sheaf on $Y$. The discussion above shows that a descent datum $\phi: \text{End}_{\mathcal{O}_X} \, \pi_*(\mathcal{O}_Y) \, \rightarrow \, \text{End}_{\mathcal{O}_X} \, \mathcal{F}$ is equivalent to a map $\nabla : \, \mathcal{D} \, \rightarrow \, \text{End}_{\mathcal{O}_X}\, \mathcal{F}$ of restricted p-Lie algebras over $\mathcal{O}_X$ satisfying the Leibniz condition
\[ (s \otimes 1) \cdot \nabla(\delta) = (1 \otimes \delta(s)) \cdot id \, + \, (1 \otimes s) \cdot \nabla(\delta)  \]
for all sections $\delta$ of $\mathcal{D}$ and sections $s$ of $\pi_*(\mathcal{O}_Y)$. This is by definition a $X$-relative connection on $\mathcal{F}$ with vanishing $p$-curvature.

We will use this version of descent in the special case of inseparable field extensions of exponent 1. This goes back to the work of Jacobson (see \cite{jacobson-algebraII} Theorem 8.45). Another example is the case when $\pi$ is a Frobenius morphism. This is worked out in \cite{katz-nilpotent} Theorem 5.1.
\end{example}
Finally, let us describe descent for morphims of quasicoherent sheaves in this language. Let quasicoherent $\pi_*(\mathcal{O}_Y)$-modules $\mathcal{F}_1$ and $\mathcal{F}_2$, along with descent data 
\[\phi_i \in \text{Hom}_{\pi_*(\mathcal{O}_Y) \underset{\mathcal{O}_X}{\otimes} \pi_*(\mathcal{O}_Y)} \left(\,\text{End}_{\mathcal{O}_X} \,\pi_*(\mathcal{O}_Y),  \, \; \text{End}_{\mathcal{O}_X} \, \mathcal{F}_i \,\right)\]
for $i=1,2$. A homomorphims of $\pi_*(\mathcal{O}_Y)$-modules $f: \mathcal{F}_2 \rightarrow \mathcal{F}_2$ descends if and only if we have $f^{*}(\phi_2) = f_*(\phi_1)$, where
\begin{gather*}
     f^{*}: \, \text{Hom}_{\pi_*(\mathcal{O}_Y) \underset{\mathcal{O}_X}{\otimes} \pi_*(\mathcal{O}_Y)} \left(\,\text{End}_{\mathcal{O}_X} \,\pi_*(\mathcal{O}_Y),  \, \; \text{End}_{\mathcal{O}_X} \, \mathcal{F}_2 \,\right) \, \longrightarrow \, \text{Hom}_{\pi_*(\mathcal{O}_Y) \underset{\mathcal{O}_X}{\otimes} \pi_*(\mathcal{O}_Y)} \left(\,\text{End}_{\mathcal{O}_X} \,\pi_*(\mathcal{O}_Y),  \, \; \text{Hom}_{\mathcal{O}_X}( \mathcal{F}_1, \, \mathcal{F}_2) \,\right)
\end{gather*}
\begin{gather*} f_*: \, \text{Hom}_{\pi_*(\mathcal{O}_Y) \underset{\mathcal{O}_X}{\otimes} \pi_*(\mathcal{O}_Y)} \left(\,\text{End}_{\mathcal{O}_X} \,\pi_*(\mathcal{O}_Y),  \, \; \text{End}_{\mathcal{O}_X} \, \mathcal{F}_1 \,\right)  \, \longrightarrow \, \text{Hom}_{\pi_*(\mathcal{O}_Y) \underset{\mathcal{O}_X}{\otimes} \pi_*(\mathcal{O}_Y)} \left(\,\text{End}_{\mathcal{O}_X} \,\pi_*(\mathcal{O}_Y),  \, \; \text{Hom}_{\mathcal{O}_X}( \mathcal{F_1}, \, \mathcal{F}_2) \,\right)
\end{gather*}
 are given by precomposing (resp. postcomposing) with $f$. Let us give a few examples/consequences of this.
 \begin{example}[Subsheaves]
 Let $\mathcal{E}$ be a $\pi_*(\mathcal{O}_Y)$-module along with a descent datum \[\phi \in \text{Hom}_{\pi_*(\mathcal{O}_Y) \underset{\mathcal{O}_X}{\otimes} \pi_*(\mathcal{O}_Y)} \left(\,\text{End}_{\mathcal{O}_X} \,\pi_*(\mathcal{O}_Y),  \, \; \text{End}_{\mathcal{O}_X} \, \mathcal{E} \,\right)\]
 A subsheaf $\mathcal{F} \subset \mathcal{E}$ descends if and only if for all sections $s \in \text{End}_{\mathcal{O}_X} \,\pi_*(\mathcal{O}_Y)(U)$ over an open $U \subset X$, we have that $\mathcal{F}|_{U}$ is stable under the corresponding homomorphism $\phi(s)$.
 
 If $\text{End}_{\mathcal{O}_X}$ is generated by global sections, it suffices to check global sections. This is the case for descent along field extensions (or base-changes of field extensions).
 \end{example}
 
 \begin{example}[Diagrams of sheaves]
 Let $I$ be a small category. A diagram of sheaves on $Y$ of shape $I$ is a covariant functor $F: I \rightarrow \text{QCoh}(Y)$. A morphism of diagrams of sheaves of shape $I$ is a natural transformation. Descent data for a diagram of sheaves $F$ correspond to homomorphisms of unital algebras
 \[ \phi \in \text{Hom}_{\pi_*(\mathcal{O}_Y) \underset{\mathcal{O}_X}{\otimes} \pi_*(\mathcal{O}_Y)} \left(\,\text{End}_{\mathcal{O}_X} \,\pi_*(\mathcal{O}_Y),  \, \; \text{End}_{\mathcal{O}_X} F \,\right) \]
 Here $\text{End}_{\mathcal{O}_X} F$ is the algebra of endomorphisms of $F$ viewed as a diagram in $\text{QCoh}(X)$ via $\pi_*$.
 \end{example}
 
 \begin{example}[Parabolic vector bundles]
 Parabolic vector bundles are a particular case of diagrams of sheaves. Given a parabolic vector bundle $\mathcal{W}$ over $C\times Y$, descent data correspond to homomorphisms of unital algebras from $\text{End}_{\mathcal{O}_{C \times X}} \,(id \times\pi)_*(\mathcal{O}_{C\times Y})$ into the algebra of $\mathcal{O}_{C \times X}$-linear endomorphisms of $\mathcal{W}$.
 \end{example}
\end{section}
\begin{section}{Comparison with previous work} \label{appendix: comparison}
In this appendix we compare the stratification described in this article with the $\Theta$-stratifications defined in \cite{heinloth-hilbertmumford} and \cite{alper-existence}[\S 8]. For this appendix we will freely use the notation from these papers.

For simplicity of notation, we will stick with the case when there is a single point of degeneration $x \in C(k)$. Our arguments can be modified easily to deal with the case when there are several points of degeneration. Let's denote by $q : x \hookrightarrow C$ the closed immerion of $x$ into $C$. Let $\mathcal{V}$ be a parabolic vector bundle on $C$ with parabolic structure at $x$. Suppose that $\mathcal{V} = \left[ \; \mathcal{E}^{(0)} \,\overset{a^{(1)}}{\subset} \, \mathcal{E}^{(1)} \overset{a^{(2)}}{\subset} \cdots\; \overset{a^{(N)}}{\subset} \, \mathcal{E}^{(N)}= \mathcal{E}^{(0)}(x) \;\right]$. Fix a collection of parabolic weights $0 < \lambda^{(1)} < \lambda^{(2)} < \cdots < \lambda^{(N)}< 1$. This collection defines a Harder-Narasimhan stratification on $\text{Bun}_{\mathcal{V}}$, as in Definition \ref{defn: HN strata}. Our goal in this appendix is to explain how, in the case when $\overline{\lambda}$ is ``admissible", one can view this stratification as the $\Theta$-stratification coming from a $\mathbb{R}$-line bundle on the stack $\text{Bun}_{\mathcal{V}}$, as in \cite{alper-existence}[\S 8].

Set $n = \text{rank} \, \mathcal{V}$. By Definition \ref{defn: stack of parabolic vector bundles}, $\text{Bun}_{\mathcal{V}}$ only depends on the type of the parabolic bundle $\mathcal{V}$. Hence we can, without loss of generality, assume that $\mathcal{E}^{(0)} = \mathcal{O}_C^{\oplus n}$ and
\[ \mathcal{E}^{(i)} = \mathcal{O}_C(x)^{\oplus \sum_{j=1}^i a^{(j)}} \, \oplus \, \mathcal{O}_C^{\oplus\left(n -  \sum_{j=1}^i a^{(j)}\right)} \]
By Proposition \ref{prop: parabolic vb as parahoric torsors}, we have $\text{Bun}_{\mathcal{V}} \simeq \text{Bun}_{\text{GL}(\mathcal{V})}$. Set $\mathcal{G} \vcentcolon = \text{GL}(\mathcal{V})$. Let $\text{Lie}(\mathcal{G})$ denote the Lie algebra of $\mathcal{G}$. Since $\mathcal{G}$ is smooth over $C$, this Lie algebra is a vector bundle on $C$. Following \cite{heinloth-hilbertmumford}, we define $\mathcal{L}_{det}$ to be the line bundle on the stack $\text{Bun}_{\mathcal{G}}$ defined as follows. For any $k$-scheme $T$ and every map $f: T \rightarrow \text{Bun}_{\mathcal{G}}$ represented by a $\mathcal{G}_T$-torsor $\mathcal{P}$ on $C \times T$, we let
\[ f^*(\mathcal{L}_{det}) = \text{det} \left( R(\pi_{T})_{*} \, \mathcal{P} \times^{\mathcal{G}} \text{Lie}(\mathcal{G}) \right)^{\vee} \]
Here $\text{det}$ denotes the determinant in the $K$-theoretic sense. This makes sense because the derived pushforward is a perfect complex \cite[\href{https://stacks.math.columbia.edu/tag/0A1H}{Tag 0A1H}]{stacks-project}.

As in \cite{halpernleistner2014structure}, we let $\Theta = \left[ \, \mathbb{A}^1_{k} / \, \mathbb{G}_m \, \right]$. We write $0$ for the origin in $\mathbb{A}^1_{k}$. We can view $0$ as a point of $\Theta$.  Let $\varphi: \Theta \rightarrow \text{Bun}_{\mathcal{G}}$ be a filtration of a parabolic vector bundle $\mathcal{W} \in \text{Bun}_{\mathcal{V}}$. We would like to compute the weight of the $0$-fiber of $\mathcal{L}_{det}$. Say
\[\mathcal{W} = \left[ \;\mathcal{F}^{(0)} \,\overset{a^{(1)}}{\subset} \, \mathcal{F}^{(1)} \overset{a^{(2)}}{\subset} \cdots\; \overset{a^{(N)}}{\subset} \,\mathcal{F}^{(N)}= \, \mathcal{F}^{(0)}(x) \; \right] \]
Via the Rees construction, a filtration of $\mathcal{W}$ corresponds to a decreasing $\mathbb{Z}$-filtration of $\mathcal{F}^{(0)}$ by subbundles $\left(\mathcal{F}^{(0)}_m\right)_{m \in \mathbb{Z}}$. More precisely, for all $m \in \mathbb{Z}$ we have that $\mathcal{F}^{(0)}_m$ is a subbundle of $\mathcal{F}^{(0)}$ and $\mathcal{F}_m^{(0)} \supset \mathcal{F}^{(0)}_{m+1}$. Moreover we require that $\mathcal{F}^{(0)}_m = 0$ for $m >>0$ and $\mathcal{F}^{(0)}_m = \mathcal{F}^{(0)}$ for all $m <<0$. Alternatively, we can view this as a filtration of $\mathcal{W}$ by parabolic subbundles $(\mathcal{W}_m)_{m \in \mathbb{Z}}$ by setting
$\mathcal{F}^{(i)}_m \vcentcolon = \mathcal{F}_m^{(0)} \cap \mathcal{F}^{(i)}$ and defining
\[\mathcal{W}_m \vcentcolon = \left[ \;\mathcal{F}_m^{(0)} \,\overset{a_m^{(1)}}{\subset} \, \mathcal{F}_m^{(1)} \overset{a_m^{(2)}}{\subset} \cdots\; \overset{a_m^{(N)}}{\subset} \,\mathcal{F}_m^{(N)}= \, \mathcal{F}_m^{(0)}(x) \; \right] \]
In order to compute $wt\left(\varphi^*(\mathcal{L}_{det})|_{0} \right)$, we will set up a short exact sequence for $\text{Lie}(\mathcal{G})$. There is a monomorphism of group schemes $\mathcal{G} \hookrightarrow \text{Aut}(\mathcal{E}^{(0)}(x))$, given by forgetting the automorphisms of $\mathcal{E}^{(i)}$ for $i <N$. By \cite{mayeux2020neron}[Lemma 3.7], we have a short exact sequence of pointed \`etale sheaves
\[  1 \, \rightarrow \,\mathcal{G} \, \rightarrow \, \text{Aut}(\mathcal{E}^{(0)}(x) \, \rightarrow \text{Aut}(\mathcal{E}^{(0)}(x)|_{x}) \, / \, P \, \rightarrow \, 1   \]
where $P$ is the parabolic subgroup of $\text{Aut}(\mathcal{E}^{(0)}(x)|_{x})$ that preserves the flag
\[0 \subset \mathcal{E}^{(1)} / \, \mathcal{E}^{(0)} \, \subset \, \mathcal{E}^{(2)}/ \, \mathcal{E}^{(0)} \, \subset \cdots \, \subset \mathcal{E}^{(N)}/ \, \mathcal{E}^{(0)} \cong \mathcal{E}^{(0)}(x)|_{x} \]
We remark that in this case the short exact sequence can be  proven by hand. 

Under a choice of isomorphism $\mathcal{E}^{(0)}(x)|_{x}  \cong \mathcal{O}^{\oplus n}_C(x)|_{x} \cong k^n$, the group $P$ is the standard parabolic subgroup of $\text{GL}_n$ corresponding to the partition $(a^{(1)}, a^{(2)}, \cdots , a^{(N)})$. In other words, $P$ consists of block upper-triangular matrices in $\text{GL}_n$ with blocks of sizes given by the tuple $(a^{(1)}, a^{(2)}, \cdots , a^{(N)})$. 

It is known that $\text{GL}_n /P$ admits sections Zariski locally. Let $\mathfrak{gl}_n$ (resp. $\mathfrak{p}$) denote the Lie algebra of $\text{GL}_n$ (resp. $P$). By looking at points over the dual numbers, one can see that the short exact sequence of \'etale sheaves above induces a short exact sequence of coherent sheaves on $C$
\[0 \, \rightarrow \, \text{Lie}(\mathcal{G}) \, \rightarrow \, \text{Lie}(\text{Aut}(\mathcal{E}^{(0)}(x)) \, \rightarrow \, q_*(\mathfrak{gl}_n / \mathfrak{p}) \, \rightarrow \, 0 \]
Note that $\text{Lie}(\text{Aut}(\mathcal{E}^{(0)}(x)) = \mathcal{E}\mathit{nd}(\mathcal{E}^{(0)}(x))$. The group scheme $\mathcal{G}$ acts on both $\mathcal{E}\mathit{nd}(\mathcal{E}^{(0)}(x))$ and $q_*(\mathfrak{gl}_n / \mathfrak{p})$ via the adjoint representation. The action on $q_*(\mathfrak{gl}_n / \mathfrak{p})$ factors through the composition $\mathcal{G} \rightarrow q_*(\mathcal{G}|_{x}) \rightarrow q_*(P)$.

Let's compute the weight of $\varphi^*(\mathcal{L}_{det})|_{0}$. The composition $0 \rightarrow \Theta \xrightarrow{\varphi} \text{Bun}_{\mathcal{G}}$ is represented by the $\mathbb{Z}$-graded parabolic vector bundle $\bigoplus_{m \in \mathbb{Z}} \mathcal{W}_m / \mathcal{W}_{m+1}$. Here $\mathbb{G}_m$ acts with weight $m$ on $\mathcal{W}_m / \mathcal{W}_{m+1}$. Let $\mathcal{P}_{univ}$ denote the universal $\mathcal{G}$-torsor on $C \times \text{Bun}_{\mathcal{G}}$. We can use flat base-change to write
\[ wt\left(\varphi^*(\mathcal{L}_{det})|_{0}\right) = - wt\left( \text{det} \, R\pi_* \, \varphi_{C}^{*}(\mathcal{P}_{univ})|_{0\times C} \times^{\mathcal{G}} \text{Lie}(\mathcal{G}) \right) \]
After using the short exact sequence of vector bundles above and the additivity of the (K-theoretic) determinant, we get
\begin{gather*}
     wt\left(\varphi^*(\mathcal{L}_{det})|_{0}\right) = - wt\left( \text{det} \, R\pi_* \, \varphi_{C}^{*}(\mathcal{P}_{univ})|_{0 \times C} \times^{\mathcal{G}} \mathcal{E}\mathit{nd}(\mathcal{E}^{(0)}(x)) \right) + wt\left( \text{det} \, R\pi_* \, \varphi_{C}^{*}(\mathcal{P}_{univ})|_{0 \times C} \times^{\mathcal{G}} q_*(\mathfrak{gl}_n/\mathfrak{p}) \right)
\end{gather*}
We are left to compute each term separately. For the first term, one can unravel the definitions to obtain
\[\varphi_{C}^{*}(\mathcal{P}_{univ})|_{0 \times C} \times^{\mathcal{G}} \mathcal{E}\mathit{nd}(\mathcal{E}^{(0)}(x)) = \mathcal{E}\mathit{nd}\left( \bigoplus_{m \in \mathbb{Z}} \mathcal{F}^{(0)}_m(x) / \, \mathcal{F}^{(0)}_{m+1}(x) \right) \]
The computation in \cite{heinloth-hilbertmumford}[1.E.c] shows that
\begin{gather*}
   - wt\left(\text{det}\, R\pi_* \left( \bigoplus_{m \in \mathbb{Z}} \mathcal{F}^{(0)}_m(x) / \, \mathcal{F}^{(0)}_{m+1}(x) \right) \right) = 2 \sum_{h \in \mathbb{Z}}\left( \text{deg}(\mathcal{F}^{(0)}_h(x)) \cdot \text{rank}(\mathcal{F}^{(0)}) - \text{deg}(\mathcal{F}^{(0)}(x)) \cdot \text{rank}(\mathcal{F}_h^{(0)})\right)
\end{gather*}
This takes care of the first term above. On the other hand, we can write
\begin{gather*}
\text{det} \, R\pi_* \, \varphi_{C}^{*}(\mathcal{P}_{univ})|_{0 \times C} \times^{\mathcal{G}} q_*(\mathfrak{gl}_n/\mathfrak{p}) \, \cong \,  \text{det} \, R\pi_* q_* \left( \, \varphi_{C}^{*}(\mathcal{P}_{univ})|_{0\times x} \times^{\mathcal{G}|_{x}} P  \times^{P} \mathfrak{gl}_n/\mathfrak{p}) \right)
\end{gather*}
where we use the natural map $\mathcal{G}|_{x} \rightarrow P$ to form the associated $P$-torsor $ \varphi_{C}^{*}(\mathcal{P}_{univ})|_{0\times x} \times^{\mathcal{G}|_{x}} P$. We can further rewrite this as
\[   \text{det} \, R(\pi q)_* \left( \, \varphi_{C}^{*}(\mathcal{P}_{univ})|_{0\times x} \times^{\mathcal{G}|_{x}} P  \times^{P} \mathfrak{gl}_n/\mathfrak{p}) \right) \, \cong \, \varphi_{C}^{*}(\mathcal{P}_{univ})|_{0\times x} \times^{\mathcal{G}|_{x}} P  \times^{P} \text{det}(\mathfrak{gl}_n/\mathfrak{p})\]
The character $\text{det}(\mathfrak{gl}_n/\mathfrak{p})$ of $P$ factors through the Levi quotient $M = P/ \, U = \prod_{i=1}^N \text{GL}_{a^{(i)}}$. Therefore we can rewrite the expression above as
\[ \varphi_{C}^{*}(\mathcal{P}_{univ})|_{0\times x} \times^{\mathcal{G}|_{x}} M  \times^{M} \text{det}(\mathfrak{gl}_n/\mathfrak{p}) \]
As a $M$-representation, we have
\[ \text{det}(\mathfrak{gl}_n/\mathfrak{p}) \, \cong \, \text{det} \left( \bigoplus_{i < j} \left(St_{\text{GL}_{a^{(i)}}}\right)^{\vee} \otimes \left(St_{\text{GL}_{a^{(j)}}}\right) \right)\]
Here $St_{\text{GL}_{a^{(i)}}}$ denotes the standard $a^{(i)}$-dimensional representation of the $\text{GL}_{a^{(i)}}$-component of $M$. Using this, we get
\begin{gather*} \varphi_{C}^{*}(\mathcal{P}_{univ})|_{0\times x} \times^{\mathcal{G}|_{x}} M  \times^{M} \text{det}(\mathfrak{gl}_n/\mathfrak{p}) \, \cong \, \bigotimes_{i < j} \, \text{det} \left(\left( \bigoplus_{m \in \mathbb{Z}} \frac{\mathcal{F}^{(i)}_m / \mathcal{F}^{(i)}_{m+1}}{\mathcal{F}^{(i-1)}_m / \mathcal{F}^{(i-1)}_{m+1}} \right)^{\vee} \otimes \left( \bigoplus_{l \in \mathbb{Z}} \frac{\mathcal{F}^{(j)}_l / \mathcal{F}^{(j)}_{l+1}}{\mathcal{F}^{(j-1)}_l / \mathcal{F}^{(j-1)}_{l+1}} \right) \right) 
\end{gather*}
Here $\mathbb{G}_m$ acts on $\left( \frac{\mathcal{F}^{(i)}_m / \mathcal{F}^{(i)}_{m+1}}{\mathcal{F}^{(i-1)}_m / \mathcal{F}^{(i-1)}_{m+1}} \right)^{\vee} \otimes \left( \frac{\mathcal{F}^{(j)}_l / \mathcal{F}^{(j)}_{l+1}}{\mathcal{F}^{(j-1)}_l / \mathcal{F}^{(j-1)}_{l+1}} \right)$ with weight $l-m$. We are reduced to computing for each $i,j$ the weight
\[ wt\left( \text{det} \left(\left( \bigoplus_{m \in \mathbb{Z}} \frac{\mathcal{F}^{(i)}_m / \mathcal{F}^{(i)}_{m+1}}{\mathcal{F}^{(i-1)}_m / \mathcal{F}^{(i-1)}_{m+1}} \right)^{\vee} \otimes \left( \bigoplus_{l \in \mathbb{Z}} \frac{\mathcal{F}^{(j)}_l / \mathcal{F}^{(j)}_{l+1}}{\mathcal{F}^{(j-1)}_l / \mathcal{F}^{(j-1)}_{l+1}} \right) \right)  \right)\]
This is the weight of the $0$-fiber of a line bundle $\mathcal{L}_{\chi_{i,j}}$ in $\text{Bun}_{\mathcal{G}}$ corresponding to a character $\chi_{i,j}$ of $\mathcal{G}|_{x}$, as in \cite{heinloth-hilbertmumford}[3.B]. Indeed we can take $\chi_{i,j} = \text{det} \left( \,\left(St_{\text{GL}_{a^{(i)}}}\right)^{\vee} \otimes  \left(St_{\text{GL}_{a^{(j)}}}\right)\, \right)$, which we can view as a representation of $\mathcal{G}|_{x}$ via the natural map $\mathcal{G}|_{x} \rightarrow M$. Let's compute the corresponding weight. We have
\begin{gather*}
wt\left( \text{det} \left(\left( \bigoplus_{m \in \mathbb{Z}} \frac{\mathcal{F}^{(i)}_m / \mathcal{F}^{(i)}_{m+1}}{\mathcal{F}^{(i-1)}_m / \mathcal{F}^{(i-1)}_{m+1}} \right)^{\vee} \otimes \left( \bigoplus_{l \in \mathbb{Z}} \frac{\mathcal{F}^{(j)}_l / \mathcal{F}^{(j)}_{l+1}}{\mathcal{F}^{(j-1)}_l / \mathcal{F}^{(j-1)}_{l+1}} \right) \right)  \right) \, = \, \sum_{m, l \in \mathbb{Z}} (l-m) \cdot \text{rank} \left( \frac{\mathcal{F}^{(i)}_m / \mathcal{F}^{(i)}_{m+1}}{\mathcal{F}^{(i-1)}_m / \mathcal{F}^{(i-1)}_{m+1}} \right) \cdot \text{rank}\left( \frac{\mathcal{F}^{(j)}_l / \mathcal{F}^{(j)}_{l+1}}{\mathcal{F}^{(j-1)}_l / \mathcal{F}^{(j-1)}_{l+1}} \right)
\end{gather*}
Let's rewrite this as
\begin{gather*}
 \sum_{m < l \in \mathbb{Z}} (l-m) \cdot \text{rank} \left( \frac{\mathcal{F}^{(i)}_m / \mathcal{F}^{(i)}_{m+1}}{\mathcal{F}^{(i-1)}_m / \mathcal{F}^{(i-1)}_{m+1}} \right) \cdot \text{rank}\left( \frac{\mathcal{F}^{(j)}_l / \mathcal{F}^{(j)}_{l+1}}{\mathcal{F}^{(j-1)}_l / \mathcal{F}^{(j-1)}_{l+1}} \right) \, - \,  \sum_{m < l \in \mathbb{Z}} (l-m) \cdot \text{rank} \left( \frac{\mathcal{F}^{(i)}_l / \mathcal{F}^{(i)}_{l+1}}{\mathcal{F}^{(i-1)}_l / \mathcal{F}^{(i-1)}_{l+1}} \right) \cdot \text{rank}\left( \frac{\mathcal{F}^{(j)}_m / \mathcal{F}^{(j)}_{m+1}}{\mathcal{F}^{(j-1)}_m / \mathcal{F}^{(j-1)}_{m+1}} \right)
\end{gather*}
We can use the same trick as in \cite{heinloth-hilbertmumford}[1.E.c] to express this as a sum
\begin{gather*}
 \sum_{h \in \mathbb{Z}} \left( \sum_{m \leq h} \text{rank} \left( \frac{\mathcal{F}^{(i)}_m / \mathcal{F}^{(i)}_{m+1}}{\mathcal{F}^{(i-1)}_m / \mathcal{F}^{(i-1)}_{m+1}} \right) \right) \cdot \left(\sum_{l>h} \text{rank}\left( \frac{\mathcal{F}^{(j)}_l / \mathcal{F}^{(j)}_{l+1}}{\mathcal{F}^{(j-1)}_l / \mathcal{F}^{(j-1)}_{l+1}} \right) \right) \, - \,  \sum_{h \in \mathbb{Z}} \left(\sum_{l>h} \text{rank} \left( \frac{\mathcal{F}^{(i)}_l / \mathcal{F}^{(i)}_{l+1}}{\mathcal{F}^{(i-1)}_l / \mathcal{F}^{(i-1)}_{l+1}} \right) \right) \cdot \left(\sum_{ m \leq h} \text{rank}\left( \frac{\mathcal{F}^{(j)}_m / \mathcal{F}^{(j)}_{m+1}}{\mathcal{F}^{(j-1)}_m / \mathcal{F}^{(j-1)}_{m+1}} \right) \right)
\end{gather*}
This can be rewritten as
\begin{gather*}
 \sum_{h \in \mathbb{Z}} a^{(i)}_h \cdot \left(a^{(j)} - a^{(j)}_h \right) \, - \, \sum_{h \in \mathbb{Z}} a^{(j)}_h \cdot \left(a^{(i)} - a^{(i)}_h\right)
\end{gather*}
After some cancellation, we end up with $wt\left(\varphi^*(\mathcal{L}_{\chi_{i,j}})|_{0}\right) =  \sum_{h \in \mathbb{Z}} \left( a^{(i)}_h \cdot a^{(j)} - a^{(j)}_h \cdot a^{(i)} \right)$. So after all of our computations, we conclude that
\begin{gather*}
wt\left(\varphi^*(\mathcal{L}_{det})|_{0}\right) \, = \, 2 \sum_{h \in \mathbb{Z}}\left( \text{deg}(\mathcal{F}^{(0)}_h(x)) \cdot \text{rank}(\mathcal{F}^{(0)}) - \text{deg}(\mathcal{F}^{(0)}(x)) \cdot \text{rank}(\mathcal{F}_h^{(0)})\right) \, + \, \sum_{i<j} wt\left(\varphi^*(\mathcal{L}_{\chi_{i,j}})|_{0}\right)
\end{gather*}
Set $\chi$ to be the element in $\text{Hom}\left(\mathcal{G}|_{x}, \mathbb{G}_m\right) \otimes_{\mathbb{Z}} \mathbb{R}$ given by $\chi =  \text{det}(\mathfrak{gl}_n/\mathfrak{p})^{\vee} \, \otimes \, \bigotimes_{i,j} \chi_{i,j}^{\otimes -2 \, \lambda^{(i)}}$.
By linearizing over $\mathbb{R}$ the definition in \cite{heinloth-hilbertmumford}[3.B], we can define an element $\mathcal{L}_{\chi}$ in the $\mathbb{R}$-Picard group of the stack $\text{Bun}_{\mathcal{G}}$. Using our computations from above, we can get
\begin{gather*}
    wt\left( \varphi^*(\mathcal{L}_{det} \otimes \mathcal{L}_{\chi})|_{0}\right) \, = \, 2 \sum_{h \in \mathbb{Z}}\left( \text{deg}(\mathcal{F}^{(0)}_h(x)) \cdot \text{rank}(\mathcal{F}^{(0)}) - \text{deg}(\mathcal{F}^{(0)}(x)) \cdot \text{rank}(\mathcal{F}_h^{(0)})\right) \, - \, 2 \sum_{i,j} \lambda^{(i)} \cdot wt\left(\varphi^*(\mathcal{L}_{\chi_{i,j}})|_{0}\right) 
\end{gather*}
We can further plug in our formula for $wt\left(\varphi^*(\mathcal{L}_{\chi_{i,j}})|_{0}\right)$ to get
\begin{gather*}
    wt\left( \varphi^*(\mathcal{L}_{det} \otimes \mathcal{L}_{\chi})|_{0}\right) \, = \, 2 \sum_{h \in \mathbb{Z}}\left( \text{deg}(\mathcal{F}^{(0)}_h(x)) \cdot \text{rank}(\mathcal{F}^{(0)}) - \text{deg}(\mathcal{F}^{(0)}(x)) \cdot \text{rank}(\mathcal{F}_h^{(0)})\right) \, - \, 2 \sum_{i,j} \lambda^{(i)} \sum_{h \in \mathbb{Z}} \left( a^{(i)}_h \cdot a^{(j)} - a^{(j)}_h \cdot a^{(i)} \right)
\end{gather*}
The second term can be rewritten as
\begin{gather*}
    2 \sum_{i,j} \lambda^{(i)} \sum_{h \in \mathbb{Z}} \left( a^{(i)}_h \cdot a^{(j)} - a^{(j)}_h \cdot a^{(i)} \right) \, = \, 2 \sum_{h \in \mathbb{Z}}\left( \left(\sum_{i} a^{(i)}_h \lambda^{(i)} \right) \cdot \text{rank}(\mathcal{F}^{(0)}) - \left(\sum_{i} a^{(i)} \lambda^{(i)} \right) \cdot \text{rank}(\mathcal{F}_h^{(0)}) \right)
\end{gather*}
Plugging this back in, we end up with
\begin{gather*}
    wt\left( \varphi^*(\mathcal{L}_{det} \otimes \mathcal{L}_{\chi})|_{0}\right) \, = \, 2 \sum_{h \in \mathbb{Z}}\left( \left(\text{deg}(\mathcal{F}^{(0)}_h(x)) - \sum_{i} a^{(i)}_h \lambda^{(i)}\right) \cdot \text{rank}(\mathcal{F}^{(0)}) - \left( \text{deg}(\mathcal{F}^{(0)}(x)) - \sum_{i} a^{(i)} \lambda^{(i)} \right) \cdot \text{rank}(\mathcal{F}_h^{(0)})\right)
\end{gather*}
By our definition of parabolic degree, this is the same as
\begin{gather*}
    wt\left( \varphi^*(\mathcal{L}_{det} \otimes \mathcal{L}_{\chi})|_{0}\right) \, = \, 2 \sum_{h \in \mathbb{Z}}\left( \, \text{deg}(\mathcal{W}_h) \cdot \text{rank}(\mathcal{W}) -  \text{deg}(\mathcal{W}) \cdot \text{rank}(\mathcal{W}_h) \, \right)
\end{gather*}
Using this formula, one can see that a parabolic vector bundle is $\mathcal{L}_{det} \otimes \mathcal{L}_{\chi}$-(semi)stable in the sense of \cite{heinloth-hilbertmumford} if and only if it is slope (semi)stable as in our definition. In order to see that the HN filtrations match up, one can use the deletion lemma in \cite{halpernleistner2014structure}[Lemma 5.32], which can be verified to work the same way in our context of parabolic vector bundles.

It should be noted that the results in \cite{heinloth-hilbertmumford} \cite{alper-existence}[\S 8] are valid whenever the character $\chi$ used to define $\mathcal{L}_{\chi}$ is ``admissible" (see \cite{heinloth-hilbertmumford}[3.F] for a definition). Let's study whether the character we defined above is admissible. The diagonal torus $T$ of $\text{GL}_n$ over $k(C)$ can be viewed as a maximal split torus of the generic fiber of $\mathcal{G}$ under the identification provided by the proof of Proposition \ref{prop: groups schemes of automorphisms is smooth}. This torus extends to a split torus $\mathcal{T}$ of $\mathcal{G}$ defined over all of $C$ (take the scheme theoretic closure). The fiber $\mathcal{T}|_{x} \subset \mathcal{G}|_{x}$ maps isomorphically onto the diagonal torus $T_x$ of the Levi quotient $M = \prod_{i=1}^N \text{GL}_{a^{(i)}}$. We will use the standard basis of the character lattice of $T_x$, namely $X^{*}(T_x) = \bigoplus_{j=1}^n \mathbb{Z} e_j$, where $e_j$ is the character that projects to the $j$th component of $T_x$ (under the canonical ordered basis of $k^n$). We can view each $\chi_{i,j}$ as a $T_x$-representation via restriction. Then, we can write
\[ \chi_{i,j} = \text{det}\left(\left(St_{\text{GL}_{a^{(i)}}}\right)^{\vee} \otimes \left(St_{\text{GL}_{a^{(j)}}}\right) \right) = - \sum_{l= (\sum_{m = 1}^{i-1} a^{(m)}) +1}^{\sum_{m=1}^{i} a^{(m)}} a^{(j)} e_l \; \; + \sum_{l= (\sum_{m = 1}^{j-1} a^{(m)}) +1}^{\sum_{m=1}^{j} a^{(m)}} a^{(i)}e_l \]

We conclude that
\[ \bigotimes_{i,j} \chi_{i,j}^{\otimes -2 \, \lambda^{(i)}} = -2 \sum_{i=1}^N \left(\sum_{l= (\sum_{m = 1}^{i-1} a^{(m)}) +1}^{\sum_{m=1}^{i} a^{(m)}} e_l \right) \cdot \left( \left(\sum_{m=1}^{N} a^{(m)}\lambda^{(m)}\right) - n \lambda^{(i)}\right)\]
and
\[\text{det}(\mathfrak{gl}_n/\mathfrak{p}) \, = \, \bigotimes_{i<j} \chi_{i,j} \, = \, \sum_{i = 1}^{N} \left(\sum_{l= (\sum_{m = 1}^{i-1} a^{(m)}) +1}^{\sum_{m=1}^{i} a^{(m)}} e_l \right) \cdot \left(2\sum_{m=1}^{i-1} a^{(m)} + a^{(i)} - n\right) \]
Therefore
\begin{gather*}
    \chi = \text{det}(\mathfrak{gl}_n/\mathfrak{p})^{\vee} \, \otimes \, \bigotimes_{i,j} \chi_{i,j}^{\otimes -2 \, \lambda^{(i)}} =  \sum_{i=1}^N \left(\sum_{l= (\sum_{m = 1}^{i-1} a^{(m)}) +1}^{\sum_{m=1}^{i} a^{(m)}} e_l \right) \cdot \left(n - a^{(i)} - 2\sum_{m=1}^{i-1} a^{(m)} - 2(\sum_{m=1}^{N} a^{(n)}\lambda^{(m)}) + 2 n \lambda^{(i)}\right)
\end{gather*}
To determine admissibility, now we need to compute the inner product with each coroot of $\text{GL}_n$. We know that the coroots are of the form $\alpha_{i j} = e_i -e_j$ for some $i \neq j$ with $1 \leq i, j \leq n$. In this case the admissibility condition means that $\langle \chi, \alpha_{i,j} \rangle \leq 1$ for all $i \neq j$. Note that if the indices $i,j$ lie in the same ``block" determined by the partition $(a^{(1)}, a^{(2)}, \cdots , a^{(N)})$, then $\langle \chi, \alpha_{i,j} \rangle = 0$. Otherwise, if $i$ is in the $l$-th block and $j$ is in the $k$-th block we have
\[ \langle \chi, \alpha_{i,j} \rangle = 2n \lambda^{(l)} - 2\sum_{m=1}^{l-1} a^{(m)} - a^{(l)} - 2n \lambda^{(m)} + 2\sum_{m=1}^{k-1} a^{(m)} + a^{(k)}\]
Hence the admissibility condition reads
\[2n \lambda^{(l)} - 2\sum_{m=1}^{l-1} a^{(m)} - a^{(l)} - 2n \lambda^{(m)} + 2\sum_{m=1}^{k-1} a^{(m)} + a^{(k)} \leq 1\]
This condition can be expressed as follows. For all $1 \leq k \leq l \leq N$, we must have the pair of inequalities
\[ -\frac{1}{2n} + \frac{1}{n} \sum_{m=k}^{l-1} a^{(m)} + \frac{a^{(l)} -a^{(k)}}{2n} \, \leq \, \lambda^{(l)} - \lambda^{(m)} \, \leq \, \frac{1}{2n} + \frac{1}{n} \sum_{m=k}^{l-1} a^{(m)} + \frac{a^{(l)} -a^{(k)}}{2n} \]
In particular this implies that for all $l$ we must have
 \[\frac{a^{(l)} + a^{(l-1)}}{2n} -  \frac{1}{2n} \, \leq \, \lambda^{(l)} - \lambda^{(l-1)} \, \leq \, \frac{a^{(l)} + a^{(l-1)}}{2n} +  \frac{1}{2n} \]
From these inequalities we conclude that $ \lambda^{(l)}> \lambda^{(l-1)}$ and that the maximum difference $\lambda^{(N)} - \lambda^{(1)}$ is  strictly smaller than $1$. The notion of $\mathcal{L}_{det} \otimes \mathcal{L}_{\chi}$-(semi)stability and the admissibility of $\chi$ do not change if we translate all stability parameters $\lambda^{(l)}$ by the same constant $\lambda^{(l)} \, \mapsto \, \lambda^{(l)} +c$. In order to normalize, we can translate so that $0 < \lambda^{(1)} < \lambda^{(2)} < \cdots < \lambda^{(N)} <1$. We therefore see that, up to translation, all of the admissible tuples of parameters $\lambda^{(l)}$ in the sense of \cite{heinloth-hilbertmumford} are parabolic weights in our sense. However the admissibility conditions described above are not satisfied by all parabolic weights. In other words, the admissibility of the stability parameters from \cite{heinloth-hilbertmumford} is more restrictive that the conditions in the definition of parabolic weights (Definition \ref{defn: parabolic weights}).
\end{section}
\bibliography{parahoric_harder_narasimhan}
\bibliographystyle{alpha}

\footnotesize

  \textsc{Department of Mathematics, Cornell University,
    310 Malott Hall, Ithaca, New York 14853, USA}\par\nopagebreak
  \textit{E-mail address}, \texttt{ajf254@cornell.edu}

\end{document}